\documentclass[11pt, twoside]{article}
\usepackage{amsmath,amsthm,amssymb,color}
\usepackage{times}
\usepackage{enumerate}
\usepackage{url}
\usepackage[mathlines]{lineno}
\usepackage{cases}
\usepackage{multirow}
\usepackage{enumerate}
\usepackage{microtype}
\usepackage[american]{babel}
\usepackage{indentfirst}
\usepackage[a4paper, left=2.5cm, right=2.5cm, top=2.5cm, bottom=2.5cm]{geometry}
\usepackage[T1]{fontenc}
\usepackage{enumitem}

\author{Yuqi Ruan $^{\dag,1}$ \and Kaizhi Wang $^{\ddag,2}$  \and Jun Yan $^{\dag,3}$}
\setenumerate [1]{itemsep=0pt, partopsep=0pt, parsep=\parskip, topsep=-5pt}
\setitemize[1]{itemsep=0pt, partopsep=0pt, parsep=\parskip, topsep=-5pt}
\setdescription{itemsep=0pt, partopsep=0pt, parsep=\parskip, topsep=-5pt}

\usepackage{fancyhdr}
\usepackage{lastpage}

\def\titlerunning#1{\gdef\titrun{#1}}
\makeatletter
\def\author#1{\gdef\autrun{\def\and{\unskip, }#1}\gdef\@author{#1}}

\makeatother

\newtheorem{theorem}{Theorem}[section]
\newtheorem{lemma}{Lemma}[section]
\newtheorem{definition}{Definition}[section]
\newtheorem{proposition}{Proposition}[section]
\newtheorem{remark}{Remark}[section]
\newtheorem{corollary}{Corollary}[section]
\newtheorem{example}{Example}[section]
\numberwithin{equation}{section}

\DeclareMathOperator*{\supp}{supp}

\newcommand{\g}{\gamma}
\newcommand{\R}{\mathbf{R}}
\newcommand{\eps}{\varepsilon}
\newcommand{\T}{T^*M}
\newcommand{\va}{\varphi}
\newcommand{\s}{\mathcal{S}^-}
\newcommand{\x}{\sigma}
\begin{document}
	
	\baselineskip=16pt
	
	\titlerunning{Lyapunov stability and uniqueness problems for HJ equations}
	\title{Lyapunov stability and uniqueness problems for Hamilton-Jacobi equations without monotonicity}
	\author{Yuqi Ruan$^{\dag,1}$ \and Kaizhi Wang$^{\ddag,2}$  \and Jun Yan$^{\dag,3}$ }
	\date{June 26, 2024}
	
	\maketitle
	
	\bigskip\bigskip\bigskip

	$^\dag$ School of Mathematical Sciences, Fudan University, Shanghai 200433, China
	
	$^\ddag$ School of Mathematical Sciences, CMA-Shanghai, Shanghai Jiao Tong University,   Shanghai 200240, China
	
	$^1$ Email: 22110180038@m.fudan.edu.cn
	
	$^2$ Email: kzwang@sjtu.edu.cn
	
	$^3$ Email: yanjun@fudan.edu.cn

	\bigskip\bigskip\bigskip
	
	
	\bigskip\bigskip\bigskip
	
	
	\bigskip\bigskip\bigskip
	
	\textbf{The Data Availability Statement.} No datasets were generated or analysed during the current study
	
	\bigskip\bigskip\bigskip
	
	\textbf{The Conflict of Interest Statement.} 
	We have no conflicts of interest to disclose
	
	\newpage
	
	\bigskip
	
	\begin{abstract}
		We consider the evolutionary Hamilton-Jacobi equation
		\begin{align*}
			w_t(x,t)+H(x,Dw(x,t),w(x,t))=0, \quad(x,t)\in M\times [0,+\infty),
		\end{align*}
		where $M$ is a compact manifold, $H:T^*M\times\R\to\R$, $H=H(x,p,u)$ satisfies Tonelli conditions in $p$ and the Lipschitz condition in $u$.

		This work mainly concerns with the Lyapunov stability (including asymptotic stability, and  instability) and uniqueness of stationary viscosity solutions of the equation. A criterion for stability and a criterion for instability are given. We do not utilize auxiliary functions and thus our method is different from the classical Lyapunov's direct method.
		We also prove several uniqueness results for stationary viscosity solutions. The Hamiltonian $H$ has no concrete form and it may be non-monotonic in the argument $u$, where the situation is more complicated than the monotonic case. Several simple but nontrivial examples are provided, including the following equation on the unit circle
		\[
		w_t(x,t)+\frac{1}{2}w^2_x(x,t)-a\cdot w_x(x,t)+(\sin x+b)\cdot w(x,t)=0,\quad x\in \mathbf{S},	
		\]
		where  $a$, $b\in\R$ are parameters. We analyze the stability, and instability of the stationary solution $w=0$ when parameters vary, and show that $w=0$ is the unique stationary solution when $a=0$, $b>1$ and $a\neq0$, $b\geqslant 1$.

		The sign of the integral of $\frac{\partial H}{\partial u}$ with respect to the Mather measure of the contact Hamiltonian system generated by $H$ plays an essential role in the proofs of aforementioned results. For this reason, we first develop the Mather and weak KAM theories for contact Hamiltonian systems in this non-monotonic setting. A decomposition theorem of the Ma\~n\'e set is the main result of this part.
	\end{abstract}
	
	\newpage
	
	\tableofcontents
	\newpage
	
	
	\section{Introduction}
	\subsection{Assumptions and motivation}
	\medskip
	
	\subsubsection{Lyapunov stability theory for ordinary differential equations}
	Stability is of fundamental importance in the study of the qualitative behavior of solutions of differential equations.
	Perhaps the most important stability concept is that of stability in the sense of Lyapunov, which was introduced by Lyapunov in his doctoral dissertation in 1892. He developed the stability theory for ordinary differential equations, and provided two methods for the stability analysis of an equilibrium point: Lyapunov's direct method and Lyapunov's  indirect method.  The indirect method relies on some explicit
	representation of the solutions, in particular by infinite
	series. The direct method does not require the
	prior knowledge of the solutions themselves, while it makes an essential use of auxiliary functions, called Lyapunov
	functions. The direct method is a powerful tool for the  stability analysis, but as we all know  that there are no general rules for determining Lyapunov functions.
	Notice that converse theorems provide usually no clue to the practical search for Lyapunov functions. There is a huge literature on the development of the Lyapunov theory. Here we only mention several seminal works as well as some fundamental textbooks on the stability of dynamical systems determined by ordinary differential equations, including Hahn \cite{Hahn}, Hale \cite{Hale}, Krasovskii \cite{Kra}, LaSalle and Lefschetz \cite{LL}, Yoshizawa \cite{Yos}, Zubov \cite{Zu}. Although there have been great progresses in this direction as mentioned above, generally speaking, analyzing and judging the Lyapunov stability of an equilibrium point is not an easy task.
	
	\subsubsection{The purposes of this work}
	As Crandall pointed out in \cite{Cra}, most problems in partial differential equations arising from physical models either have the form of evolution equations, which describe the change of a physical system in time, or result from seeking stationary solutions of some evolution problem.
	In this paper, we consider the evolutionary Hamilton-Jacobi equation
	\begin{align}\label{1-1}\tag{$\mathrm{HJ_e}$}
		w_t(x,t)+H(x,Dw(x,t),w(x,t))=0,\quad x\in M,\ t\in[0,+\infty),
	\end{align}	
	where the Hamiltonian $H:T^*M\times\R\to\R$, $H=H(x,p,u)$ is of class $C^3$, $M$ is a compact and connected smooth manifold without boundary, and the symbol $D$ denotes the gradient with respect to space variables $x$. This work is devoted to studying the following two problems:
	\medskip
	\begin{itemize}
		\item [{\bf(P1)}] Lyapunov stability and instability of  stationary solutions of \eqref{1-1}.
		\item [{\bf(P2)}] The uniqueness of stationary solutions of \eqref{1-1}.
	\end{itemize}
	\medskip
	Problems (P1) and (P2) are closely related and we will explain it later.

	In general, equation \eqref{1-1} has no global classical solutions. Hence, one has to consider weak solutions, for example, the generalized solutions in the sense of Kru$\mathrm{\check{z}}$kov \cite{Kru}. At the beginning of 1980s, Crandall and Lions \cite{CL83} introduced the notion of the viscosity solution for Hamilton-Jacobi equations. They got very general existence, uniqueness and continuous dependence results for viscosity solutions of many problems arising in fields of application. It is worth mentioning that the concept of viscosity solutions is closely related to some previous work by Evans in \cite{Evans1980}, where he used the Minty trick to study the vanishing viscosity method and gave definitions of possibly weak solutions. The viscosity solution theory is an adequate framework for the present work.

	\subsubsection{Solution semigroups and operator semigroup theory}
	In the literature, people often call \eqref{1-1} the Hamilton-Jacobi equation.  To distinguish \eqref{1-1} from the classical Hamilton-Jacobi equation where the Hamiltonian is defined on the cotangent bundle of $M$, introduced by Hamilton and Jacobi in the study of classical mechanics in 1834-1837, we call \eqref{1-1} the contact Hamilton-Jacobi equation, where the Hamiltonian is defined on $T^*M\times\R$.

	It is well-known that the Lax-Oleinik semigroup solution is a viscosity solution of the classical Hamilton-Jacobi equation. See for example \cite{LPV} for details which is a famous work on the homogenization problem of classical Hamilton-Jacobi equations. See \cite{Arm,Arm1,Su} and the references therein for the stochastic homogenization problem of classical Hamilton-Jacobi equations.

	For the contact Hamilton-Jacobi equation \eqref{1-1}, under the assumptions
	\medskip
	\begin{itemize}
		\item [(H1)] $\frac{\partial^2 H}{\partial p^2} (x,p,u)>0$ for each $(x,p,u)\in T^*M\times\R$;
		\item  [(H2)] for each $(x,u)\in M\times\R$, $H(x,p,u)$ is superlinear in $p$;
		\item [(H3)] there is $\lambda>0$ such that
		$$
		\Big|\frac{\partial H}{\partial u}(x,p,u)\Big| \leqslant \lambda,\quad \forall (x,p,u)\in T^*M \times\R,
		$$
	\end{itemize}
	it was shown in \cite{WWY192} that there is still a nonlinear solution semigroup, denoted by $\left\{T_t^{-}\right\}_{t \geqslant 0}: C(M, \mathbf{R}) \mapsto C(M, \mathbf{R})$, where $C(M,\R)$ denotes the space of continuous functions on $M$ with the usual uniform topology defined by the supremum norm $\|\cdot\|_\infty$. It means that $\left\{T_t^{-}\right\}_{t \geqslant 0}$ is a one-parameter semigroup of operators and for each $\varphi \in$ $C(M, \mathbf{R})$, the function $(x,t)\mapsto T_t^{-} \varphi(x)$ is the unique viscosity solution of equation \eqref{1-1} with initial value condition $w(x, 0)=\varphi(x)$. Thus, our problems are transformed into the long-time behavior problem and the uniqueness problem for fixed points of the solution semigroup. This is strongly reminiscent of the operator semigroup theory for partial differential equations. Notice that the solution  semigroup is an infinite dimensional dynamical system.

	In the late 1940s and early 1950s the operator semigroup theory and its applications to differential equations made great progress. This is a broad topic with numerous references and  the authors are competent to mention only small parts of it.
	The semigroup theory concerns the solving of the abstract Cauchy problem of this form
	\begin{align}\label{A}
		\frac{dx}{dt}+A(x(t))=0,\quad t>0
	\end{align}
	with $x(t)=x_0$. Here, $A$ denotes an operator mapping some domain $D(A)\subset X$ into a Banach space $X$, and $x_0\in D(A)$. For the linear case (i.e., $A$ is a linear operator), Hille \cite{H} and Yosida \cite{Y} established a bijective correspondence between maximal monotone operators and continuous nonexpansive semigroups, where the Yosida approximation plays an essential role.
	The Hille-Yosida theorem can be used to deal with the existence and uniqueness problems for solutions of some linear partial differential equations, such as the heat equation, the advection-diffusion equation. For the nonlinear case (i.e., $A$ is a nonlinear operator), since the Yosida approximation does not work for a nonlinear operator in
	general, new techniques are needed. By using the discrete approximation method, Crandall-Liggett \cite{CL} gave a   generation theorem and the expotential formula for the nonlinear nonexpansive semigroup when $A$ is an m-accretive operator. Crandall-Liggett theorem has been used to study the existence and uniqueness problems for some quasilinear hyperbolic equations, quasilinear and semilinear parabolic equtions, etc.
	Especially, see \cite{Aizawa,Burch,Evans} for the semigroup treatment of the classical Hamilton-Jacobi
	equation $w_t+K(Dw)=0$. Choosing a precise notion of weak solutions that is suitable to the problem at hand is a subtle problem. Different kinds of notions of solutions of partial differential equations were used in previous research in different situations in the semigroup framework, such as mild solutions, strong solutions, semiconcave solutions.

	There is of course a wide literature concerning
	the stability problems for infinite dimensional  dynamical systems. For instance, Hale \cite{Hale1} studied  the stability problem for dynamical systems arising from hyperbolic partial differential equations. Pazy \cite{Pazy} studied  the  stability problem for semigroups for nonlinear contractions in Banach spaces.
	Pritchard and Zabczyk \cite{PZ} studied the stabilizability problem for autonomous infinite dimensional systems described in terms of linear and nonlinear semigroup.
	Henry \cite{Henry} studied the stability problem for  \eqref{A}, where $A$ is a sectorial linear operator on the Banach space $X$. \cite{Martin} is devoted to some applications of semigroup theory and ordinary differential equations in Banach spaces to the asymptotic behavior of solutions for reaction-diffusion systems, where Martin got some stability results. In the aforementioned works, the Lyapunov function played an important role. See \cite{Bruck,DS,F4} for convergence results for nonlinear nonexpansive semigroups.

	\subsubsection{Difficulties and challenges}
	{\it From now on, we always assume (H1)-(H3)}.
	
	Let us pay attention to the main difficulties we will encounter in the study of (P1) and (P2): (1) As mentioned above, we already have a viscosity solution semigroup of \eqref{1-1}. Thus, we neither need a generation theorem nor take care of the choice of a siutable notion of weak solutions. But, as pointed out by Evans in \cite{Evans} the problem is that the abstract semigroup theory alone rarely provides sufficiently detailed knowledge and we have to develop such expertise about the partial differential equation itself. Note that the semigroup $\{T^-_t\}_{t\geqslant 0}$ is nonlinear and it  does not have nonexpansiveness under assumptions (H1)-(H3) in general. So, it seems that the abstract semigroup theory  will not help resolve our problems.
	(2) The Lyapunov functions were used in most of previous works on the stability problems for infinite dimensional  dynamical systems. Usually one needs some a little strong assumptions to discover the Lyapunov function. For the Hamiltonian $H$ satisfying (H1)-(H3) without a concrete form it is hard to use (generalized) Lyapunov methods. (3) We aim to study the Lyapunov stability of stationary solutions of  \eqref{1-1}, or equivalently the Lyapunov stability of equilibrium points of the solution semigroup. This is a nonlinear stability problem. We can hardly benifit from the linearization techniques provided by some existing works on related topics. (4) Comparison theorems are powerful tools for the study of the uniqueness problem for viscosity solutions of Hamilton-Jacobi equations. However, we have no idea how to apply comparison theorems to problem (P2), since the term $\frac{\partial H}{\partial u}$ may be  sign-changing.

	\subsubsection{A closer look at (P1) and (P2)}
	
	The triple $\big(C(M,\R),\{T^-_t\}_{t\geqslant 0},\R^+\big)$ is a dynamical system, where $\R^+=[0,+\infty)$.  If there exists  $\varphi \in C(M, \mathbf{R})$ such that $T^-_t \varphi=\varphi$ for all $t \in \mathbf{R}^{+}$, then $\varphi$ is called an equilibrium point of the dynamical system determined by the semigroup $\{T_t^{-}\}_{t\geqslant 0}$. In this case, it is direct to see that $\varphi$ is a stationary viscosity solution of \eqref{1-1} which means it does not change in time, or equivalently, it is a viscosity solution of
	\begin{align}\label{1-2}\tag{$\mathrm{HJ_s}$}
		H(x,Du(x),u(x))=0,\quad x\in M.
	\end{align}

	Problem (P1) concerns with the law of evolution of $T^-_t\varphi$ described by equation \eqref{1-1}.
	More precisely, our research is motivated by the question: will $T^-_t\varphi$ that emanates from a point in a sufficiently small neighborhood of an equilibrium point remains near the equilibrium point?

	Let us recall the definitions of Lyapunov stability and instability in the setting of this paper.	
	\begin{definition}
		Let	$u\in C(M,\R)$ be a stationary viscosity solution of \eqref{1-1}. Then
		\begin{itemize}
			\item [(i)] $u$ is called stable (or Lyapunov stable) if for any $\eps>0$ there exists $\delta>0$ such that for any $\varphi\in C(M,\R)$ with $\|\va-u\|_\infty<\delta$ there holds
			\[
			\|T^-_t\va-u\|_\infty<\eps, \quad \forall t>0.
			\]
			Otherwise, $u$ is unstable (or Lyapunov unstable).
			\item [(ii)] $u$ is called asymptotically stable (or Lyapunov asymptotically stable) if it is stable and there is $\delta>0$ such that
			\[
			\lim_{t\to+\infty}\|T^-_t\va-u\|_\infty=0
			\]
			for any $\varphi\in C(M,\R)$ satisfying $\|\va-u\|_\infty<\delta$. If $\delta$ can be $+\infty$, we say that $u$ is globally  asymptotically stable.
		\end{itemize}
	\end{definition}

	\begin{remark}	
		In order to understand (P1), take as an example the classical Hamilton-Jacobi equation
		\begin{align}\label{1-3}
			w_t(x,t)+F(x,Dw(x,t))=0,\quad x\in M,\ t\in[0,+\infty),
		\end{align}
		where $F=F(x,p):\T\to\R$ is a Tonelli Hamiltonian, i.e., $F$ is at least of class $C^2$, strictly convex and superlinear in the argument $p$. It is clear that 	\eqref{1-3} is a special form of equation \eqref{1-1}.
		Notice that the Lax-Oleinik semigroup $\{S_t\}_{t\geqslant 0}$ generated by \eqref{1-3} has the following properties:
		\medskip
		\begin{itemize}
			\item[(i)] $\|S_t\varphi-S_t\phi\|_\infty\leqslant \|\varphi-\phi\|_\infty$, $\forall \varphi$, $\phi\in C(M,\R)$, $\forall t\geqslant 0$;
			\item [(ii)] $S_t(\varphi+a)=S_t\varphi + a$, $\forall t\geqslant 0$, $\forall a\in\R$.
		\end{itemize}
		In view of (i), all stationary viscosity solutions of \eqref{1-3} are stable. By (ii),  if $u$ is a stationary viscosity solution, then $u+a$ is still a stationary viscosity solution for any $a\in\R$. Thus, none of stationary viscosity solutions is  asymptotically stable.

		However, in general, the above two properties do not hold  for $\{T^-_t\}_{t\geqslant 0}$ generated by \eqref{1-1}.
	\end{remark}

	\begin{remark}	
		Note that the stability problem (P1) is quite different from the ones which have been deeply studied by many authors in various situations. A standard stability result in the study of viscosity solution theory shows that if $H_k \rightarrow H$, $w_{0, k} \rightarrow w_0$,  $w_k \rightarrow w$  locally uniformly, as $k\to+\infty$,
		and for each $k \in \mathbf{N}$, $w_k$ is a viscosity solution to
		\begin{align*}
			({w_k})_t(x,t)+H_k(x,Dw_k(x,t),w_k(x,t))=0,\quad w_k(x,0)=w_{0,k}(x),
		\end{align*}
		then $w$ is a viscosity solution to \eqref{1-1} with $w(x,0)=w_0(x)$.
		There is a large literature on this issue. See, for instance, \cite{L82, Bar} for details.
	\end{remark}

	Problem (P2) concerns with the uniqueness problem of viscosity solutions of \eqref{1-2}.
	The question of uniqueness of the solutions sometimes is more difficult than the existence one. There are many different kinds of existing uniqueness results for viscosity solutions of Hamilton-Jacobi equations in
	the literature, such as \cite{CL83,CEL84,I1,I2,I3}, where the uniqueness follows from comparison theorems of maximum principle type. See \cite{J} and \cite{C} for uniqueness results for viscosity solutions of \eqref{1-2}, where $H(x,p,u)$ is nondecreaing and convex in $u$.  As far as we know, little has been known about the uniquess of viscosity solutions of \eqref{1-2} without the monotonicity condition imposed on $H(x,p,u)$ with respect to $u$. We aim to provide several uniqueness results for viscosity solutions of \eqref{1-2} under the assumptions (H1)-(H3).
	
	As explained in Sect.1.1.4, we need to develop new methods and techniques to deal with (P1) and (P2). Our idea is inspired by the action minimizing method for the study of Hamiltonian dynamical systems. The central tool is the Mather measure.

	\subsubsection{Methods and tools}	
	In the 1980s, the Aubry-Mather theory \cite{A,M82} provided a variational approach to study the dynamics of twist diffeomorphisms of the annulus and got the existence of different action minimizing sets. Later, Moser \cite{Moser} pointed out that a smooth area-preserving monotone twist mapping of an annulus can be interpolated by the flow of a time-periodic and  convex Hamiltonian system. In the 1990s, Mather \cite{M91,M93} generalized the Aubry-Mather theory to higher dimensional convex Hamiltonian systems. Later, Fathi \cite{F1,F2,F3,F4,F-b} established the weak KAM theory to connect the Mather theory and viscosity solution theory for Hamilton-Jacobi equations. There are many interesting works on Mather and weak KAM theories, see for example, \cite{Ar,Ber,CIPP,Con} and references therein. A nice introductory lecture notes on Mather theory has been written by Sorrentino \cite{S-b}.

	Recently, several authors attempted to study the Aubry-Mather and weak KAM theories for convex contact Hamiltonian systems. Mar\`o and Sorrentino \cite{MS}  developed an analogue of Aubry-Mather theory for the conformally symplectic system. Mitake and Soga \cite{MK} developed the weak KAM theory for discounted Hamilton-Jacobi equations and the corresponding discounted Lagrangian and Hamiltonian dynamics. One thing the aforementioned two works have in common is that the Hamiltonian has the discounted form: $H(x,p,u)=u+K(x,p)$, which is a specific contact one. Consider more general contact Hamiltonian systems
	\begin{align}\label{c}\tag{C}
		\left\{
		\begin{array}{l}
			\dot{x}=\frac{\partial H}{\partial p}(x,p,u),\\[1mm]
			\dot{p}=-\frac{\partial H}{\partial x}(x,p,u)-\frac{\partial H}{\partial u}(x,p,u)p,\qquad (x,p,u)\in T^*M\times\mathbf{R},\\[1mm]
			\dot{u}=\frac{\partial H}{\partial p}(x,p,u)\cdot p-H(x,p,u).
		\end{array}
		\right.
	\end{align}
	The authors of \cite{WWY19} provided the Aubry-Mather and weak KAM type results for \eqref{c} where the Hamiltonian satisfies (H1), (H2) and
	\begin{align}\label{incre1}
		0<\frac{\partial H}{\partial u}(x,p,u)\leqslant \lambda,\quad \forall (x,p,u)\in T^*M\times\R.
	\end{align}
	
	In order to study (P1) and (P2), we need to use some  Aubry-Mather and weak KAM type results for \eqref{c}. But the existing results can not be applied to our case, since (H3) is weaker than \eqref{incre1}.
	It is worth noting that $H$ in the present work may be  non-monotonic in the argument $u$, where the situation will be quite different from the monotonic case.
	Hence, we will first prove some Aubry-Mather and weak KAM type results for \eqref{c} under assumptions (H1)-(H3). On the one hand, this will play an essential role in the study of (P1) and (P2). On the other hand, it has an independent dynamical significance.

	We will use the information carried by a stationary viscosity solution $u_-$ of \eqref{1-1} and by Mather measures whose supports are contained in the 1-graph of $u_-$ to judge the sability and instability of $u_-$. We will use the information of all Mather measures to give several uniqueness results for stationary viscosity solutions of \eqref{1-1}. To summarize, we provide sufficient conditions for the stability, instability and uniqueness results for stationary viscosity solutions of \eqref{1-1} by utilizing the information coming from the stationary viscosity solution itself and from the dynamical system \eqref{c}, which is the characteristic equations of \eqref{1-1}.
	
	We refer the reader to the series of the papers on contact Hamiltonian systems by de Le\'on and his cooperators. See for instance \cite{E,Leon}.

	\subsection{Main results}
	Tonelli Theorem is very important in  Mather and weak KAM theories for Hamiltonian systems. For contact Hamiltonian systems, in order to study the action minimizing orbits by variational methods, the authors of \cite{WWY17} introduced an  implicit variational principle for \eqref{c}. See Proposition \ref{prop4} below for details.
	Our study will start with this variational principle. First of all, we give the definition of the semi-static orbit by using the action function which plays a key role in the variational principle. Then we call the set of all semi-static orbits the Ma\~n\'e set $\tilde{\mathcal{N}}$ of \eqref{c}. We say that a Borel probability measure $\mu$ on $T^*M\times\R$ is a Mather measure if $\mu$ is $\Phi^H_t$-invariant and $\supp \mu$ is contained in $\tilde{\mathcal{N}}$, where
	$\Phi^H_t$ stands for the local flow of \eqref{c}.
	Denote by $\mathfrak{M}$ the set of all Mather measures.
	See Section 3.1 for the defnitions of semi-static orbits, the Ma\~n\'e set and Mather measures. Note that all the above definitions are independent of viscosity solutions of \eqref{1-2}.

	Let $\mathcal{S}^-$ denote the set of viscosity solutions of equation \eqref{1-2}.  We need to clarify the existence of viscosity solutions of \eqref{1-2}. Under assumptions (H1)-(H3) it was shown in \cite{WWY192} that there exists a constant $c$ such that
	\begin{align}\label{1-666}
		H(x,Du(x),u(x))=c
	\end{align}
	has viscosity solutions. Recently, under the same assumptions the authors of \cite{WY21} showed that \eqref{1-666} admits viscosity solutions if and only if $c$ belongs to a nonempty interval, which is called the admissible set for the generalized ergodic problem \eqref{1-666} in \cite{WY21}.

	For any $u_-\in\s$, there exists a compact $\Phi^H_t$-invariant subset of $T^*M\times\R$, denoted by $\tilde{\mathcal{N}}_{u_-}$. See Section 3.3 for the definition,  existence and properties of $\tilde{\mathcal{N}}_{u_-}$.

	The first main result of this paper is stated as follows.
	\begin{theorem}\label{Man}
		\[
		\tilde{\mathcal{N}}=\bigcup_{u_-\in\mathcal{S}^-}\tilde{\mathcal{N}}_{u_-}.
		\]
	\end{theorem}
	In view of Theorem \ref{Man}, we call $\tilde{\mathcal{N}}_{u_-}$ the Ma\~n\'e set associated with $u_-$.
	From Proposition \ref{pr1011} below,  $\tilde{\mathcal{N}}_{u_-}$ is a subset of
	\[
	\Lambda_{u_-}:=\operatorname{cl}\Big(\big\{(x,p,u): x \ \text{is a point of differentiability of} \ u_{-}, p=Du_{-}(x),u=u_{-}(x)\big\}\Big),
	\]
	where $\operatorname{cl}(B)$ denotes the closure of $B \subset \T \times \mathbf{R}$. More precisely,
	\begin{align}\label{1-1111}
		\tilde{\mathcal{N}}_{u_-}=\bigcap_{t\geq 0}\Phi^H_{-t}(\Lambda_{u_-}).
	\end{align}
	Notice that $u_-$ is Lipschitz and $M$ is compact. Thus, $\Lambda_{u_-}$ is a compact subset of $\T\times\R$. Moreover, $\Lambda_{u_-}$ is negatively invariant under $\Phi^H_t$.
	By the Krylov-Bogoliubov Theorem, there exist Borel $\Phi^H_t$-invariant
	probability measures supported in $\Lambda_{u_-}$.
	Hence, the set
	\begin{align}\label{M}
		\mathfrak{M}_{u_-}:=\big\{\mu\in\mathcal{P}(\T\times\R): \supp\mu\subset \Lambda_{u_-},\ (\Phi^H_t)_\sharp\mu=\mu,\,\ \forall t\in\mathbf{R}\big\}
	\end{align}
	is nonempty, where  $\mathcal{P}(\T\times\R)$ denotes the space of Borel probability
	measures on $\T\times\R$, and $(\Phi^H_t)_\sharp\mu$ denotes the push-forward of $\mu$ through $\Phi^H_t$.
	
	By the definitions of $\mathfrak{M}_{u_-}$, $\mathfrak{M}$ and Theorem \ref{Man}, \eqref{1-1111}, one can deduce that
	\begin{corollary}\label{mather}
		\[
		\mathrm{co}\Big(\bigcup_{u_-\in\mathcal{S}^-}\mathfrak{M}_{u_-}\Big)=\mathfrak{M},
		\]
		where $\mathrm{co}(B)$ denotes the convex hull of $B \subset \mathcal{P}(\T\times\R)$.
	\end{corollary}
	Due to Corollary \ref{mather}, we say a measure $\mu\in \mathfrak{M}_{u_-}$ is a Mather measure associted with $u_-$.
	Consider the condition
	\begin{align}\label{A1}\tag{A1}
		\int_{T^*M\times\R}\frac{\partial H}{\partial u}(x,p,u)d\mu>0, \quad \forall \mu\in\mathfrak{M}_{u_-}.
	\end{align}
	Note that $\mathfrak{M}_{u_-}$ is convex and compact with respect to the weak topology or weak-$^*$ topology.
	The above condition implies that there is a constant $A>0$ such that
	\begin{align}\label{A2}\tag{A1'}
		\int_{T^*M\times\R}\frac{\partial H}{\partial u}(x,p,u)d\mu>A, \quad\forall \mu\in\mathfrak{M}_{u_-}.
	\end{align}

	We are now in a position to give a criterion for the Lyapunov stability of stationary solutions.
	\begin{theorem}\label{th1}
		Let $u_-\in \mathcal{S}^-$ satisfy \eqref{A1}.
		Then $u_-$ is locally asymptotically stable. More precisely, there is $\Delta>0$, such that for any $\varphi\in C(M,\R)$ satisfying $\|\varphi-u_-\|_{\infty}\leqslant\Delta$,
		\[	\limsup_{t\to +\infty}\dfrac{\ln\|T^-_t\varphi-u_-\|_{\infty}}{t}\leqslant-A,
		\]
		where  $A$ is as in \eqref{A2}.
	\end{theorem}

	Recall that equation \eqref{1-3} is a special case of \eqref{1-1} where the Hamiltonian does not depend on the argument $u$. As mentioned above,  all
	stationary solutions of \eqref{1-3} are stable. But, none of them is asymptotically stable. Theorem \ref{th1} tells us that a new dynamical phenomenon  appears when the Hamiltonian depends on the argument $u$. Comparing the characteristic equations \eqref{c} of \eqref{1-1} with the ones of \eqref{1-3}:
	\begin{align}\label{h}\tag{H}
		\left\{
		\begin{array}{l}
			\dot{x}=\frac{\partial F}{\partial p}(x,p),\\[1mm]
			\dot{p}=-\frac{\partial F}{\partial x}(x,p),
		\end{array}
		\right.
	\end{align}
	it is not hard to see that the term $\frac{\partial H}{\partial u}$ makes the essential differences between them. The sign of the integral of $\frac{\partial H}{\partial u}$ with respect to Mather measures plays an essential  role in Theorem \ref{th1}.


	The next result is on the Lyapunov instability of stationary solutions.
	\begin{theorem}\label{th2}
		Let  $u_-\in \mathcal{S}^-$. If
		\begin{align}\label{A3}\tag{A2}
			\int_{T^*M\times\R}\frac{\partial H}{\partial u}(x,p,u)d\mu<0, \quad\text{for some}\  \mu\in\mathfrak{M}_{u_-},
		\end{align}
		then $u_-$ is unstable. More precisely, there exists a constant $\Delta'>0$ such that for any $\eps>0$, there is $\varphi_\eps\in C(M,\R)$ satisfying $\|\varphi_\eps-u_-\|_{\infty}\leqslant\eps$ and $$\limsup_{t\to+\infty}\|T^-_t\varphi_\eps-u_-\|_{\infty}\geqslant\Delta'.$$
	\end{theorem}
	
	In \cite{X}, the Lyapunov stability and instability problems for stationary solutions of equation \eqref{1-1} were studied, where the state space is the unit circle. Assume
	\[
	\frac{\partial H}{\partial p}(x,p,u)\Big|_{\Lambda_{u_-}}\neq 0, \quad (x,p,u)\in\mathbf{S}\times\R\times\R.
	\]
	Then the simple structure of the unit circle implies that the Aubry set consists of  a periodic orbit. The authors of \cite{X} used this fact and the subsolution method to get stability and instability results. We can not follow this line since the state space $M$ in this paper is a manifold of arbitrary dimension.

	Let us take a look at an example, where Theorem \ref{th1} and Theorem \ref{th2} can be applied.
	\begin{example}\label{ex1}
		Consider the equation
		\begin{align}\label{1313}
			w_t(x,t)+\|Dw(x,t)\|^2-\langle Dg(x),Dw(x,t)\rangle-f(x)(w(x,t)-g(x))=0,\quad x\in M,
		\end{align}
		where $f$, $g$ are smooth functions on $M$. It is clear that $g(x)$ is a classical solution of  \eqref{1313}. We will use Theorems \ref{th1} and \ref{th2} to analyze the stability and instability of $g(x)$.
		\begin{itemize}
			\item  If $f(x)<0$ on $\{x\in M: Dg(x)=0\}$, then assumption \eqref{A1} holds ture. By Theorem \ref{th1},
			$g(x)$ is  locally asymptotically stable.
			\item If $f(\bar{x})>0$ for some $\bar{x}$ in  $\{x\in M: Dg(x)=0\}$, then assumption \eqref{A3} holds ture. By Theorem \ref{th2},
			$g(x)$ is unstable.
		\end{itemize}
		See Section 3 for details.
	\end{example}


	The last main result of the present paper gives a criterion for the uniqueness of viscosity solutions of \eqref{1-2} by using the information from $\mathfrak{M}$.
	\begin{theorem}\label{th3}
		If
		\begin{align}\label{A4}\tag{A3}
			\int_{T^*M\times\R}\frac{\partial H}{\partial u}(x,p,u)d\mu>0,\quad 	\forall \mu\in\mathfrak{M},
		\end{align}
		then equation \eqref{1-2} has at most one viscosity solution. Moreover, if $u_-$ is the unique viscosity of equation \eqref{1-2}, then $u_-$ is globally asymptotically stable. 	
	\end{theorem}
	
	
	\begin{remark}
		It is well known that equation \eqref{1-2} admits a unique viscosity solution when $H(x,p,u)$ is strictly increasing in $u$. Jing et al. \cite{J} showed that nonuniqueness appears even when $\frac{\partial H}{\partial u}\geqslant 0$. In \cite{J} and \cite{WWY21}, examples on nonuniqueness of viscosity solutions of \eqref{1-2} were given when $H(x,p,u)$ is strictly decreasing in $u$. Assuming $H(x,p,u)$ is nondecreasing and convex in $u$, the authors of \cite{J} provided an interesting uniqueness result for viscosity solutions of \eqref{1-2}, where the set of adjoint measures played an important role. As they pointed out, the  set of adjoint measures is defined implicitly as it depends on all viscosity solutions of \eqref{1-2}, which are not known a priori. When $H(x,p,u)$ does not depend on $u$, the adjoint measure is a projected Mather measure for the  Hamiltonian system generated by $H$.
	\end{remark}
	
	\begin{remark}\label{re}
		Equation \eqref{1-2} determines a hypersurface in the space of 1-jets of functions of $x\in M$.
		Notice that there is a common point among assumptions (A1)-(A3): we only focus on the average value of term $\frac{\partial H}{\partial u}(x,p,u)$ over subsets of the level surface of $H$
		\[
		\mathcal{E}:=\{(x,p,u)\in T^*M\times\R: H(x,p,u)=0\}.
		\]
	\end{remark}
	
	A direct consequence of Theorem \ref{th3} is stated as follows.
	\begin{corollary}\label{addcoro1}
		Let $H\in C^3(T^*M\times\R,\R)$ satisfy  conditions (H1)-(H3) and
		\begin{align}\label{1-202}
			H(x,p,u)=H(x,-p,u), \quad \forall (x,p,u)\in T^*M\times\R.
		\end{align}
		Define $\mathcal{B}:=\Big\{(x,u)\in M\times\R:  H(x,0,u)=\frac{\partial H}{\partial x}(x,0,u)=0\Big\}$.
		If
		\begin{align}\label{1-206}
			\frac{\partial H}{\partial u}(x,0,u)\Big|_\mathcal{B}>0,
		\end{align}
		then \eqref{1-2} has at most one viscosity solution.
	\end{corollary}

	\begin{example}\label{ex2}
		Consider
		\begin{align}\label{1515}
			\frac{1}{2}\|Du(x)\|^2+\frac{1}{2}u(x)+\sin u(x)=0,\quad x\in M.	
		\end{align}
		It is clear that $u(x)\equiv0$ is a  solution of \eqref{1515}. Here, $H(x,p,u)=\frac{1}{2}\|p\|^2+\frac{1}{2}u+\sin u $. We will check in Section 4 that \eqref{1-206} holds true, which implies that $u(x)=0$ is the unique viscosity solution of \eqref{1515}.
	\end{example}

	

	Another consequence of Theorem \ref{th3} is stated  as follows.

	\begin{corollary}\label{addcoro2}
		If $\frac{\partial H}{\partial u}(x,p,u)\Big|_\mathcal{E}\geqslant0$ and
		\begin{equation}\label{add1.10}
			\Big((\frac{\partial H}{\partial u})^2+(\mathfrak{L}_H\frac{\partial H}{\partial u})^2+(\mathfrak{L}^2_H\frac{\partial H}{\partial u})^2\Big)\neq0, \quad\forall  (x,p,u)\in\mathcal{E},
		\end{equation}
		then \eqref{1-2} has at most one viscosity solution, where $\mathfrak{L}_H$ is the Lie derivative and $\mathfrak{L}_HF(x,p,u)=\frac{d}{dt}\big|_{t=0}F(\Phi^H_t(x,p,u))$, $\mathfrak{L}^k_HF(x,p,u)=\frac{d^k}{dt^k}\big|_{t=0}F(\Phi^H_t(x,p,u))$ for all $F\in C(T^*M\times\R,\R)$.
	\end{corollary}
	\begin{remark}
		It is worth pointing out that if the Hamiltonian $H\in C^{k+1}(T^*M\times\R,\R)\ (k\geqslant2)$ and
		\[
		\Big((\frac{\partial H}{\partial u})^2+(\mathfrak{L}_H\frac{\partial H}{\partial u})^2+\cdots+(\mathfrak{L}_H^k\frac{\partial H}{\partial u})^2\Big)\neq0,\quad\forall  (x,p,u)\in\mathcal{E},
		\]
		then \eqref{1-2} has at most one viscosity solution. Similarly, if $H\in C^{\infty}(T^*M\times\R,\R)$ and for any $(x,p,u)\in\mathcal{E}$, there is $k\in\mathbf{N}$ such that  $\mathfrak{L}_H^k\frac{\partial H}{\partial u}\neq0$, then the result still holds true.
		The proofs are quite similar to the one of Corollary \ref{addcoro2}.
	\end{remark}
	
	We give another interesting example to finish this section. Theorem \ref{th1}, \ref{th2}, \ref{th3} and Corollary \ref{addcoro2} will be used to analyze this example.
	\begin{example}\label{ex3}
		Consider the contact Hamilton-Jacobi equation
		\begin{align}\label{addeq1.11}
			w_t(x,t)+\frac{1}{2}w^2_x(x,t)-a\cdot w_x(x,t)+(\sin x+b)\cdot w(x,t)=0,\quad x\in \mathbf{S},	
		\end{align}
		where $\mathbf{S}$ is the unit circle and $a$, $b\in\R$ are parameters. Here $H(x,p,u)=\frac{1}{2}p^2-a\cdot p+(\sin x+b)\cdot u$. It is clear that the constant function $w=0$ is a solution of \eqref{addeq1.11}.

		By using Theorem \ref{th1}, \ref{th2}, \ref{th3} and Corollary \ref{addcoro2}, we get the following results listed in Table 1.
		\begin{table}[htbp]
			\centering
			\caption{Stationary solution $w=0$: stability, instability and uniqueness}
			\vskip0.1cm
			\begin{tabular}{|c|c|c|}
				\hline
				\ & $a=0$ & $a\neq0$ \\
				\hline
				$b>1$ & unique, globally asymptotically stable & \multirow{2}{*}{ unique, globally asymptotically stable}\\
				\cline{1-2} $b=1$ & critical case& \\
				\hline
				$0<b<1$ & \multirow{3}{*}{unstable} & locally asymptotically stable \\
				\cline{1-1}\cline{3-3} $b=0$ & & critical case \\
				\cline{1-1}\cline{3-3} $b<0$ & & unstable \\
				\hline
			\end{tabular}
		\end{table}
	\end{example}


	The rest of the paper is organized as follows. Section 2 collects some preliminary results and tools which will be used later. The readers can skip it and keep reading. Section 3 is devoted to the Mather and weak KAM theories for contact Hamiltonian system \eqref{c}.	
	We will prove Theorem \ref{th1}, Theorem \ref{th2}  and analyze Example \ref{ex1} in Section 4. The uniqueness problem will be studied in the last section, where we will  revisit Example \ref{ex2} and \ref{ex3}.

	\section{Notations and Preliminaries}

	\subsection{Notations}	
	We write as follows a list of symbols used throughout this paper.
	\begin{itemize}
		\item We choose, once and for all, a $C^{\infty}$ Riemannian metric on $M$. It is classical that there is a canonical way to associate to it a Riemannian metric on the tangent bundle $TM$. We use the same symbol $d$ to denote the distance function defined by the Riemannian metric on $M$ and the distance function defined by the Riemannian metric on $TM$ or on $TM\times\R$. We use the same symbol $\|\cdot\|_x$ to denote the
		norms induced by the Riemannian metrics on $T_xM$ and $T^*_xM$ for $x\in M$, and by $\langle \cdot,\cdot\rangle_x$ the canonical pairing between the tangent space $T_xM$ and the cotangent space $T^*_xM$.  Sometimes we will use $\|\cdot\|$ and $\langle
		\cdot,\cdot\rangle$ to denote $\|\cdot\|_x$ and $\langle \cdot,\cdot\rangle_x$ for brevity, respectively.
		\item $C_b(TM\times\R,\R)$ stands for the space of all continuous and bounded functions on $TM\times\R$.
		\item  $\mathcal{P}(\T\times\R)$ denotes the space of Borel probability measures on $\T\times\R$.
		\item  $Du(x)=(\frac{\partial u}{\partial x_1},\dots,\frac{\partial u}{\partial x_n})$ and $ Dw(x,t)=(\frac{\partial w}{\partial x_1},\dots,\frac{\partial w}{\partial x_n})$.
		\item $\mathcal{S}^-$ (resp. $\mathcal{S}^+$) denotes the set of all  backward (resp. forward) weak KAM solutions of equation \eqref{1-2}. Backward weak KAM solutions and viscosity solutions are the same in the setting of this paper.
		\item Let $\Phi^H_t$ (resp. $\Phi^L_t$) denote the local flow of contact Hamiltonian system (\ref{c}) (resp. Lagrangian system \eqref{L} below).
		\item $h_{x_0,u_0}(x,t)$ (resp. $h^{x_0,u_0}(x,t)$) denotes the forward (resp. backward) implicit action function associated with $L$.
		\item $\{T^{-}_t\}_{t\geqslant 0}$ (resp. $\{T^{+}_t\}_{t\geqslant 0}$) denotes the backward (resp. forward) solution semigroup associated with $L$.
		\item Let $(X_1,\mathcal{B}_1,\mu)$ be a measure space, $(X_2,\mathcal{B}_2)$ a measurable space, and $f:X_1\to X_2$ a measurable map. The push-forward of $\mu$ through $f$ is the measure $f_\sharp \mu$ on $(X_{2}, \mathcal{B}_{2})$ defined by
		\begin{equation*}
			f_\sharp \mu(B):=\mu\left(f^{-1}(B)\right), \quad \forall B \in \mathcal{B}_2.
		\end{equation*}
		The push-forward has the property that a measurable map $g: X_{2} \to \mathbf{R}$ is integrable with respect to $f_\sharp \mu$ if and only if $g \circ f$ is integrable on $X_{1}$ with respect to $\mu$. In this case, we have that
		\begin{equation*}
			\int_{X_1}g(f(x))\,d\mu(x)=\int_{X_2}g(y)\,df_\sharp\mu(y).
		\end{equation*}
	\end{itemize}

	\subsection{Preliminaries}
	The authors of \cite{WWY17} provided the implicit variational principle for contact Hamiltonian systems and introduced the notion of action functions. Later, by using solution semigroups, they studied the existence of viscosity solutions of contact Hamilton-Jacobi equations in \cite{WWY192}.
	The assumptions used in \cite{WWY17,WWY192} are the same as the ones in the present paper. We collect some known results from \cite{WWY17,WWY192} and also prove some new results in this part.
	
	\subsubsection{Implicit variational principles for contact Hamiltonian systems}
	We use $\mathcal{L}: T^*M\rightarrow TM$ to denote  the Legendre transform. Let
	$\bar{\mathcal{L}}:=(\mathcal{L}, Id)$, where $Id$ denotes the identity map from $\R$ to $\R$. Then
	\[
	\bar{\mathcal{L}}:T^*M\times\R\to TM\times\R,\quad (x,p,u)\mapsto \left(x,\frac{\partial H}{\partial p}(x,p,u),u\right)
	\]
	is a diffeomorphism. Using $\bar{\mathcal{L}}$, we can define
	the contact Lagrangian $L(x,\dot{x},u)$ associated to $H(x,p,u)$ as
	\[
	L(x,\dot{x},u):=\sup_{p\in T^*_xM}\{\langle \dot{x},p\rangle_x-H(x,p,u)\},\quad (x,\dot{x},u)\in TM\times\R.
	\]
	By (H1)-(H3), it is direct to check that
	\medskip
	
	\begin{itemize}
		\item [\bf(L1)] $\frac{\partial^2 L}{\partial \dot{x}^2} (x,\dot{x},u)>0$ for each $(x,\dot{x},u)\in TM\times\R$;
		\item  [\bf(L2)] for each $(x,u)\in M\times\R$, $L(x,\dot{x},u)$ is superlinear in $\dot{x}$;
		\item [\bf(L3)] there is $\lambda>0$ such that
		$$
		\Big|\frac{\partial L}{\partial u}(x,\dot{x},u)\Big| \leqslant \lambda,\quad \forall(x,\dot{x},u)\in TM \times\R.
		$$
	\end{itemize}
	The contact Lagrangian system reads
	\begin{align}\label{L}\tag{CL}	
		\left\{
		\begin{array}{l}
			\frac{d}{dt}\frac{\partial L}{\partial \dot{x}}\big(x(t),\dot{x}(t),u(t)\big)=\frac{\partial L}{\partial x}\big(x(t),\dot{x}(t),u(t)\big)+\frac{\partial L}{\partial u}\big(x(t),\dot{x}(t),u(t)\big)\cdot \frac{\partial L}{\partial \dot{x}}\big(x(t),\dot{x}(t),u(t)\big),\\[1mm]
			\dot{u}(t)=L\big(x(t),\dot{x}(t),u(t)\big).
		\end{array}
		\right.	
	\end{align}

	\begin{proposition}(\cite[Theorem A]{WWY17},\cite{WWY192})\label{prop4}
		For any given $x_0\in M$, $u_0\in\R$, there exist two continuous functions $h_{x_0,u_0}(x,t)$ and $h^{x_0,u_0}(x,t)$ defined on $M\times(0,+\infty)$ satisfying	
		\begin{align}
			h_{x_0,u_0}(x,t)&=u_0+\inf_{\substack{\gamma(0)=x_0 \\  \gamma(t)=x} }\int_0^tL\big(\gamma(\tau),\dot{\gamma}(\tau),h_{x_0,u_0}(\gamma(\tau),\tau)\big)d\tau,\label{2-1}\\
			h^{x_0,u_0}(x,t)&=u_0-\inf_{\substack{\gamma(t)=x_0 \\  \gamma(0)=x } }\int_0^tL\big(\gamma(\tau),\dot{\gamma}(\tau),h^{x_0,u_0}(\gamma(\tau),t-\tau)\big)d\tau,\label{2-2}
		\end{align}
		where the infimums are taken among the Lipschitz continuous curves $\gamma:[0,t]\rightarrow M$.
		Moreover, the infimums in (\ref{2-1}) and \eqref{2-2} can be achieved.
		If $\gamma_1$ and $\gamma_2$ are curves achieving the infimums \eqref{2-1} and \eqref{2-2} respectively, then $\gamma_1$ and $\gamma_2$ are of class $C^1$.
		Let
		\begin{align*}
			x_1(s)&:=\gamma_1(s),\quad u_1(s):=h_{x_0,u_0}(\gamma_1(s),s),\,\,\,\qquad  p_1(s):=\frac{\partial L}{\partial \dot{x}}\big(\gamma_1(s),\dot{\gamma}_1(s),u_1(s)\big),\\
			x_2(s)&:=\gamma_2(s),\quad u_2(s):=h^{x_0,u_0}(\gamma_2(s),t-s),\quad   p_2(s):=\frac{\partial L}{\partial \dot{x}}\big(\gamma_2(s),\dot{\gamma}_2(s),u_2(s)\big).
		\end{align*}
		Then $(x_1(s),p_1(s),u_1(s))$ and $(x_2(s),p_2(s),u_2(s))$ satisfy equations \eqref{c} with
		\begin{align*}
			x_1(0)=x_0, \quad x_1(t)=x, \quad \lim_{s\rightarrow 0^+}u_1(s)=u_0,\\
			x_2(0)=x, \quad x_2(t)=x_0, \quad \lim_{s\rightarrow t^-}u_2(s)=u_0.
		\end{align*}
	\end{proposition}

	We call $h_{x_0,u_0}(x,t)$ (resp. $h^{x_0,u_0}(x,t)$) a forward (resp. backward) implicit action function associated with $L$
	and the curves achieving the infimums in (\ref{2-1}) (resp. \eqref{2-2}) minimizers of $h_{x_0,u_0}(x,t)$ (resp. $h^{x_0,u_0}(x,t)$).  The following are some propositions about the implicit action functions which have been proved in previous papers.

	\begin{proposition}(\cite[Proposition 3.5]{WWY192})\label{addprop1}
		Let $x_0$, $x\in M$, $u_0$, $u\in \R$ and $t\in(0,+\infty)$. Then $h_{x_0,u_0}(x,t)=u\Leftrightarrow h^{x,u}(x_0,t)=u_0.$
	\end{proposition}

	\begin{proposition}(\cite[Lemma 2.1]{WWY192})\label{prop5}
		For any given $a,b,\delta,T\in\R$ with $a<b$, $0<\delta<T$, denote $\Omega_{a,b,\delta,T}:=M\times[a,b]\times M\times[\delta,T].$ Then there exists a compact set $\mathcal{K}:=\mathcal{K}_{a,b,\delta,T}\subset T^*M\times \R$ such that for any $(x_0,u_0,x,t)\in\Omega_{a,b,\delta,T}$ and any minimizer $\gamma(s)$ of $h_{x_0,u_0}(x,t)$, we have
		$$\big(\gamma(s),p(s),u(s)\big)\subset\mathcal{K},\ \forall s\in[0,t],$$
		where $u(s)=h_{x_0,u_0}(\gamma(s),s)$, $p(s)=\frac{\partial L}{\partial\dot{x}}\big(\gamma(s),\dot{\gamma}(s),u(s)\big)$ and $\mathcal{K}$ depends only on $a,b,\delta$ and $T$.
	\end{proposition}

	\begin{proposition}(\cite[Theorem C, Theorem D]{WWY17},\cite[Prposition 3.3]{WWY192})\label{prop3}
		\item [(1)] Given $x_0\in M$, $u_0\in \R$,
		$$h_{x_0,u_0}(x,t+s)=\inf_{y\in M}h_{y,h_{x_0,u_0}(y,t)}(x,s)$$
		for all $s,t>0$ and all $x\in M$. Moreover, the infimum is attained at $y$ if and only if there exists a minimizer $\gamma$ of $h_{x_0,u_0}(x,t+s)$ with $\gamma(t)=y$.
		\item [(2)] The function $(x_0,u_0,x,t)\to h_{x_0,u_0}(x,t)$ is Lipschitz continuous on $M\times[a,b]\times M\times[c,d]$ for all real numbers $a,b,c,d$ with $a<b$ and $0<c<d$.
		\item [(3)] Given $x_0\in M$ and $u_1$, $u_2\in \R$,
		if $u_1<u_2$, then $h_{x_0,u_1}(x,t)<h_{x_0,u_2}(x,t)$, for all $(x,t)\in M\times(0,+\infty)$.
	\end{proposition}

	\begin{proposition}(\cite[Theorem 3.1]{WWY192})\label{prop8}
		\item [(1)] Given $x_0\in M$, $u_0\in \R$,
		$$h^{x_0,u_0}(x,t+s)=\sup_{y\in M}h^{y,h^{x_0,u_0}(y,t)}(x,s)$$
		for all $s,t>0$ and all $x\in M$. Moreover, the supremum is attained at $y$ if and only if there exists a minimizer $\gamma$ of $h^{x_0,u_0}(x,t+s)$ with $\gamma(t)=y$.
		\item [(2)] The function $(x_0,u_0,x,t)\to h^{x_0,u_0}(x,t)$ is Lipschitz continuous on $M\times[a,b]\times M\times[c,d]$ for all real numbers $a,b,c,d$ with $a<b$ and $0<c<d$.
		\item [(3)] Given $x_0\in M$ and $u_1,u_2\in \R$,  if $u_1<u_2$, then $h^{x_0,u_1}(x,t)<h^{x_0,u_2}(x,t)$, for all $(x,t)\in M\times(0,+\infty)$.
	\end{proposition}

	\subsubsection{Solution semigroups for contact Hamilton-Jacobi equations}
	Let us recall two semigroups of operators introduced in \cite{WWY192}. Define a family of nonlinear operators $\{T^-_t\}_{t\geqslant 0}$ from $C(M,\mathbf{R})$ to itself as follows. For each $\varphi\in C(M,\mathbf{R})$, denote by $(x,t)\mapsto T^-_t\varphi(x)$ the unique continuous function on $ (x,t)\in M\times[0,+\infty)$ such that
	\begin{equation}\label{eq:semigroup}
		T^-_t\varphi(x)=\inf_{\gamma}\left\{\varphi(\gamma(-t))+\int_{-t}^0L\big(\gamma(\tau),\dot{\gamma}(\tau),T^-_{t+\tau}\varphi(\gamma(\tau))\big)d\tau\right\},
	\end{equation}
	where the infimum is taken among the absolutely continuous curves $\gamma:[-t,0]\to M$ with $\gamma(0)=x$. The infimum in \eqref{eq:semigroup} can be achieved and a curve achieving the infimum is called a minimizer of $T^-_t\varphi(x)$. Minimizers are of class $C^1$. Let $\gamma$ be a minimizer and $x(s):=\gamma(s)$, $u(s):=T_s^-\varphi(x(s))$, $p(s):=\frac{\partial L}{\partial \dot{x}}(x(s),\dot{x}(s),u(s))$.
	Then $(x(s),p(s),u(s))$ satisfies equations (\ref{c}) with $x(t)=x$. $\{T^-_t\}_{t\geqslant 0}$ is a semigroup of operators and the function $(x,t)\mapsto T^-_t\varphi(x)$ is a viscosity solution of \eqref{1-1} with initial condition $w(x,0)=\varphi(x)$. Thus, we call $\{T^-_t\}_{t\geqslant 0}$ the backward solution semigroup. Similarly, one can define another semigroup of operators $\{T^+_t\}_{t\geqslant 0}$, called the forward solution semigroup
	by
	\begin{equation*}\label{2-4}
		T^+_t\varphi(x)=\sup_{\gamma}\left\{\varphi(\gamma(t))-\int_0^tL\big(\gamma(\tau),\dot{\gamma}(\tau),T^+_{t-\tau}\varphi(\gamma(\tau))\big)d\tau\right\},
	\end{equation*}
	where the supremum is taken among the absolutely continuous curves $\gamma:[0,t]\to M$  with $\gamma(0)=x$.

	\begin{proposition}(\cite[Proposition 4.2, 4.3, Corollary 4.1]{WWY192})\label{prop1}
		Let $\varphi$, $\psi\in C(M,\R)$.
		\item [(1)] If $\psi\leqslant\varphi$, then $T^{\pm}_t\psi\leqslant T^{\pm}_t\varphi,\ \forall t\geqslant0$. Moreover, if $\psi<\varphi$, then $T^{\pm}_t\psi<T^{\pm}_t\varphi$.
		\item [(2)] The function $(x,t)\to T^{\pm}_t\varphi(x)$ is locally Lipschitz on $M\times(0,+\infty)$.
		\item [(3)] $\|T^{\pm}_t\varphi-T^{\pm}_t\psi\|_{\infty}\leqslant e^{\lambda t}\cdot\|\varphi-\psi\|_{\infty},\ \forall t\geqslant0$.
		\item [(4)] For each $\varphi\in C(M,\R)$,
		
		(i)  $T^-_t\varphi(x)=\inf_{y\in M}h_{y,\varphi(y)}(x,t),\ \forall(x,t)\in M\times(0,+\infty)$;
		
		(ii) $T^+_t\varphi(x)=\sup_{y\in M}h^{y,\varphi(y)}(x,t),\ \forall(x,t)\in M\times(0,+\infty)$.
		\item [(5)] $\{T^{\pm}_t\}_{t\geqslant0}$ are one-parameter semigroups of operators. For all $x_0$,  $x\in M$, all $u_0\in \R$ and all $s$, $t>0$,
		
		(i) $T^-_sh_{x_0,u_0}(x,t)=h_{x_0,u_0}(x,t+s)$,\quad  $T^-_{t+s}\varphi(x)=\inf_{y\in M}h_{y,T^-_s\varphi(y)}(x,t)$;
		
		(ii) $T^+_sh^{x_0,u_0}(x,t)=h^{x_0,u_0}(x,t+s)$,\quad   $T^+_{t+s}\varphi(x)=\sup_{y\in M}h^{y,T^+_s\varphi(y)}(x,t)$.
	\end{proposition}

	Proposition \ref{prop2} and Corollary \ref{prop10} come from \cite{WY21} and we give new brief proofs here.

	\begin{proposition}(\cite{WY21})\label{prop2}
		Let $\varphi\in C(M,\R)$. Then
		\item [(1)] $T^-_t\circ T^+_t\varphi(x)\geqslant\varphi(x)$, for all $t>0,\ x\in M$;
		\item [(2)] $T^+_t\circ T^-_t\varphi(x)\leqslant\varphi(x)$, for all $t>0,\ x\in M$.
	\end{proposition}
	
	\begin{proof}
		For any $x\in M$ and $t>0$, we know from Proposition \ref{prop1} (4) that
		$$T^-_t\varphi(y)=\inf_{z\in M}h_{z,\varphi(z)}(y,t)\leqslant h_{x,\varphi(x)}(y,t),\quad T^+_t\varphi(y)=\sup_{z\in M}h^{z,\varphi(z)}(y,t)\geqslant h^{x,\varphi(x)}(y,t),$$
		for all $y\in M$. Thus from Proposition \ref{addprop1}, Proposition \ref{prop3} (3) and Proposition \ref{prop8} (3), we get that
		$$T^-_t\circ T^+_t\varphi(x)=\inf_{y\in M}h_{y,T^+_t\varphi(y)}(x,t)\geqslant\inf_{y\in M}h_{y,h^{x,\varphi(x)}(y,t)}(x,t)=\varphi(x),$$
		and
		\[T^+_t\circ T^-_t\varphi(x)=\sup_{y\in M}h^{y,T^-_t\varphi(y)}(x,t)\leqslant\sup_{y\in M}h^{y,h_{x,\varphi(x)}(y,t)}(x,t)=\varphi(x).\]
	\end{proof}

	\begin{corollary}(\cite{WY21})\label{prop10}
		Let $u$, $v\in C(M,\R)$ and let $t\geqslant0$. Then $v\leqslant T^-_tu$ if and only if $T^+_tv\leqslant u$.
	\end{corollary}
	
	\begin{proof}
		If $v\leqslant T^-_tu$, from Proposition \ref{prop1} (1) and Proposition \ref{prop2} (2), we have $T^+_tv\leqslant T^+_t\circ T^-_tu\leqslant u.$
		On the other hand, if $u\geqslant T^+_tv$, from Proposition \ref{prop1} (1) and Proposition \ref{prop2} (1), we have $T^-_tu\geqslant T^-_t\circ T^+_tv\geqslant v.$
	\end{proof}

	\begin{proposition}\label{prop5.2}
		Let $\varphi\in C(M,\R)$. Then
		\item [(1)] If the function $(x,t)\to T^+_t\varphi(x)$ is unbounded from below on $M\times[0,+\infty)$, then for any $Q\in \R$, there is a constant $s\in[0,+\infty)$ such that $T^+_s\varphi(x)\leqslant Q,\ \forall x\in M$.
		\item [(2)] If the function $(x,t)\to T^-_t\varphi(x)$ is unbounded from above on $M\times[0,+\infty)$, then for any $Q\in \R$, there is a constant $s\in[0,+\infty)$ such that $T^-_s\varphi(x)\geqslant Q,\ \forall x\in M$.
	\end{proposition}
	
	\begin{proof}
		We only prove the first item, since the second one can be proved in a similar manner.
		
		From the assumption, we can find a sequence $\{(x_n,t_n)\}\subset M\times[0,+\infty)$ satisfying $t_n\to+\infty$ and
		\begin{equation}\label{eq5.7}
			T^+_{t_n}\varphi(x_n)\to-\infty,\quad  n\to+\infty.
		\end{equation}
		To prove the result, we argue by contradiction. For, otherwise, there would be $Q\in \R$ such that for any $t\in[0,+\infty)$, $T^+_t\varphi(y_t)>Q$ for some $y_t\in M$. Thus we can select a sequence $\{y_n\}\subset M$ corresponding to $\{t_n-1\}$ such that
		$$T^+_{t_n-1}\varphi(y_n)>Q,\quad n\in \mathbf{N}.$$
		By the monotonicity of the function $u_0\to h^{x_0,u_0}(x,1)$ and Proposition \ref{prop1} (5) (ii), we have
		\begin{equation*}
			T^+_{t_n}\varphi(x_n)=\sup_{x\in M}h^{x,T^+_{t_n-1}\varphi(x)}(x_n,1)
			\geqslant h^{y_n,T^+_{t_n-1}\varphi(y_n)}(x_n,1)
			\geqslant h^{y_n,Q}(x_n,1).
		\end{equation*}
		Since the function $(x,y)\to h^{y,Q}(x,1)$ is bounded on $M\times M$, then  $T^+_{t_n}\varphi(x_n)$ is bounded from below which contradicts \eqref{eq5.7}.
	\end{proof}

	The following two propositions show the relationship between the minimizer of semigroups and the minimizer of implicit action functions.

	\begin{proposition}\label{propA1}
		Let $\varphi\in C(M,\R)$ and $t>0$.  Let $\gamma:[-t,0]\to M$ be a minimizer of $T^-_t\varphi(x)$ with  $\gamma(0)=x$ and $\gamma(-t)=y$. Then
		\begin{itemize}
			\item [(1)] $\gamma|_{[-t,-s]}$ is a minimizer of $T^-_{t-s}\varphi(\gamma(-s))$, $\forall s\in[0,t]$.
			\item [(2)] $\gamma|_{[-s,0]}$ is a minimizer of $T^-_s(T^-_{t-s}\varphi)(x)$, $\forall s\in[0,t]$.
			\item [(3)] $T^-_t\varphi(x)=h_{y,\varphi(y)}(x,t)$.
			\item [(4)] $\gamma$ is a minimizer of $h_{y,\varphi(y)}(x,t)$.
		\end{itemize}
	\end{proposition}
	
	\begin{proof}
		(1). Since $\g$  is a minimizer of $T^-_t\varphi(x)$, then we get that
		\begin{equation}\label{eqA1}
			\begin{split}
				T^-_t\varphi(x)=&\inf_{\bar{\gamma}(0)=x}\Big\{\varphi(\bar{\gamma}(-t))+\int^0_{-t}L\big(\bar{\gamma}(\tau),\dot{\bar{\gamma}}(\tau),T^-_{\tau+t}\varphi(\bar{\gamma}(\tau))\big)d\tau\Big\}\\
				=&\varphi(\gamma(-t))+\big(\int^{-s}_{-t}+\int^0_{-s}\big)L\big(\gamma(\tau),\dot{\gamma}(\tau),T^-_{\tau+t}\varphi(\gamma(\tau))\big)d\tau.
			\end{split}
		\end{equation}
		By definition, for any $s\in[0,t]$, we have
		\begin{equation}\label{eqA2}
			\begin{split}
				T^-_s(T^-_{t-s}\varphi)(x)=&\inf_{\bar{\gamma}(0)=x}\Big\{T^-_{t-s}\varphi(\bar{\gamma}(-s))+\int^0_{-s}L\big(\bar{\gamma}(\tau),\dot{\bar{\gamma}}(\tau),T^-_{\tau+s}(T^-_{t-s}\varphi(\bar{\gamma}(\tau)))\big)d\tau\Big\}\\
				\leqslant&T^-_{t-s}\varphi(\gamma(-s))+\int^0_{-s}L\big(\gamma(\tau),\dot{\gamma}(\tau),T^-_{\tau+s}(T^-_{t-s}\varphi(\gamma(\tau)))\big)d\tau.
			\end{split}
		\end{equation}
		In view of $T^-_t\varphi(x)=T^-_{\tau}\circ T^-_{t-\tau}\varphi(x)$ for all $\tau\in[0,t]$, we get that
		$$\varphi(\gamma(-t))+\int^{-s}_{-t}L\big(\gamma(\tau),\dot{\gamma}(\tau),T^-_{\tau+t}\varphi(\gamma(\tau))\big)d\tau\leqslant T^-_{t-s}\varphi(\gamma(-s)).$$
		From the definition of $T^-_{t-s}\varphi(\gamma(-s))$, we get
		$$T^-_{t-s}\varphi(\gamma(-s))=\inf_{\bar{\gamma}(-s)=\gamma(-s)}\Big\{\varphi(\bar{\gamma}(-t))+\int^{-s}_{-t}L\big(\bar{\gamma}(\tau),\dot{\bar{\gamma}}(\tau),T^-_{\tau+t}\varphi(\bar{\gamma}(\tau))\big)d\tau\Big\}.$$
		Thus, $\gamma|_{[-t,-s]}$ is a minimizer of $T^-_{t-s}\varphi(\gamma(-s))$.
		
		(2). By (1), we have
		$$T^-_{t-s}\varphi(\gamma(-s))=\varphi(\gamma(-t))+\int^{-s}_{-t}L\big(\gamma(\tau),\dot{\gamma}(\tau),T^-_{\tau+t}\varphi(\gamma(\tau))\big)d\tau.$$
		In view of \eqref{eqA1} and \eqref{eqA2}, one can deduce that
		\begin{equation*}
			T^-_s(T^-_{t-s}\varphi)(x)
			\leqslant T^-_{t-s}\varphi(\gamma(-s))+\int^0_{-s}L\big(\gamma(\tau),\dot{\gamma}(\tau),T^-_{\tau+t}\varphi(\gamma(\tau))\big)d\tau=T^-_t\varphi(x)
		\end{equation*}
		which together with $T^-_s(T^-_{t-s}\varphi)(x)=T^-_t\varphi(x)$ implies $\gamma|_{[-s,0]}$ is a minimizer of $T^-_s(T^-_{t-s}\varphi)(x)$.

		(3). By Proposition \ref{prop1} (4), we have
		$$T^-_t\varphi(x)=\inf_{z\in M}h_{z,\varphi(z)}(x,t)\leqslant h_{y,\varphi(y)}(x,t).$$
		So it suffices to prove $T^-_t\varphi(x)\geqslant h_{y,\varphi(y)}(x,t).$
		Assume by contradiction that $T^-_t\varphi(x)<h_{y,\varphi(y)}(x,t)$. Let $u(s)=T^-_s\varphi(\gamma(s-t))$, $h(s)=h_{y,\varphi(y)}(\gamma(s-t),s)$ and $F(s)=h(s)-u(s)$. Note that $u(0)=T^-_0\varphi(\gamma(-t))=\varphi(y)$. From Lemma 3.1 and Lemma 3.2 in \cite{WWY17}, we get $h(0)=\varphi(y)$. Then $F(0)=0$ and $F(t)>0$. There is $s_0\in[0,t)$ such that $F(s_0)=0$ and $F(s)>0$ for $s\in(s_0,t]$. In view of (1),  $\gamma|_{[-t,s-t]}$ is a minimizer of $T^-_{s}\varphi(\gamma(s-t))$. By (2), we deduce that $\gamma|_{[s_0-t,s-t]}$ is a minimizer of $T^-_{s-s_0}(T^-_{s_0}\varphi)(\gamma(s-t))$. Therefore, for any $s\in(s_0,t]$, we have
		\begin{equation*}
			\begin{split}
				&u(s)=T^-_{s}\varphi(\gamma(s-t))=T^-_{s-s_0}(T^-_{s_0}\varphi)(\gamma(s-t))\\
				=&T^-_{s_0}\varphi(\gamma(s_0-t))+\int^{s}_{s_0}L\big(\gamma(\tau-t),\dot{\gamma}(\tau-t),T^-_{\tau}\varphi(\gamma(\tau-t))\big)d\tau\\
				=&u(s_0)+\int^{s}_{s_0}L\big(\gamma(\tau-t),\dot{\gamma}(\tau-t),u(\tau)\big)d\tau,
			\end{split}
		\end{equation*}
		and
		\begin{equation*}
			\begin{split}
				h(s)=&h_{y,\varphi(y)}(\gamma(s-t),s)=\varphi(y)+\inf_{\substack{\bar{\gamma}(0)=y \\  \bar{\gamma}(s)=\gamma(s-t)} }\int_0^sL\big(\bar{\gamma}(\tau),\dot{\bar{\gamma}}(\tau),h_{y,\varphi(y)}(\bar{\gamma}(\tau),\tau)\big)d\tau \\
				\leqslant&\varphi(y)+\big(\inf_{\substack{\bar{\gamma}(0)=y \\  \bar{\gamma}(s_0)=\gamma(s_0-t)} }\int_0^{s_0}+\inf_{\substack{\bar{\gamma}(s_0)=\gamma(s_0-t) \\  \bar{\gamma}(s)=\gamma(s-t)} }\int_0^{s}\big)L\big(\bar{\gamma}(\tau),\dot{\bar{\gamma}}(\tau),h_{y,\varphi(y)}(\bar{\gamma}(\tau),\tau)\big)d\tau \\
				\leqslant&h_{y,\varphi(y)}(\gamma(s_0-t),s_0)+\int^{s}_{s_0}L\big(\gamma(\tau-t),\dot{\gamma}(\tau-t),h_{y,\varphi(y)}(\gamma(\tau-t),\tau)\big)d\tau\\
				=&h(s_0)+\int^{s}_{s_0}L\big(\gamma(\tau-t),\dot{\gamma}(\tau-t),h(\tau)\big)d\tau.
			\end{split}
		\end{equation*}
		Hence, we get that
		\begin{equation*}
			F(s)=h(s)-u(s)
			\leqslant\int^{s}_{s_0}\Big|\frac{\partial L}{\partial u}\Big|\cdot(h(\tau)-u(\tau)) d\tau
			\leqslant\int^{s}_{s_0}\lambda F(\tau)d\tau
		\end{equation*}
		which together with Gronwall inequality implies $F(t)\leqslant0$. It contradicts $F(t)>0$.
		
		(4). From (1) and  (3), we have
		\begin{equation}\label{eqA3}
			T^-_{t-s}\varphi(\gamma(-s))=h_{y,\varphi(y)}(\gamma(-s),t-s).
		\end{equation}
		By (2) and (3) again, we have
		\begin{equation}\label{eqA4}
			T^-_s(T^-_{t-s}\varphi)(x)=h_{\gamma(-s),T^-_{t-s}\varphi(\gamma(-s))}(x,s).
		\end{equation}
		In view of \eqref{eqA3}, \eqref{eqA4} and (3), for all $s\in[0,t]$, we get
		\begin{align*}
			h_{y,\varphi(y)}(x,t)
			&=T^-_t\varphi(x)=T^-_s(T^-_{t-s}\varphi)(x)
			=h_{\gamma(-s),T^-_{t-s}\varphi(\gamma(-s))}(x,s)\\
			&=h_{\gamma(-s),h_{y,\varphi(y)}(\gamma(-s),t-s)}(x,s).
		\end{align*}
		From Proposition \ref{prop3} (1), we deduce that $\gamma$ is a minimizer of $h_{y,\varphi(y)}(x,t)$.
	\end{proof}

	\begin{proposition}\label{propA2}
		If $\gamma:[-t,0]\to M$ with $\gamma(0)=x$,  $\gamma(-t)=y$ is a minimizer of $h_{y,\varphi(y)}(x,t)=\inf_{z\in M}h_{z,\varphi(z)}(x,t)$, then $\gamma$ is a minimizer of $T^-_t\varphi(x)$.
	\end{proposition}
	
	\begin{proof}
		Since $\g$ is a minimizer of $h_{y,\varphi(y)}(x,t)$, then by the method of substitution,
		$$T^-_t\varphi(x)=h_{y,\varphi(y)}(x,t)=\varphi(\gamma(-t))+\int^0_{-t}L\big(\gamma(\tau),\dot{\gamma}(\tau),h_{y,\varphi(y)}(\gamma(\tau),t+\tau)\big)d\tau.$$
		Since $T^-_{\tau+t}\varphi(\gamma(\tau))=\inf_{z\in M}h_{z,\varphi(z)}(\gamma(\tau),t+\tau)$ for all $\tau\in[-t,0]$, it is sufficient to show
		\[
		h_{y,\varphi(y)}(\gamma(\tau),t+\tau)=\inf_{z\in M}h_{z,\varphi(z)}(\gamma(\tau),t+\tau), \quad\forall \tau\in[-t,0].
		\]
		If $h_{y,\varphi(y)}(\gamma(\tau),t+\tau)>\inf_{z\in M}h_{z,\varphi(z)}(\gamma(\tau),t+\tau)=h_{y_0,\varphi(y_0)}(\gamma(\tau),t+\tau)$ for some $y_0\in M$, then, by Proposition \ref{prop3} (1) and (3),
		\begin{equation*}
			\begin{split}
				h_{y,\varphi(y)}(x,t)
				=&h_{\gamma(\tau),h_{y,\varphi(y)}(\gamma(\tau),t+\tau)}(x,-\tau)\\
				>&h_{\gamma(\tau),h_{y_0,\varphi(y_0)}(\gamma(\tau),t+\tau)}(x,-\tau)\\
				\geqslant&\inf_{z\in M}h_{z,h_{y_0,\varphi(y_0)}(z,t+\tau)}(x,-\tau)\\
				=&h_{y_0,\varphi(y_0)}(x,t).
			\end{split}
		\end{equation*}
		It contradicts the assumption $h_{y,\varphi(y)}(x,t)=\inf_{z\in M}h_{z,\varphi(z)}(x,t)$.
	\end{proof}

	\section{Mather and weak KAM theories for contact Hamiltonian systems}
	
	We will establish Aubry-Mather and weak KAM type results for contact Hamiltonian systems under assumptions (H1)-(H3).
	The authors of \cite{WWY19} generalized part of Aubry-Mather and weak KAM theories from Hamiltonian systems to contact Hamiltonian systems under the assumptions (H1), (H2) and
	\begin{align}\label{incre}
		0<\frac{\partial H}{\partial u}\leqslant \lambda.
	\end{align}
	It should be noted that some results obtained in \cite{WWY19} are independent of the monotonicity \eqref{incre} of $H$. When we use these results, if the proofs given in \cite{WWY19} depend on \eqref{incre}, we give a new proof here, otherwise we will use them directly without proof.

	We  will also provide  a series of new results in this part, where  the definitions of the Ma\~n\'e set and Mather measures for \eqref{c} will be given. These definitions are independent of viscosity solutions (or equivalently, backward weak KAM solutions). We will clarify the relation between Mather measures associated with a given viscosity solution defined in the Introduction and the ones provided in Section 3.2. This is quite important for understanding Theorem \ref{th3}.


	\subsection{The notions of the Ma\~n\'e set and Mather measures}
	
	\begin{definition}
		A curve $(x(\cdot),u(\cdot)):\mathbf{R}\to M\times\mathbf{R}$ is called globally minimizing, if it is locally Lipschitz and
		for each $t_1$, $t_2\in\mathbf{R}$ with $t_1< t_2$, there holds
		\begin{align}\label{3-1}
			u(t_2)=h_{x(t_1),u(t_1)}(x(t_2),t_2-t_1).
		\end{align}
		A curve $(x(\cdot),u(\cdot)):\mathbf{R}\to M\times\mathbf{R}$ is called semi-static, if it is globally minimizing and for each $t_1$, $t_2\in\mathbf{R}$ with $t_1\leqslant t_2$, there holds
		\begin{equation}\label{3-3}
			u(t_2)=\inf_{s>0}h_{x(t_1),u(t_1)}(x(t_2),s).
		\end{equation}	
	\end{definition}
	
	\begin{remark}
		The original notion of semi-static curves comes from Ma\~n\'e's work on the Mather theory. See \cite{Man92,Man96,Man97} for more details.
	\end{remark}

	The relation between globally minimizing curves and system \eqref{c} is stated as follows.
	\begin{proposition}(\cite[Proposition 3.1]{WWY19})\label{pr22}
		If a curve $(x(\cdot),u(\cdot)):\mathbf{R}\to M\times\mathbf{R}$ is  globally minimizing, then $x(t)$	is of class $C^1$. Let $p(t):=\frac{\partial L}{\partial \dot{x}}\big(x(t),\dot{x}(t),u(t)\big)$. Then $\big(x(t),p(t),u(t)\big)$ is a solution of equations (\ref{c}).  Moreover, for each $t_1$, $t_2\in\mathbf{R}$ with $t_1<t_2$, $\bar{x}(t):=x(t+t_1)$ for $t\in[0,t_2-t_1]$, achieves the  minimum in $h_{x(t_1),u(t_1)}(x(t_2),t_2-t_1)$.
	\end{proposition}

	If a curve $(x(\cdot),u(\cdot)):\mathbf{R}\to M\times\mathbf{R}$ is semi-static, then by Proposition \ref{pr22}, $\big(x(t),p(t),u(t)\big)$ is an orbit of $\Phi^H_t$, where $p(t)=\frac{\partial L}{\partial \dot{x}}\big(x(t),\dot{x}(t),u(t)\big)$. We call it  a semi-static orbit of $\Phi^H_t$.

	\begin{proposition}\label{ssb}
		Each semi-static orbit of $\Phi^H_t$ is bounded.
	\end{proposition}
	
	\begin{proof}
		Let $(x(t),p(t),u(t))$ be a semi-static orbit. We aim to prove that both $\|p(t)\|$ and $u(t)$ are bounded in $\R$.
		
		First, we claim that
		\begin{align}\label{3-555}
			u(t)\leqslant h_{x(s),u(s)}(x(t),1),\quad  \forall s, t\in\R.
		\end{align}
		Indeed, for the case $t\geqslant s$, by the definition of semi-static curve we have
		$$u(t)=\inf_{\sigma>0}h_{x(s),u(s)}(x(t),\sigma)\leqslant h_{x(s),u(s)}(x(t),1).$$
		For the case of $t<s$, by the definition of globally minimizing curve we get that  $u(s)=h_{x(t),u(t)}(x(s),s-t)$. Thus from Proposition \ref{prop3} (1), we have
		\begin{equation*}
			\begin{split}
				u(t)&=\inf_{\sigma>0}h_{x(t),u(t)}(x(t),\sigma)\\
				&\leqslant h_{x(t),u(t)}(x(t),1+s-t)\\
				&= \inf_{z\in M}h_{z,h_{x(t),u(t)}(z,s-t)}(x(t),1)\\
				&\leqslant h_{x(s),h_{x(t),u(t)}(x(s),s-t)}(x(t),1)\\
				&= h_{x(s),u(s)}(x(t),1).
			\end{split}
		\end{equation*}
		So far, we have proved \eqref{3-555}. By  \eqref{3-555} we deduce that
		$$u(t)\leqslant h_{x_0,u_0}(x(t),1),\quad u_0\leqslant h_{x(t),u(t)}(x_0,1),\quad \forall t\in\R,$$
		where $x_0=x(0)$ and $u_0=u(0)$.
		Using Proposition \ref{addprop1}, we get that
		$$h^{x_0,u_0}(x(t),1)\leqslant u(t)\leqslant h_{x_0,u_0}(x(t),1), \quad \forall t\in\R.$$
		According to the continuity of functions $x\to h_{x_0,u_0}(x,1)$ and $x\to h^{x_0,u_0}(x,1)$ on $M$,  $u(t)$ is bounded on $t\in\R$.

		Next, we show the boundedness of $\|p(t)\|$.
		Let $K$ be a positive number such that $|u(t)|\leqslant K$ for all $t\in \R$. By Proposition \ref{pr22}, for any $t\in\R$, $x(s)|_{[t-1,t+1]}$ is a minimizer of $h_{x(t-1),u(t-1)}(x(t+1),2)$. Notice that $\big(x(t-1),u(t-1),x(t+1),2\big)\in M\times[-K,K]\times M\times\{2\}$. And thus from Proposition \ref{prop5}, we can deduce that  $\|p(t)\|$ is bounded by a constant depending only on $K$.
	\end{proof}

	Now we are ready for giving the definitions of the Ma\~n\'e set and Mather measures.

	\begin{definition}\label{audeine}
		The Ma\~n\'e set of $H$ is defined by
		\[
		\tilde{\mathcal{N}}:=\{(x,p,u): \big(x(t),p(t),u(t)\big):=\Phi_t^H(x,p,u)\text{ is a semi-static orbit}\}.
		\]
	\end{definition}
	
	\begin{definition}\label{audeine1}	
		Define
		\[
		\mathfrak{M}:=\{\mu\in\mathcal{P}(\T\times\R): \mu \ \text{is}\ \Phi^H_t\text{-invariant with}\ \supp\mu\subset \tilde{\mathcal{N}}\}.
		\]
		A measure $\mu\in \mathfrak{M}$ is called a Mather measure.
	\end{definition}	
	We will show that $\tilde{\mathcal{N}}$ and $\mathfrak{M}$ are nonempty later in this section.

	\subsection{Weak KAM solutions}
	\begin{definition}\label{bwkam}
		A function $u\in C(M,\R)$ is called a backward (resp. forward) weak KAM solution of \eqref{1-2}, if
		\begin{itemize}
			\item [(1)] for each continuous piecewise $C^1$ curve $\gamma:[t_1,t_2]\rightarrow M$, we have
			\begin{align}\label{do}
				u(\gamma(t_2))-u(\gamma(t_1))\leqslant\int_{t_1}^{t_2}L\big(\gamma(s),\dot{\gamma}(s),u(\gamma(s))\big)ds;
			\end{align}
			\item [(2)] for each $x\in M$, there exists a $C^1$ curve $\gamma:(-\infty,0]\rightarrow M$ (resp. $\gamma:[0,+\infty)\rightarrow M$) with $\gamma(0)=x$ such that for any $t_1,t_2\in(-\infty,0]$ (resp. $t_1,t_2\in[0,+\infty)$)
			\begin{align}\label{cali1}
				u(\gamma(t_2))-u(\gamma(t_1))=\int^{t_2}_{t_1}L\big(\gamma(s),\dot{\gamma}(s),u(\gamma(s))\big)ds, \quad \forall t_1<t_2.
			\end{align}
		\end{itemize}
		We say that $u$ in \eqref{do} is dominated by $L$, denoted by $u\prec L$. We call curves $\gamma:[a,b]\to M$ ($-\infty\leqslant a<b\leqslant+\infty$) satisfying \eqref{cali1} (u, L, 0)-calibrated curves. We use $\mathcal{S}^-$ (resp. $\mathcal{S}^+$) to denote the set of all backward (resp. forward) weak KAM solutions.
	\end{definition}
	\begin{remark}\label{ll}
		By the definition of dominiated functions, it is not hard to show any dominiated function  is Lipschitz continuous on $M$. See \cite[Lemma 4.1]{WWY19} for the proof. Hence, weak KAM solutions are Lebesgue almost everywhere differentiable.
	\end{remark}

	\subsubsection{Weak KAM solutions and solution semigroups}
	The proposition below points out that weak KAM solutions are just fixed points of solution semigroups, and that viscosity solutions are the same as backward weak KAM solutions in the setting of this paper.
	\begin{proposition}(\cite[Proposition 2.7]{WWY19})\label{prop9}
		The backward weak KAM solutions of \eqref{1-2} are the same as the viscosity solutions of \eqref{1-2}. Moreover,
		\begin{itemize}
			\item [(1)] $u_-\in \mathcal{S}^-$ if and only if $T^-_tu_-=u_-$ for all $t\geqslant0$.
			\item [(2)] $u_+\in \mathcal{S}^+$  if and only if $T^+_tu_+=u_+$ for all $t\geqslant0$.
		\end{itemize}
	\end{proposition}

	Next result is new, which is useful in the proof of Theorem \ref{th3}.

	\begin{proposition}\label{prop5.1}
		Let $\varphi\in C(M,\R)$. Then
		\item [(1)] If the function $(x,t)\to T^+_t\varphi(x)$ is bounded on $M\times[0,+\infty)$, then $\varphi_{\infty}:=\limsup_{t\to+\infty}T^+_t\varphi(x)$ belongs to $\mathcal{S}^+$ and $\lim_{t\to+\infty}\sup_{s\geqslant t}T^+_s\varphi(x)=\varphi_{\infty}(x)$ uniformly on $x\in M$.
		\item [(2)] If the function $(x,t)\to T^-_t\varphi(x)$ is bounded on $M\times[0,+\infty)$, then $\varphi_{\infty}:=\liminf_{t\to+\infty}T^-_t\varphi(x)$ belongs to $\mathcal{S}^-$ and $\lim_{t\to+\infty}\inf_{s\geqslant t}T^-_s\varphi(x)=\varphi_{\infty}(x)$ uniformly on $x\in M$.
	\end{proposition}

	\begin{proof}
		We only prove the first item, since the second one can be proved in a similar manner.	
		
		Let $P$ be a positive constant such that
		\begin{equation}\label{eq5.1}
			|T^+_t\varphi(x)|\leqslant P,\quad \forall (x,t)\in M\times[0,+\infty).
		\end{equation}
		By Proposition \ref{prop8} (2), the function $(x_0,u_0,x)\to h^{x_0,u_0}(x,1)$ is Lipschitz on $M\times[-P,P]\times M$. Denote $l_1>0$ the Lipschitz constant of the function.
		
		First we show that $\{T^+_t\varphi(x)\}_{t>1}$ is equi-Lipschitz on $M$. From Proposition \ref{prop1} (5) (ii), we have
		$$|T^+_t\varphi(x)-T^+_t\varphi(y)|\leqslant\sup_{z\in M}\Big|h^{z,T^+_{t-1}\varphi(z)}(x,1)-h^{z,T^+_{t-1}\varphi(z)}(y,1)\Big|$$
		for all $x,y\in M$, and all $t>1$. In view of \eqref{eq5.1}, the above inequality implies that
		$$|T^+_t\varphi(x)-T^+_t\varphi(y)|\leqslant l_1\cdot d(x,y),\quad \forall t>1.$$
		
		Let $\varphi_{\infty}(x):=\limsup_{t\to+\infty}T^+_t\varphi(x)$. We show that $\varphi_{\infty}$ is a fixed point of $\{T^+_t\}_{t\geqslant0}$. Since $\{T^+_t\varphi(x)\}_{t>1}$ is equi-Lipschitz on $M$, it is easy to see that
		$$\lim_{t\to+\infty}\sup_{s\geqslant t}T^+_s\varphi(x)=\varphi_{\infty}(x),\ \text{uniformly on}\ x\in M.$$
		Note that
		$$\varphi_{\infty}(x)=\limsup_{t\to+\infty}T^+_s\circ T^+_t\varphi(x),\ \forall s\geqslant0.$$
		By Proposition \ref{prop1} (1) and (3), we have
		\begin{equation*}
			\begin{split}
				\varphi_{\infty}(x)=&\lim_{m\to+\infty}\lim_{n\to+\infty}\max_{m\leqslant t\leqslant n}T^+_s\circ T^+_t\varphi(x)\\
				=&\lim_{m\to+\infty}\lim_{n\to+\infty}T^+_s(\max_{m\leqslant t\leqslant n}T^+_t\varphi)(x)\\
				=&T^+_s(\lim_{m\to+\infty}\lim_{n\to+\infty}\max_{m\leqslant t\leqslant n}T^+_t\varphi)(x)\\
				=&T^+_s\varphi_{\infty}(x).
			\end{split}
		\end{equation*}
	\end{proof}

	The following result comes from \cite{WY21}. Here we give an independent proof.
	\begin{proposition}(\cite{WY21})\label{prop6}\
		\begin{itemize}
			\item [(1)] For each $u\in \mathcal{S}^-$, the uniform limit $\lim_{t\to+\infty}T^+_tu=:v$ exists and $v\in \mathcal{S}^+$.
			\item [(2)] For each $v\in \mathcal{S}^+$, the uniform limit $\lim_{t\to+\infty}T^-_tv=:u$ exists and $u\in \mathcal{S}^-$.
		\end{itemize}
	\end{proposition}
	
	\begin{proof}
		We only prove the first item, since the second one can be proved in a similar manner.
		
		If $u\in\mathcal{S}^-$, then $u=T^-_tu,\ \forall t\geqslant0$ by Proposition \ref{prop9}. Using Corollary \ref{prop10}, we get that  $T^+_tu\leqslant u$, $\forall t\geqslant0$. According to Proposition \ref{prop1} (1), we have
		$$T^+_{t_1}u\leqslant T^+_{t_2}u,\quad \forall t_1\geqslant t_2\geqslant0.$$
		Hence by Proposition \ref{prop5.1} (1), to finish the proof, it suffices to show that the function $(x,t)\to T^+_tu(x)$ is bounded from below on $M\times[0,+\infty)$. Assume by contradiction that $T^+_tu(x)$ is unbounded from below on $M\times[0,+\infty)$. Let $Q=-\max_{x\in M}|u(x)|-1$. Then $Q<u$. Then from Proposition \ref{prop5.2} (1), we can find  $s\geqslant0$ such that $T^+_su\leqslant Q\ on\ M.$ And from Corollary \ref{prop10}, we get that
		$$u\leqslant T^-_sQ.$$
		On the other hand, we have $T^-_sQ< T^-_su=u$ from Proposition \ref{prop1} (1), which yields a contradiction.
	\end{proof}

	By Corollary \ref{prop10}, Proposition \ref{prop9} and Proposition \ref{prop6}, one can deduce that
	\begin{corollary}\label{uplus}\
		\begin{itemize}
			\item [(1)]Let $u_-\in\mathcal{S}^-$ and $u_+:=\lim_{t\to+\infty}T^+_tu_-\in\mathcal{S}^+$. Then $T_t^+u_-\leqslant u_-,\ \forall t\geqslant 0$ and $u_+\leqslant u_-$.
			\item [(2)]Let $u_+\in\mathcal{S}^+$ and $u_-:=\lim_{t\to+\infty}T^-_tu_+\in\mathcal{S}^-$. Then $T_t^-u_+\geqslant u_+,\ \forall t\geqslant 0$ and $u_-\geqslant u_+$.
		\end{itemize}
	\end{corollary}

	\subsubsection{Calibrated curves,  1-graphs of backward weak KAM solutions: $\Lambda_{u_-}$}
	
	Let $u_-\in\mathcal{S}^-$. Denote by $l$ the Lipschitz constant of $u_-$.
	\begin{proposition}\label{houjia}
		Let $u_-\in\mathcal{S}^-$ and let $\g:[a,b]\to M$ be a $(u_-,L,0)$-calibrated curve. Then $\gamma\in C^1((a,b),M)$ and $\|\dot{\gamma}(t)\|\leqslant C_{u_-,l}$, $\forall t\in(a,b)$, where $C_{u_-,l}$ depends only on $u_-$ and its Lipschitz constant $l$. Meanwhile, for any $t_1$,  $t_2\in[a,b]$ with $t_1<t_2$, we have
		\[
		u_-(\g(t_2))=h_{\g(t_1),u_-(\g(t_1))}(\g(t_2),t_2-t_1).
		\]
	\end{proposition}

	\begin{proof}
		From the definition of $(u_-,L,0)$-calibrated curve and Proposition \ref{prop9}, we get
		$$T^-_{t_2-t_1} u_-(\gamma(t_2))=u_-(\gamma(t_2))=u_-(\gamma(t_1))+\int^{t_2}_{t_1}L\big(\gamma(s),\dot{\gamma}(s),T^-_{s-t_1}u_-(\gamma(s))\big)ds.$$
		That means $\gamma|_{[t_1,t_2]}$ is a minimizer of $T^-_{t_2-t_1} u_-(\gamma(t_2))$. Thus $\gamma$ belongs to $C^1$. By Proposition \ref{propA1},  $\gamma|_{[t_1,t_2]}$ is a minimizer of $h_{\gamma(t_1),u_-(\gamma(t_1))}(\gamma(t_2),t_2-t_1)$ and
		$$u_-(\gamma(t_2))=T^-_{t_2-t_1}u_-(\gamma(t_2))=h_{\gamma(t_1),u_-(\gamma(t_1))}(\gamma(t_2),t_2-t_1).$$
		Moreover, by the definition of calibrated curve, we have
		\begin{equation}\label{eq1.111}
			l\cdot d(\gamma(t_1),\gamma(t_2))\geqslant u_-(\gamma(t_2))-u_-(\gamma(t_1))=\int^{t_2}_{t_1}L\big(\gamma(s),\dot{\gamma}(s),u_-(\gamma(s))\big)ds.
		\end{equation}
		In view of (L2), we get that
		\begin{equation}\label{eq1.222}
			L\big(\gamma,\dot{\gamma},u_-(\gamma)\big)\geqslant(l+1)\|\dot{\gamma}\|-C_{u_-,l}
		\end{equation}
		where the constant $C_{u_-,l}>0$ depends  on $u_-$ and $l$. By \eqref{eq1.111} and \eqref{eq1.222}, we deduce that
		$$d(\gamma(t_1),\gamma(t_2))\leqslant C_{u_-,l}\cdot(t_2-t_1),$$
		which shows that $\|\dot{\gamma}(t)\|\leqslant C_{u_-,l},\ \forall t\in(a,b)$.
	\end{proof}

	The following two lemmas are direct consequences of the definition of dominited functions and Remark \ref{ll}.
	
	\begin{lemma}(\cite[Lemma 4.2]{WWY19})\label{labell}
		Let $u\in \mathcal{S}^-$ and $\gamma:[a,b]\rightarrow M$ be a $(u,L,0)$-calibrated curve. If $u$ is differentiable at $\gamma(t)$ for some $t\in [a,b]$, then we have that
		\[
		H\big(\gamma(t),Du(\gamma(t)),u(\gamma(t))\big)=0,\qquad Du(\gamma(t))=\frac{\partial L}{\partial \dot{x}}\big(\gamma(t),\dot{\gamma}(t),u(\gamma(t))\big).
		\]
	\end{lemma}

	\begin{lemma}(\cite[Lemma 4.3]{WWY19})\label{diffe}
		Given any $a>0$, let $u\prec L$ and let $\gamma:[-a,a]\rightarrow M$ be a $(u,L,0)$-calibrated curve. Then $u$ is differentiable at $\gamma(0)$.
	\end{lemma}

	Next proposition is substantially the same as Corollary \ref{inva} below. Different expressions are given in Lagrangian and Hamiltnonian frameworks, respectively.
	
	\begin{proposition}\label{propA5}
		If $u_-\in \mathcal{S}^-$ and $\gamma:[a,b]\to M$ is a $(u_-,L,0)$-calibrated curve, then
		$$\Phi^L_t\big(\gamma(s),\dot{\gamma}(s),u_-(\gamma(s))\big)=\big(\gamma(s+t),\dot{\gamma}(s+t),u_-(\gamma(s+t))\big),$$
		for any $s\in\R,\ t\geqslant0$ satisfying $a\leqslant s\leqslant s+t\leqslant b$.
	\end{proposition}
	\begin{proof}
		For any $\tau\in(0,b-a)$, by Proposition \ref{houjia}, we have
		$$u_-(\gamma(s))=T^-_{s-b+\tau}u_-(\gamma(s))=h_{\gamma(b-\tau),u_-(\gamma(b-\tau))}(\gamma(s),s-b+\tau),\quad \forall s\in[b-\tau,b].$$
		Since $\gamma$ is a minimizer of $h_{\gamma(b-\tau),u_-(\gamma(b-\tau))}(\gamma(b),\tau)$, from Proposition \ref{prop4}, we deduce that
		$$\dot{\gamma}(s)=\frac{\partial H}{\partial p}\big(\gamma(s),p(s),u(s)\big),\quad \forall s\in[b-\tau,b],$$
		where $u(s)=h_{\gamma(b-\tau),u_-(\gamma(b-\tau))}(\gamma(s),s-b+\tau)=u_-(\gamma(s))$ and $p(s)=\frac{\partial L}{\partial \dot{x}}\big(\gamma(s),\dot{\gamma}(s),u(s)\big)$. From Lemma \ref{labell} and \ref{diffe}, we get that  $u_-$ is differentiable at $\gamma(s)$ for any  $s\in(b-\tau,b)$ and
		$$Du_-(\gamma(s))=\frac{\partial L}{\partial \dot{x}}\big(\gamma(s),\dot{\gamma}(s),u_-(\gamma(s))\big)=p(s),\quad \forall s\in(b-\tau,b).$$
		Thus
		\begin{equation}\label{eqA5}
			\dot{\gamma}(s)=\frac{\partial H}{\partial p}\big(\gamma(s),Du_-(\gamma(s)),u_-(\gamma(s))\big),\quad \forall s\in(b-\tau,b).
		\end{equation}
		Since $\Phi^L_t=\bar{\mathcal{L}}\circ\Phi^H_t\circ\bar{\mathcal{L}}^{-1}$, then by  \eqref{eqA5}, we have
		\begin{equation*}
			\begin{split}
				&\Phi^L_t\big(\gamma(s),\dot{\gamma}(s),u_-(\gamma(s))\big)\\
				=&\bar{\mathcal{L}}\circ\Phi^H_t\circ\bar{\mathcal{L}}^{-1}\big(\gamma(s),\dot{\gamma}(s),u_-(\gamma(s))\big)\\
				=&\bar{\mathcal{L}}\circ\Phi^H_t\circ\bar{\mathcal{L}}^{-1}\big(\gamma(s),\frac{\partial H}{\partial p}(\gamma(s),Du_-(\gamma(s)),u_-(\gamma(s))),u_-(\gamma(s))\big)\\
				=&\bar{\mathcal{L}}\circ\Phi^H_t\big(\gamma(s),Du_-(\gamma(s)),u_-(\gamma(s))\big)\\
				=&\bar{\mathcal{L}}\circ\Phi^H_t\big(\gamma(s),p(s),u(s)\big)\\
				=&\bar{\mathcal{L}}\big(\gamma(s+t),p(s+t),u(s+t)\big)\\
				=&\bar{\mathcal{L}}\big(\gamma(s+t),Du_-(\gamma(s+t)),u_-(\gamma(s+t))\big)\\
				=&\big(\gamma(s+t),\dot{\gamma}(s+t),u_-(\gamma(s+t))\big).
			\end{split}
		\end{equation*}
		for any $t\in\R^+$ satisfying $s+t\leqslant b$.
	\end{proof}

	Recall
	$$\Lambda_{u_-}:=\operatorname{cl}\Big(\big\{(x,p,u): x \ \text{is a point of differentiability of} \ u_{-}, p=Du_{-}(x),u=u_{-}(x)\big\}\Big).$$
	The next result is devoted to the relation between $(u_-,L,0)$-calibrated curves and $\Lambda_{u_-}$.
	
	\begin{corollary}\label{propA6}
		Let $u_-\in \mathcal{S}^-$. Let $\mathcal{C}$ be the set of all $(u_-,L,0)$-calibrated curves $\g:[-a,0]\to M$ with $0<a\leqslant+\infty$.  Define
		\[
		\tilde{\mathcal{C}}:=\Big\{\big(\gamma(t),\dot{\gamma}(t),u_-(\gamma(t))\big): \gamma\in\mathcal{C},\ t\in \mathrm{Dom}(\g)\Big\}.
		\]
		Then  $\mathrm{cl}(\tilde{\mathcal{C}})=\bar{\mathcal{L}}(\Lambda_{u_-})$.
	\end{corollary}
	\begin{proof}
		For any $\gamma\in\mathcal{C}$, by \eqref{eqA5} we get that
		$$\big(\gamma(s),\dot{\gamma}(s),u_-(\gamma(s))\big)=\bar{\mathcal{L}}\big(\gamma(s),Du_-(\gamma(s)),u_-(\gamma(s))\big).$$
		Thus, we have  $\tilde{\mathcal{C}}\subset\bar{\mathcal{L}}(\Lambda_{u_-})$, which implies that  $\mathrm{cl}(\tilde{\mathcal{C}})\subset\bar{\mathcal{L}}(\Lambda_{u_-})$ since $\Lambda_{u_-}$ is closed.

		On the other hand, for any point $(x,p,u)\in \Lambda_{u_-}$, there is a sequence $\{x_n\}\subset M$ such that  $u_-$ is differentiable at $x_n$ and
		$$\big(x_n,Du_-(x_n),u_-(x_n)\big)\longrightarrow(x,p,u),\quad   n\to+\infty.$$
		Due to $u_-\in \mathcal{S}^-$, we can find a $(u_-,L,0)$-calibrated curve $\gamma_n:[-a,0]\to M$ for each $x_n$ with $\gamma_n(0)=x_n$. From Lemma \ref{diffe}, we deduce that $u_-$ is differentiable at $\gamma_n(t)$ for any $t\in(-a,0)$. By Lemma \ref{labell}, for each $\gamma_n$ we can find a sequence $\{t^n_m\}_m$  such that
		$$t^n_m\to0,\quad m\to+\infty,\quad   Du_-(\gamma_n(t^n_m))=\frac{\partial L}{\partial \dot{x}}\big(\gamma_n(t^n_m),\dot{\gamma}_n(t^n_m),u_-(\gamma_n(t^n_m))\big).$$
		Since $u_-$ is differentiable at $x_n=\gamma_n(0)$, we get that
		$$\gamma_n(t^n_m)\to x_n,\quad Du_-(\gamma_n(t^n_m))\to Du_-(x_n),\quad m\to+\infty.$$
		Therefore, by the diagonal argument, one can get that
		$$\big(\gamma_{n_k}(t^{n_k}_{n_k}),Du_-(\gamma_{n_k}(t^{n_k}_{n_k})),u_-(\gamma_{n_k}(t^{n_k}_{n_k}))\big)\longrightarrow(x,p,u),\quad k\to+\infty.$$
		It implies that $\bar{\mathcal{L}}(x,p,u)\in \mathrm{cl}(\tilde{\mathcal{C}})$.
	\end{proof}

	The following two results are direct consequences of Proposition \ref{pr22}, Proposition \ref{houjia}, Proposition \ref{propA5}, Lemma \ref{labell} and Lemma \ref{diffe}.
	
	\begin{corollary}\label{inva}
		Let $u_-\in \mathcal{S}_-$.
		Let $\gamma:[a,b]\rightarrow M$ be a  $(u_-,L,0)$-calibrated curve. Then $u_-$ is differentiable at $\g(s)$ for all $s\in(a,b)$ and 	
		\[
		\Big(\gamma(s),\frac{\partial L}{\partial \dot{x}}(\g(s),\dot{\g}(s),u_-(\g(s))),u_-(\gamma(s))\Big)
		\]
		is a solution of (\ref{c}).
		Moreover,
		\[
		\big(\gamma(t+s),Du_-(\gamma(t+s)),u_-(\gamma(t+s))\big)=\Phi^H_{s}\big(\gamma(t),Du_-(\gamma(t)),u_-(\gamma(t))\big), \quad\forall t, \ t+s\in[a,b],
		\]
		and
		\[
		H\big(\gamma(s),Du_-(\g(s)),u_-(\gamma(s))\big)=0,\quad Du_-(\g(s))=\frac{\partial L}{\partial \dot{x}}\big(\gamma(s),\dot{\gamma}(s),u_-(\gamma(s))\big).
		\]
	\end{corollary}

	\begin{corollary}\label{inva1}
		Let $u_-\in \mathcal{S}_-$. Then $\Lambda_{u_-}$ is negatively invariant under $\Phi^H_t$.
	\end{corollary}		
	
	Define
	\[\tilde{\Sigma}_{u_-}:=\bigcap_{t\geq 0}\Phi^H_{-t}(\Lambda_{u_-})\quad \text{and} \quad \Sigma_{u_-}:=\pi\tilde{\Sigma}_{u_-},
	\]
	where $\pi:T^*M\times\R\rightarrow M$ denotes the orthogonal projection.
	By Corollary \ref{inva1}, $\Lambda_{u_-}$ is negatively invariant under $\Phi^H_t$.
	Thus, we have that
	\begin{align}\label{non}
		\Phi^H_{-t}(\Lambda_{u_-})\subset \Phi^H_{-s}(\Lambda_{u_-})\subset \Lambda_{u_-},\quad \forall t\geq s>0.
	\end{align}
	Since $u_-$ is Lipschitz continuous on $M$, then
	$ \Lambda_{u_-}$ is a nonempty compact subset of $T^*M\times\R$. In view of \eqref{non}, it follows from Cantor's intersection theorem that $\tilde{\Sigma}_{u_-}$ is non-empty and compact. Note that for $s>0$, we have
	\begin{align*}
		\Phi^H_s(\tilde{\Sigma}_{u_-})=\Phi^H_s\left(\bigcap_{t\geq 0}\Phi^H_{-t}(\Lambda_{u_-})\right)=\bigcap_{t\geq 0}\Phi^H_{-t+s}(\Lambda_{u_-})\subset \bigcap_{t\geq 0}\Phi^H_{-t}(\Lambda_{u_-})=\tilde{\Sigma}_{u_-},
	\end{align*}
	which implies that $\tilde{\Sigma}_{u_-}$ is invariant under $\Phi^H_t$.

	\subsection{The Ma\~n\'e set and Mather measures, revisited}	
	
	\subsubsection{The Ma\~n\'e set associated with a backward weak KAM solution:  $\tilde{\mathcal{N}}_{u_-}$}

	Suppose $u_-\in\mathcal{S}^-$. Then by Proposition \ref{prop6},
	\begin{align}\label{2-585}
		u_+:=\lim_{t\to+\infty}T^+_tu_-\in\mathcal{S}^+.
	\end{align}
	In view of Corollary \ref{uplus} and the $\Phi^H_t$-invariance of $\tilde{\Sigma}_{u_-}$, we can prove that
	\begin{lemma}(\cite[Lemma 4.5]{WWY19})\label{ufixplus}
		Let $u_-\in\mathcal{S}^-$ and let $u_+$ be as in \eqref{2-585}.  Then $u_-=u_+$ on $\Sigma$.
	\end{lemma}

	\begin{definition}
		Let $u_-\in\mathcal{S}^-$ and let $u_+$ be as in \eqref{2-585}. Define
		\[
		\mathcal{N}_{u_-}:=\{x\in M: u_-(x)=u_+(x)\}.
		\]
	\end{definition}	
	By Lemma  \ref{ufixplus}, we deduce that $\mathcal{N}_{u_-}$ is nonempty.

	\begin{proposition}(\cite[Lemma 4.7]{WWY19})\label{iinv}
		Let $u_-\in\mathcal{S}^-$ and let $u_+$ be as in \eqref{2-585}.
		For any given $x\in M$ with $u_-(x)=u_+(x)$,  there exists a curve $\gamma:(-\infty,+\infty)\rightarrow M$ with $\gamma(0)=x$ such that  $u_-(\gamma(t))=u_+(\gamma(t))$ for each $t\in \mathbf{R}$, and
		\begin{equation}\label{upmm}
			u_{\pm}(\gamma(t'))-u_{\pm}(\gamma(t))=\int_t^{t'}L\big(\gamma(s),\dot{\gamma}(s),u_{\pm}(\gamma(s))\big)ds, \quad \forall t\leq t'\in\mathbf{R}.
		\end{equation}
		Moreover, $u_{\pm}$ are differentiable at $x$ with the same derivative $Du_{\pm}(x)=\frac{\partial L}{\partial \dot{x}}\big(x,\dot{\gamma}(0),u_{\pm}(x)\big)$.
	\end{proposition}	
	\begin{remark}
		In \cite{WWY19} the above proposition was proved using the monotonicity assumption \eqref{incre}. A key point in the proof is the fact that calibrated curves admit the action minimizing property. We prove it in Proposition \ref{houjia} under (H1)-(H3). So the above result still holds without the monotonicity of $H$.
	\end{remark}

	\begin{definition}
		Let $u_-\in\mathcal{S}^-$ and let $u_+$ be as in \eqref{2-585}. Define
		\[
		\tilde{\mathcal{N}}_{u_-}:=\{(x,p,u)\in T^*M\times\R: x\in\mathcal{N}_{u_-}, \ p=Du_{\pm}(x),\ u=u_{\pm}(x)\}.
		\]
	\end{definition}	
	From Proposition \ref{iinv}, we know that $\tilde{\mathcal{N}}_{u_-}$ is well defined and nonempty.

	\subsubsection{$\tilde{\mathcal{N}}_{u_-}$ is a compact $\Phi_t^H$-invariant subset of the Ma\~n\'e set $\tilde{\mathcal{N}}$}

	\begin{lemma}\label{com}
		Let $u_-\in\mathcal{S}^-$. Then $\tilde{\mathcal{N}}_{u_-}$ is a compact subset of $T^*M\times\R$.
	\end{lemma}

	\begin{proof}
		Since $u_-$ is Lipschitz, then it is clear that  $\tilde{\mathcal{N}}_{u_-}$ is bounded. To finish the proof, it suffices to show that if $\tilde{\mathcal{N}}_{u_-}\ni(x_n,p_n,u_n)\to(x_0,p_0,u_0)$ as $n\to+\infty$, then $(x_0,p_0,u_0)\in \tilde{\mathcal{N}}_{u_-}$. Since $x_n\in\mathcal{N}_{u_-} $, then $u_-(x_n)=u_+(x_n)$. Thus $u_-(x_0)=u_+(x_0)$ implying that $x_0\in \mathcal{N}_{u_-}$. In view of $(x_n,p_n,u_n)\in\tilde{\mathcal{N}}_{u_-}$, $u_n=u_-(x_n)=u_+(x_n)$ and thus $u_0=u_-(x_0)=u_+(x_0)$. So, we only need to prove $p_0=Du_-(x_0)$.
		
		From Corollary \ref{inva} and Proposition \ref{iinv}, there exists a sequence of $(u_-,L,0)$-calibrated curves $\gamma_n:(-\infty,+\infty)\rightarrow M$ with $\gamma_n(0)=x_n$ such that $\big(\gamma_n(s),p_n(s),u_n(s)\big)$ are solutions of (\ref{c}), where $p_n(s)=\frac{\partial L}{\partial \dot{x}}\big(\g_n(s),\dot{\g}_n(s),u_-(\g_n(s))\big),\ u_n(s)=u_-(\g_n(s))$.
		
		Based on Proposition \ref{houjia}, $\{\dot{\g}_n\}$ and $\{p_n\}$ are uniformly bounded. Since
		\[
		\dot{\gamma}_n(s)=\frac{\partial H}{\partial p}\big(\gamma_n(s),p_n(s),u_-(\g_n(s))\big),
		\]
		then
		\begin{equation*}
			\begin{split}
				\ddot{\gamma}_n=&\frac{\partial^2H}{\partial p\partial x}(\gamma_n,p_n,u_n)\cdot\dot{\gamma}_n+\frac{\partial^2H}{\partial p\partial u}(\gamma_n,p_n,u_n)\cdot\dot{u}_n+\frac{\partial^2H}{\partial p^2}(\gamma_n,p_n,u_n)\cdot\dot{p}_n\\
				=&\Big\{\frac{\partial^2H}{\partial p\partial x}(\gamma_n,p_n,u_n)\cdot\frac{\partial H}{\partial p}(\gamma_n,p_n,u_n)+\frac{\partial^2H}{\partial p\partial u}(\gamma_n,p_n,u_n)\cdot\big(\frac{\partial H }{\partial p}(\gamma_n,p_n,u_n)\cdot p_n\\
				-&H(\gamma_n,p_n,u_n)\big)+\frac{\partial^2H}{\partial p^2}(\gamma_n,p_n,u_n)\cdot\big(-\frac{\partial H}{\partial x}(\gamma_n,p_n,u_n)-\frac{\partial H}{\partial u}(\gamma_n,p_n,u_n)\cdot p_n\big)\Big\}.
			\end{split}
		\end{equation*}
		Therefore, $\{\ddot{\g}_n\}$ are also uniformly bounded. Without any loss of generality, we may assume that $\g_n\to \g$ and $\dot{\g}_n\to \dot{\g}$ uniformly on $[-1,1]$, where the curve $\g$ can be verified as a $(u_-,L,0)$-calibrated curve with $\g(0)=x_0$ by \eqref{upmm}. Hence, $Du_-(x_0)=\frac{\partial L}{\partial \dot{x}}\big(\g(0),\dot{\g}(0),u_-(\g(0))\big)$. Recall that $p_n=Du_-(x_n)=\frac{\partial L}{\partial \dot{x}}\big(\g_n(0),\dot{\g}_n(0),u_-(\g_n(0))\big)$. Letting $n\to+\infty$, we get that $p_0=Du_-(x_0)$.
	\end{proof}

	\begin{lemma}\label{ss}
		Let $u_-\in\mathcal{S}^-$. Then $\tilde{\mathcal{N}}_{u_-}$ is $\Phi_t^H$-invariant. For any $(x_0,p_0,u_0)\in \tilde{\mathcal{N}}_{u_-}$, let $\big(x(t),p(t),u(t)\big):=\Phi_t^H(x_0,p_0,u_0)$. Then $\big(x(t),u(t)\big)$ is a semi-static curve.
	\end{lemma}
	
	\begin{proof}
		Since $(x_0,p_0,u_0)\in \tilde{\mathcal{N}}_{u_-}$, by Corollary \ref{inva} and Proposition \ref{iinv} there exists a curve $\gamma:(-\infty,+\infty)\rightarrow M$ with $\gamma(0)=x_0$ such that $\big(\g(t),Du_-(\g(t)),u_-(\g(t))\big)$ is a solution of \eqref{c} and $\g(0)=x_0$, $Du_-(\g(0))=Du_-(x_0)=p_0$, $u_-(\g(0))=u_-(x_0)=u_0$. From the uniqueness of
		solutions of initial value problem of ordinary differential equations,  we have $\big(x(t),p(t),u(t)\big)= \big(\g(t),Du_-(\g(t)),u_-(\g(t))\big)$, $\forall t\in\R$. Hence, $\big(x(t),p(t),u(t)\big)\in \tilde{\mathcal{N}}_{u_-},\ \forall t\in\R$.
		
		By Proposition \ref{iinv} agian, $\g$ is a $(u_-,L,0)$-calibrated curve. From Proposition \ref{houjia}, for any $t_1$,  $t_2\in\R$ with $t_1<t_2$,
		\[
		u_-(\g(t_2))=h_{\g(t_1),u_-(\g(t_1))}(\g(t_2),t_2-t_1).
		\]
		So $(x(t),u(t))=(\gamma(t),u_-(\gamma(t)))$ is globally minimizing. Note that
		\[
		u_-(\gamma(t_2))=T^-_su_-(\gamma(t_2))=\inf_{z\in M}h_{z,u_-(z)}(\gamma(t_2),s)\leqslant h_{\gamma(t_1),u_-(\gamma(t_1))}(\gamma(t_2),s),\quad \forall s>0.
		\]
		From above formulas, we can deduce that $(x(t),u(t))$ is a semi-static curve.
	\end{proof}

	By Lemma \ref{ss}, we have
	\begin{align}\label{ss1}
		\tilde{\mathcal{N}}_{u_-}\subset \tilde{\mathcal{N}}.
	\end{align}
	Thus, $\tilde{\mathcal{N}}$ is nonempty.

	\begin{proposition}\label{pr1011} Let $u_-\in\mathcal{S}^-$. Then
		\[
		\tilde{\mathcal{N}}_{u_-}=\tilde{\Sigma}_{u_-}=\bigcap_{t\geq 0}\Phi^H_{-t}(\Lambda_{u_-}).
		\]
	\end{proposition}
	\begin{proof}
		Since $\tilde{\mathcal{N}}_{u_-}\subset\Lambda_{u_-}$ and  $\tilde{\mathcal{N}}_{u_-}$ is $\Phi_t^H$-invariant, then $
		\tilde{\mathcal{N}}_{u_-}\subset\tilde{\Sigma}_{u_-}$.
		On the other hand, for any $(x,p,u)\in\tilde{\Sigma}_{u_-}$, by Lemma \ref{ufixplus}, we deduce that $u_-(x)=u_+(x)$. From Proposition \ref{iinv}, $u_-$ is differentiable at $x$. Notice that $(x,p,u)\in\Lambda_{u_-}$. Therefore, $p=Du_-(x)$ and $u=u_-(x)$.
	\end{proof}

	\subsubsection{The decompositions of the Ma\~n\'e set and the set of Mather measures}

	\begin{proof}[Proof of Theorem \ref{Man}]
		Take any point $(x_0,p_0,u_0)\in\tilde{\mathcal{N}}$. Let  $\big(x(s),p(s),u(s)\big):=\Phi^H_s(x_0,p_0,u_0)$. In view of \eqref{ss1}, we only need to show that 	
		there is  $u_-\in\mathcal{S}^-$ such that
		$(x_0,p_0,u_0)\in\tilde{\mathcal{N}}_{u_-}$, i.e.,
		\[
		u_{\pm}(x_0)=u_0,\quad p_0=Du_{\pm}(x_0).
		\]
		
		Define $u(x,t):=\inf_{s\in\R}h_{x(s),u(s)}(x,t)$ for  $(x,t)\in M\times(0,+\infty)$. For any $t>0$, we notice that the function $x\to u(x,t)=\inf_{s\in\R}h_{x(s),u(s)}(x,t)$ is continuous on $M$ since $u(s)$ is bounded by Proposition \ref{ssb}.
		
		\item [\bf{Step 1:}]  We assert that  $u(x(\sigma),t)=u(\sigma)$, $\forall t>0$, $\forall\x\in\R$. By definition, we have
		\[
		u(x(\x),t)=\inf_{s\in\R}h_{x(s),u(s)}(x(\x),t).
		\]
		Since $\big(x(s),u(s)\big)$ is semi-static, then
		\[
		h_{x(\x-t),u(\x-t)}(x(\x),t)=u(\x).
		\]
		So, in order to prove the above assertion, it suffices to show that
		\begin{align}\label{3-501}
			h_{x(s),u(s)}(x(\x),t)\geqslant u(\x),\quad \forall s\in\R.
		\end{align}
		For $s\leqslant \x$, since $\big(x(s),u(s)\big)$ is semi-static, we get that
		\[
		u(\x)=\inf_{\tau>0}h_{x(s),u(s)}(x(\x),\tau)\leqslant h_{x(s),u(s)}(x(\x),t).
		\]
		For $s>\x$, we have $u(s)=h_{x(\x-t),u(\x-t)}(x(s),s-\x+t)$. Thus
		\begin{equation*}
			\begin{split}
				u(\x)&=\inf_{\tau>0}h_{x(\x-t),u(\x-t)}(x(\x),\tau)\\
				&\leqslant h_{x(\x-t),u(\x-t)}(x(\x),s-\x+2t)\\
				&=\inf_{z\in M}h_{z,h_{x(\x-t),u(\x-t)}(z,s-\x+t)}(x(\x),t)\\
				&\leqslant h_{x(s),h_{x(\x-t),u(\x-t)}(x(s),s-\x+t)}(x(\x),t)\\
				&\leqslant h_{x(s),u(s)}(x(\x),t).
			\end{split}
		\end{equation*}
		So far, we have proved \eqref{3-501} and thus the assertion holds true.
		
		\item [\bf{Step 2:}] We assert that for any $t_2>t_1>0$, we have $u(x,t_2)\leqslant u(x,t_1)$ for any $x\in M$. In fact, by Proposition \ref{prop3} (1) again, we have
		\[
		h_{x(s-t_2+t_1),u(s-t_2+t_1)}(x,t_2)\leqslant h_{x(s),h_{x(s-t_2+t_1),u(s-t_2+t_1)}(x(s),t_2-t_1)}(x,t_1)=h_{x(s),u(s)}(x,t_1),
		\]
		where the equality holds since $\big(x(s),u(s)\big)$ is semi-static. Therefore,
		\[
		u(x,t_2)=\inf_{s\in\R}h_{x(s-t_2+t_1),u(s-t_2+t_1)}(x,t_2)\leqslant \inf_{s\in\R}h_{x(s),u(s)}(x,t_1)=u(x,t_1).
		\]

		\item [\bf{Step 3:}]	
		We assert that the limit $\lim_{t\to+\infty}u(x,t)=:u_-(x)$ exists and $u_-\in\mathcal{S}^-$. If the assertion is true, then by Step 1, we get that $u_-(x(s))=\lim_{t\to+\infty}u(x(s),t)=u(s)$ for any $s\in\R$.
		
		In fact, from Proposition \ref{prop3} (1) and (3), for any $t>\tau>0,\ x\in M$ we get that
		\begin{align}\label{3-506}
			\begin{split}
				u(x,t)&=\inf_{s\in\R}h_{x(s),u(s)}(x,t)\\
				&=\inf_{s\in\R}\inf_{z\in M}h_{z,h_{x(s),u(s)}(z,t-\tau)}(x,\tau)\\
				&=\inf_{z\in M}\inf_{s\in\R}h_{z,h_{x(s),u(s)}(z,t-\tau)}(x,\tau)\\
				&=\inf_{z\in M}h_{z,\inf_{s\in\R}h_{x(s),u(s)}(z,t-\tau)}(x,\tau)\\
				&=\inf_{z\in M}h_{z,u(z,t-\tau)}(x,\tau)=T^-_{\tau}u(x,t-\tau).
			\end{split}
		\end{align}
		Therefore, according to Step 1 and \eqref{3-506}, we have
		$$u_0=u(0)=u(x(0),t+1)=T^-_1u(x_0,t)=\inf_{z\in M}h_{z,u(z,t)}(x_0,1)\leqslant h_{x,u(x,t)}(x_0,1)$$
		for all $t\geqslant0$ and all $x\in M$, where the third equality holds by Proposition \ref{prop1} (4) (i).
		For Proposition \ref{addprop1}, we get that
		$$h^{x_0,u_0}(x,1)\leqslant h^{x_0,h_{x,u(x,t)}(x_0,1)}(x,1)=u(x,t),\quad \forall x\in M,\ \forall t\geqslant0.$$
		Since the function $x\to h^{x_0,u_0}(x,1)$ is bounded on $M$, then $u(x,t)$ is bounded from below on $M\times[0,+\infty)$. And by Step 2, we have
		\[
		u(x,t)\leqslant u(x,1)=\inf_{s\in\R}h_{x(s),u(s)}(x,1)\leqslant h_{x_0,u_0}(x,1),\quad \forall x\in M,\ \forall t>1.
		\]
		Hence, $u(x,t)=T^-_{t-1}u(x,1)$ is bounded on $M\times[1,+\infty)$.
		
		From Proposition \ref{prop5.1} (2), one can get that
		$\liminf_{t\to+\infty}T^-_{t-1}u(x,1)$ exists and belongs to $\mathcal{S}^-$. Hence from Step 2 and \eqref{3-506}, we get that
		$$u_-(x):=\lim_{t\to+\infty}u(x,t)=\lim_{t\to+\infty}T^-_{t-1}u(x,1)=\liminf_{t\to+\infty}T^-_{t-1}u(x,1)\in\mathcal{S}^-.$$

		\item [\bf{Step 4:}] We assert that $x(s)$ is a $(u_-,L,0)$-calibrated curve. Since $\big(x(s),u(s)\big)=\big(x(s),u_-(x(s))\big)$ is semi-static, then by Proposition \ref{pr22}, we have
		\[
		h_{x(t_1),u_-(x(t_1))}(x(t_2),t_2-t_1)=u_-(x(t_1))+\int_{t_1}^{t_2}L\big(x(s),\dot{x}(s),h_{x(t_1),u_-(x(t_1))}(x(s),s-t_1)\big)ds
		\]
		for any $t_1<t_2$. Note that
		\[
		h_{x(t_1),u_-(x(t_1))}(x(t_2),t_2-t_1)=u_-(x(t_2)), \quad h_{x(t_1),u_-(x(t_1))}(x(s),s-t_1)=u_-(x(s)),\ s\in[t_1,t_2].
		\]
		Thus, we get that
		\[
		u_-(x(t_2))-u_-(x(t_1))=\int_{t_1}^{t_2}L\big(x(s),\dot{x}(s),u_-(x(s))\big)ds,
		\]	
		impliying that $x(s)$ is a $(u_-,L,0)$-calibrated curve.
		
		Then by Corollary \ref{inva}, $p_0=\frac{\partial L}{\partial \dot{x}}\big(x(0),\dot{x}(0),u(0)\big)=Du_-(x_0)$ since $u(0)=u_-(x(0))$.	
		
		\item [\bf{Step 5:}] Let $u_+:=\lim\limits_{t\to+\infty}T_t^+u_-$. Then we assert that $u_0=u_{\pm}(x_0)$ and $p_0=Du_{\pm}(x_0)$. By Proposition 	\ref{prop1}, we have
		\[
		T^+_tu_-(x_0)=\sup_{y\in M}h^{y,u_-(y)}(x_0,t)\geqslant h^{x(t),u_-(x(t))}(x_0,t),\quad \forall t>0.
		\]
		Since $(x(s),u_-(x(s)))$ is semi-static, then $
		u_-(x(t))=h_{x_0,u_0}(x(t),t)$.
		By Proposition \ref{addprop1}, we get that
		\[
		h^{x(t),u_-(x(t))}(x_0,t)=u_0.
		\]
		Hence, we have $T^+_tu_-(x_0)\geqslant u_0$, implying that
		\[
		u_+(x_0)\geqslant u_0=u_-(x_0).
		\]
		In view of Corollary \ref{uplus}, $u_+\leqslant u_-$. So far, we have proved that
		\[
		u_{+}(x_0)=u_0=u_-(x_0).
		\]
		At last, by Step 4 and Proposition \ref{iinv}, we have
		\[
		p_0=Du_-(x_0)=Du_+(x_0).
		\]

		Now the proof is complete.
	\end{proof}

	A direct consequence of Theorem \ref{Man} and Proposition \ref{pr1011} is as follows.
	\begin{corollary}\label{mather1}
		$\tilde{\mathcal{N}}\subset
		\mathcal{E}=\{(x,p,u)\in T^*M\times\R: H(x,p,u)=0\}.$
	\end{corollary}

	The following results in this part are devoted to the relation between the set of Mather measure $\mathfrak{M}$ and
	\[
	\mathfrak{M}_{u_-}:=\{\mu\in\mathcal{P}(\T\times\R): \supp\mu\subset \Lambda_{u_-},\ (\Phi^H_t)_\sharp\mu=\mu,\,\ \forall t\in\mathbf{R}\}.
	\]
	
	For any $u_-\in\mathcal{S}^-$, notice from Lemma \ref{com} and Lemma \ref{ss} that $\tilde{\mathcal{N}}_{u_-}$ is a compact $\Phi_t^H$-invariant subset of $T^*M\times\R$. By the Krylov-Bogoliubov Theorem, there exist Borel $\Phi^H_t$-invariant probability measures supported in $\tilde{\mathcal{N}}_{u_-}$.
	Define
	\[
	\tilde{\mathfrak{M}}_{u_-}:=\{\mu\in\mathcal{P}(\T\times\R): \mu \ \text{is}\ \Phi^H_t\text{-invariant with}\ \supp\mu\subset \tilde{\mathcal{N}}_{u_-}\}.
	\]
	
	A direct consequence of Proposition \ref{pr1011} is as follows.
	\begin{corollary}
		Let $u_-\in\mathcal{S}^-$. Then $\tilde{\mathfrak{M}}_{u_-}=\mathfrak{M}_{u_-}$.
	\end{corollary}
	In view of \eqref{ss1}, we get that ,
	\begin{align}\label{2-290}
		\mathfrak{M}_{u_-}=\tilde{\mathfrak{M}}_{u_-}\subset \mathfrak{M}.
	\end{align}
	Hence, $\mathfrak{M}$ is nonempty. Moreover, we can deduce Corollary \ref{mather} directly from Theorem \ref{Man}.

	%
	
	%

	\section{Stability and instability of  stationary solutions}
	We will prove Theorem \ref{th1}, Theorem \ref{th2} and take a closer look at Example \ref{ex1} in this section. Let $u_-\in\mathcal{S}^-$. For any positive real number $\delta$, define
	\[
	u^\delta:=u_-+\delta,\quad 	u_\delta:=u_--\delta.
	\]
	
	\subsection{Key lemmas}
	\begin{lemma}\label{lem1}
		Let $u_-\in \mathcal{S^-}$ satisfy \eqref{A1}.  Then for any constant $c>0$, there is $t_c>0$ depending on $c$ such that each $(u_-,L,0)$-calibrated curve $\gamma:[-t_c,0]\to M$ satisfies
		$$\int^0_{-t_c}\frac{\partial L}{\partial u}\big(\gamma(s),\dot{\gamma}(s),u_-(\gamma(s))\big)ds<(-A+c)t_c,$$
		where $A$ is as in \eqref{A2}.
	\end{lemma}
	
	\begin{proof}
		Assume by contradiction that there are a constant $c_0>0$, a sequence $\{t_n\}$ with $t_n\to+\infty$ as $n\to+\infty$ and
		a sequence of $(u_-,L,0)$-calibrated curve $\gamma_n:[-t_n,0]\to M$ such that
		$$\frac{1}{t_n}\int^0_{-t_n}\frac{\partial L}{\partial u}\big(\gamma_n(s),\dot{\gamma}_n(s),u_-(\gamma_n(s))\big)ds\geqslant-A+c_0.$$
		Define a sequence of measures $\{\mu_n\}$ on $TM\times \R$ by
		$$\int_{TM\times \R}f(x,\dot{x},u)d\mu_n=\frac{1}{t_n}\int^0_{-t_n}f\big(\gamma_n(s),\dot{\gamma}_n(s),u_-(\gamma_n(s))\big)ds,\ \forall f\in C_b(TM\times \R,\R).$$
		From Proposition \ref{houjia}, all $(u_-,L,0)$-calibrated curves are equi-Lipschitz. Thus, for each $n$,
		\[
		\supp \mu_n\subset \{(x,\dot{x},u)\in TM\times\R:  |u|\leqslant \|u_-\|_\infty,\ \|\dot{x}\|\leqslant C_{u_-,l}\}.
		\]
		By Prokhorov Theorem, there exists a point of accumulation of $\{\mu_n\}$ with respect to the vague
		topology, denoted by $\bar{\mu}$. By the similar arguments used in \cite{M91}, one can deduce that $\bar{\mu}$ is  $\Phi^L_t$-invariant. From the definition of $\mu_n$, we get that
		\begin{align}\label{010}
			\int_{TM\times \R}\frac{\partial L}{\partial u}(x,\dot{x},u)d\bar{\mu}\geqslant-A+c_0.
		\end{align}
		Moreover, by Corollary \ref{propA6}, we have $\supp\mu_n\subset\bar{\mathcal{L}}(\Lambda_{u_-})$. Thus $\supp\bar{\mu}\subset\bar{\mathcal{L}}(\Lambda_{u_-})$. Let $\tilde{\mu}=(\bar{\mathcal{L}}^{-1})_\sharp \bar{\mu}$. Then $\tilde{\mu}$ is a $\Phi^H_t$-invariant measure whose support is contained in $\Lambda_{u_-}$. Thus, $\tilde{\mu}\in\mathfrak{M}_{u_-}$. In view of  \begin{equation*}
			\begin{split}
				&\frac{\partial H}{\partial u}(x,p,u)=p\cdot\frac{\partial^2H}{\partial p\partial u}(x,p,u)-\frac{\partial}{\partial u}L(x,\frac{\partial H}{\partial p}(x,p,u),u)\\
				=&p\cdot\frac{\partial^2H}{\partial p\partial u}(x,p,u)-\frac{\partial L}{\partial \dot{x}}(x,\frac{\partial H}{\partial p}(x,p,u),u)\cdot\frac{\partial^2H}{\partial p\partial u}(x,p,u)-\frac{\partial L}{\partial u}(x,\frac{\partial H}{\partial p}(x,p,u),u)\\
				=&-\frac{\partial L}{\partial u}(x,\frac{\partial H}{\partial p}(x,p,u),u)=-\frac{\partial L}{\partial u}\circ\bar{\mathcal{L}}(x,p,u),
			\end{split}
		\end{equation*}
		we get that $$\int_{T^*M\times \R}\frac{\partial H}{\partial u}(x,p,u)d\tilde{\mu}=\int_{TM\times \R}-\frac{\partial L}{\partial u}(x,\dot{x},u)d\bar{\mu}\leqslant A-c_0<A,$$
		which contradicts (A1).
	\end{proof}

	\begin{lemma}\label{coro1}
		Let $u_-\in\mathcal{S}^-$. For any $c>0$, there is a constant $\bar{\delta}>0$ such that for any $(u_-,L,0)$-calibrated curve $\gamma:[a,b]\to M$( $-\infty\leqslant a<b\leqslant+\infty$), there holds
		\begin{equation}\label{eq1.7}
			\Big|\int^1_0\frac{\partial L}{\partial u}\big(\gamma(s),\dot{\gamma}(s),u_-(\gamma(s))+\theta\xi(s)\big)d\theta- \frac{\partial L}{\partial u}\big(\gamma(s),\dot{\gamma}(s),u_-(\gamma(s))\big)	\Big|\leqslant c,\ \forall s\in[a,b]
		\end{equation}
		for any $\xi\in C([a,b],\R)$ satisfying $\|\xi\|_{\infty}\leqslant\bar{\delta}$.
	\end{lemma}
	
	\begin{proof}
		Let $0<\bar{\delta}<1$ be a constant to be determined later. For any $(u_-,L,0)$-calibrated curve $\gamma:[a,b]\to M$, we have that
		\begin{equation}\label{eq1.8}
			\begin{split}
				&\Big|\int^1_0\frac{\partial L}{\partial u}\big(\gamma(s),\dot{\gamma}(s),u_-(\gamma(s))+\theta\xi(s)\big)d\theta-\frac{\partial L}{\partial u}\big(\gamma(s),\dot{\gamma}(s),u_-(\gamma(s))\big)\Big|\\
				&=\Big|\frac{\partial L}{\partial u}\big(\gamma(s),\dot{\gamma}(s),u_-(\gamma(s))+\sigma_s\big)-\frac{\partial L}{\partial u}\big(\gamma(s),\dot{\gamma}(s),u_-(\gamma(s))\big)\Big|\\
				&=\Big|\frac{\partial^2 L}{\partial u^2}\big(\gamma(s),\dot{\gamma}(s),u_-(\gamma(s))+\eta_s\big)\Big|\cdot|\sigma_s|, \quad \forall s\in[a,b],
			\end{split}
		\end{equation}
		where   $|\eta_s|\leqslant|\sigma_s|\leqslant\bar{\delta}$ for any $s\in[a,b].$ From Proposition \ref{houjia}, one can deduce that  $\{\big(\gamma(s),\dot{\gamma}(s)\big): s\in[a,b]\}\subset K:=\{(x,\dot{x})\in TM:  \|\dot{x}\|\leqslant C_{u_-,l}\}$. Thus,
		\[
		\{\big(\gamma(s),\dot{\gamma}(s),u_-(\gamma(s))+\eta_s\big): s\in[a,b]\} \subset K\times[-\|u_-\|_\infty-1,\|u_-\|_\infty+1].
		\]
		Let
		$$V:=\max\Big\{\Big|\frac{\partial^2 L}{\partial u^2}(x,\dot{x},u)\Big|: (x,\dot{x},u)\in K\times[-\|u_-\|_\infty-1,\|u_-\|_\infty+1]\Big\}.$$
		Let $\bar{\delta}=\min\{1,\frac{c}{V}\}$. Then by \eqref{eq1.8}, inequality \eqref{eq1.7} holds true.
	\end{proof}

	\begin{lemma}\label{lem2}
		Let $u_-\in\mathcal{S}^-$ and let $t_0>0$ be a constant. For any $\varepsilon>0$, there is  $\delta_0>0$ such that for any $x\in M$ and any $\delta\in[0,\delta_0]$, there are a  minimizer $\gamma_{\delta}:[-t_0,0]\to M$ of $T^-_{t_0}u_{\delta}(x)$ with $\gamma_{\delta}(0)=x$ and a $(u_-,L,0)$-calibrated curve $\gamma:[-t_0,0]\to M$, satisfying
		\[
		d\Big(\big(\gamma_{\delta}(s),\dot{\gamma}_{\delta}(s)\big),\big(\gamma(s),\dot{\gamma}(s)\big)\Big)<\varepsilon, \quad  \forall s\in[-t_0,0].
		\]
	\end{lemma}
	
	\begin{proof}
		Assume by contradition that there are $\eps_0>0$, a sequence of positive numbers $\{\delta_n\}$ with $1>\delta_n\to0$, a sequence of points $\{x_n\}\subset M$, such that for any minimizer $\gamma_{n}:[-t_0,0]\to M$ of $T^-_{t_0}u_{\delta_n}(x_n)$ with $\gamma_n(0)=x_n$ and any  $(u_-,L,0)$-calibrated curve $\gamma:[-t_0,0]\to M$, there holds
		\begin{align}\label{333}
			d\Big(\big(\gamma_{n}(s_n),\dot{\gamma}_{n}(s_n)\big),\big(\gamma(s_n),\dot{\gamma}(s_n)\big)\Big)\geqslant\eps_0,
		\end{align}
		for some  $s_n\in[-t_0,0]$. Passing to a subsequence if necessary, we may suppose that  $x_n\to x$ and $\gamma_n(-t_0)\to y$.

		Since $\gamma_n$ is the minimizer of $T^-_{t_0}u_{\delta_n}(x_n)$, by Proposition \ref{propA1}, we get that
		\begin{align}\label{000}
			T^-_{t_0}u_{\delta_n}(x_n)=h_{\gamma_n(-t_0),u_{\delta_n}(\gamma_n(-t_0))}(\gamma_n(0),t_0),
		\end{align}
		and $\gamma_n$ is a minimizer of $h_{\gamma_n(-t_0),u_{\delta_n}(\gamma_n(-t_0))}(\gamma_n(0),t_0)$. In view of  Proposition \ref{prop5}, we have
		\[
		\{\big(\gamma_n(s),p_n(s),u_n(s)\big):  n\in \mathbf{N},\  s\in[-t_0,0]\}\subset \mathcal{K}_{u_-,t_0},
		\]
		where $u_n(s)=h_{\gamma_n(-t_0),u_{\delta_n}(\gamma_n(-t_0))}(\gamma_n(s),t_0+s)$, $p_n(s)=\frac{\partial L}{\partial v}\big(\gamma_n(s),\dot{\gamma}_n(s),u_n(s)\big)$,  and $\mathcal{K}_{u_-,t_0}$ is a compact subset of $T^*M\times\R$ depending only on $u_-$ and $t_0$.
		By Proposition \ref{prop4}, one can deduce that
		\begin{equation}\label{eq1.9}
			\dot{\gamma}_n(s)=\frac{\partial H}{\partial p}\big(\gamma_n(s),p_n(s),u_n(s)\big),\quad  \forall s\in[-t_0,0].
		\end{equation}
		So, $\{\dot{\g}_n\}$ is uniformly bounded on $[-t_0,0]$. That means $\{\gamma_n\}$ is equi-Lipschitz. Utilizing Arzel\`a-Ascoli Theorem   and if necessary passing to a subsequence, we have
		\begin{align}\label{111}
			\g_n\to\g,\quad \text{uniformly on}\ [-t_0,0],
		\end{align}
		for some $\g:[-t_0,0]\to M$ with $\g(0)=x$ and $\g(-t_0)=y$.
		From \eqref{eq1.9}, we obtain that
		\begin{equation*}
			\begin{split}
				\ddot{\gamma}_n=&\frac{\partial^2H}{\partial p\partial x}(\gamma_n,p_n,u_n)\cdot\dot{\gamma}_n+\frac{\partial^2H}{\partial p\partial u}(\gamma_n,p_n,u_n)\cdot\dot{u}_n+\frac{\partial^2H}{\partial p^2}(\gamma_n,p_n,u_n)\cdot\dot{p}_n\\
				=&\Big\{\frac{\partial^2H}{\partial p\partial x}(\gamma_n,p_n,u_n)\cdot\frac{\partial H}{\partial p}(\gamma_n,p_n,u_n)+\frac{\partial^2H}{\partial p\partial u}(\gamma_n,p_n,u_n)\cdot\big(\frac{\partial H }{\partial p}(\gamma_n,p_n,u_n)\cdot p_n\\
				-&H(\gamma_n,p_n,u_n)\big)+\frac{\partial^2H}{\partial p^2}(\gamma_n,p_n,u_n)\cdot\big(-\frac{\partial H}{\partial x}(\gamma_n,p_n,u_n)-\frac{\partial H}{\partial u}(\gamma_n,p_n,u_n)\cdot p_n\big)\Big\},
			\end{split}
		\end{equation*}
		for $\forall s\in[-t_0,0]$. This implies that  $\{\ddot{\gamma}_n\}$ is uniformly bounded in $[-t_0,0]$. By Arzel\`a-Ascoli Theorem again,
		passing to a subsequence if necessary, we may suppose that
		\begin{align}\label{222}
			\dot{\g}_n\to\alpha,\quad \text{uniformly on}\ [-t_0,0].
		\end{align}
		In view of \eqref{111} and \eqref{222}, we deduce that
		$\alpha=\dot{\g}$.
		
		To finish the proof, we only need to show that $\g$ is a $(u_-,L,0)$-calibrated curve, which contradicts \eqref{333}.
		By Proposition \ref{prop3} (2), we have for each $n\in\mathbf{N}$ that
		\begin{align}\label{555}
			|h_{\gamma_n(-t_0),u_-(\gamma_n(-t_0))}(x_n,t_0)-h_{\gamma_n(-t_0),u_{\delta_n}(\gamma_n(-t_0))}(x_n,t_0)|\leqslant k\|u_--u_{\delta_n}\|_{\infty}=k\delta_n,
		\end{align}
		where $k$ denotes the Lipschitz constant of  $h_{\cdot,\cdot}(\cdot,t_0)$ on $M\times [-\|u_-\|_\infty-1,\|u_-\|_\infty+1]\times M$.		
		By Proposition \ref{prop1} (3), we have
		\begin{align}\label{444}
			|T^-_{t_0}u_{\delta_n}(x_n)-u_-(x_n)|\leqslant\|T^-_{t_0}u_{\delta_n}-u_-\|_\infty\leqslant e^{\lambda t_0}\|u_{\delta_n}-u_-\|_\infty\leqslant e^{\lambda t_0}\delta_n.
		\end{align}
		Combining \eqref{000}, \eqref{555} and \eqref{444}, we have
		$$|h_{\gamma_n(-t_0),u_-(\gamma_n(-t_0))}(x_n,t_0)-u_-(x_n)|\leqslant(e^{\lambda t_0}+k)\delta_n,\quad \forall n\in\mathbf{N}.
		$$
		Letting $n$ tend to $+\infty$, it yields $u_-(x)=h_{\gamma(-t_0),u_-(\gamma(-t_0))}(x,t_0)$. Moreover, by Proposition \ref{prop3} (1) and Proposition \ref{propA1}, we have
		$$h_{\gamma_n(-t_0),u_{\delta_n}(\gamma_n(-t_0))}(x_n,t_0)=h_{\gamma_n(s-t_0),h_{\gamma_n(-t_0),u_{\delta_n}(\gamma_n(-t_0))}(\gamma_n(s-t_0),s)}(x_n,t_0-s),\ \forall s\in[0,t_0].$$
		Letting $n$ tend to $+\infty$, we get that
		$$h_{\gamma(-t_0),u_-(\gamma(-t_0))}(x,t_0)=h_{\gamma(s-t_0),h_{\gamma(-t_0),u_-(\gamma(-t_0))}(\gamma(s-t_0),s)}(x,t_0-s),\ \forall s\in[0,t_0].$$
		Thus, by Proposition \ref{prop3} (1), $\gamma$ is a minimizer of $h_{\gamma(-t_0),u_-(\gamma(-t_0))}(x,t_0)$. Since $u_-\in \mathcal{S}^-$,  by Proposition \ref{prop1} (4),
		$$\inf_{y\in M}h_{y,u_-(y)}(x,t_0)=T^-_{t_0}u_-(x)=u_-(x)=h_{\gamma(-t_0),u_-(\gamma(-t_0))}(x,t_0).$$
		Then from Proposition \ref{propA2} we deduce that  $\gamma$ is a minimizer of $T^-_{t_0}u_-(x)$. Using Proposition \ref{propA1},
		$$u_-(\gamma(b))=T^-_{b-a}u_-(\gamma(b))=u_-(\gamma(a))+\int^b_aL\big(\gamma(s),\dot{\gamma}(s),u_-(\gamma(s))\big)ds,\ \forall -t_0\leqslant a<b\leqslant0.$$
		This means that $\gamma$ is a $(u_-,L,0)$-calibrated curve. The proof is complete.
	\end{proof}
	
	\begin{corollary}\label{coro2}
		Let $u_-\in\mathcal{S}^-$ and let $t_0>0$ be a constant. For any $c>0$, there is a constant $\delta_c>0$ such that for any $x\in M$ and any $\delta\in[0,\delta_c]$, there are a minimizer $\gamma_{\delta}:[-t_0,0]\to M$ of $T^-_{t_0}u_{\delta}(x)$ with $\g_{\delta}(0)=x$ and a $(u_-,L,0)$-calibrated curve $\gamma:[-t_0,0]\to M$, satisfying
		\begin{equation*}
			\Big|\int^1_0\frac{\partial L}{\partial u}\Big(\gamma_{\delta}(s),\dot{\gamma}_{\delta}(s),u_-(\gamma_{\delta}(s))-\theta w_{\delta}(s+t_0)\Big)d\theta-\frac{\partial L}{\partial u}\big(\gamma(s),\dot{\gamma}(s),u_-(\gamma(s))\big)\Big|\leqslant c,
		\end{equation*}
		for all $s\in[-t_0,0]$, where $w_{\delta}(\tau)=u_-(\gamma_{\delta}(\tau-t_0))-T^-_{\tau}u_{\delta}(\gamma_{\delta}(\tau-t_0)),\ \forall\tau\in[0,t_0]$.
	\end{corollary}
	
	\begin{proof}
		We get from Proposition \ref{houjia} that for any $(u_-,L,0)$-calibrated curve $\gamma:[-t_0,0]\to M$,
		\[
		\{\big(\gamma(s),\dot{\gamma}(s)\big): s\in[-t_0,0]\}\subset K:=\{(x,\dot{x}):x\in M, \|\dot{x}\|\leqslant C_{u_-,l}\}\subset TM,
		\]
		where $l$ and $C_{u_-,l}$ are as in Lemma \ref{lem1}. 		
		Note that $\frac{\partial L}{\partial u}(\cdot,\cdot,\cdot)$ is uniformly continuous on the compact set $Q=K_{1}\times[-\|u_-\|_\infty-1,\|u_-\|_\infty+1]$ where $K_1:=\{(x,\dot{x})\in TM: d(K,(x,\dot{x}))\leqslant 1\}$, $d(K,(x,\dot{x}))=\inf_{(y,\dot{y})\in K}d\big((x,\dot{x}),(y,\dot{y})\big)$. Thus, there is a constant $\varepsilon\in(0,1)$ such that
		\begin{equation}\label{eq1.10}
			\Big|\frac{\partial L}{\partial u}(x,\dot{x},u)-\frac{\partial L}{\partial u}(x_0,\dot{x}_0,u_0)\Big|\leqslant c
		\end{equation}
		holds for any $(x,\dot{x},u)$, $(x_0,\dot{x}_0,u_0)\in Q$ satisfying $d\big((x,\dot{x},u),(x_0,\dot{x}_0,u_0)\big)\leqslant\varepsilon$.
		
		From Lemma \ref{lem2}, we deduce that for any $x\in M$ and any $\delta\in[0,\delta_0]$ where $\delta_0$ is as in Lemma \ref{lem2}, we can find a minimizer $\gamma_{\delta}:[-t_0,0]\to M$ of $T^-_{t_0}u_{\delta}(x)$ with $\gamma_{\delta}(0)=x$ and a $(u_-,L,0)$-calibrated curve $\gamma:[-t_0,0]\to M$, such that
		\begin{align}\label{777}
			d\Big(\big(\gamma(s),\dot{\gamma}(s)\big),\big(\gamma_{\delta}(s),\dot{\gamma}_{\delta}(s)\big)\Big)<\min\{\frac{\varepsilon}{3},\frac{\varepsilon}{6}l^{-1}\}, \quad \forall s\in[-t_0,0].
		\end{align}
		
		Taking $\delta_c=\min\{\delta_0,\frac{\varepsilon}{6}e^{-\lambda t_0}\}$. By Proposition \ref{prop1} (3), we have for any $\delta\in[0,\delta_c]$ that
		\begin{align}\label{666}
			|w_{\delta}(\tau)|\leqslant\|T^-_{\tau}u_{\delta}-u_-\|_{\infty}\leqslant e^{\lambda \tau}\|u_{\delta}-u_-\|_{\infty}\leqslant e^{\lambda t_0}\delta_c\leqslant\frac{\varepsilon}{6},\quad   \forall \tau\in[0,t_0].
		\end{align}
		It is clear that for any $s\in[-t_0,0]$, we have
		\begin{equation*}
			\int^1_0\frac{\partial L}{\partial u}\Big(\gamma_{\delta}(s),\dot{\gamma}_{\delta}(s),u_-(\gamma_{\delta}(s))-\theta w_{\delta}(s+t_0)\Big)d\theta=\frac{\partial L}{\partial u}\Big(\gamma_{\delta}(s),\dot{\gamma}_{\delta}(s),u_-(\gamma_{\delta}(s))-\xi_sw_{\delta}(s+t_0)\Big),	
		\end{equation*}
		for some  $\xi_s$ satisfying $|\xi_s|\leqslant 1$.
		Combining  \eqref{777} and \eqref{666}, we have
		\begin{align*}
			d\Big(\big(\gamma(s),\dot{\gamma}(s),u_-(\gamma(s))\big),\big(\gamma_{\delta}(s),\dot{\gamma}_{\delta}(s),u_-(\gamma_{\delta}(s))-\xi_sw_{\delta}(s+t_0)\big)\Big)\leqslant \eps
		\end{align*}
		for any $s\in[-t_0,0]$. Using \eqref{eq1.10}, we get that
		\begin{equation*}
			\begin{split}
				&\Big|\int^1_0\frac{\partial L}{\partial u}\Big(\gamma_{\delta}(s),\dot{\gamma}_{\delta}(s),u_-(\gamma_{\delta}(s))-\theta w_{\delta}(s+t_0)\Big)d\theta-\frac{\partial L}{\partial u}\big(\gamma(s),\dot{\gamma}(s),u_-(\gamma(s))\big)\Big|\\
				&=\Big|\frac{\partial L}{\partial u}\Big(\gamma_{\delta}(s),\dot{\gamma}_{\delta}(s),u_-(\gamma_{\delta}(s))-\xi_sw_{\delta}(s+t_0)\Big)-\frac{\partial L}{\partial u}\big(\gamma(s),\dot{\gamma}(s),u_-(\gamma(s))\big)\Big|\leqslant c,
			\end{split}
		\end{equation*}
		for all $s\in[-t_0,0]$.
	\end{proof}

	\subsection{The proof of Theorem \ref{th1}}
	
	\begin{proof}[Proof of Theorem \ref{th1}]
		Let $A$ be as in \eqref{A2}. Let $0<c<\frac{A}{2}$ be an arbitrary constant. By Lemma \ref{lem1}, there is $t_c>0$ such that
		\[
		\int^0_{-t_c}\frac{\partial L}{\partial u}\big(\tilde{\gamma}(s),\dot{\tilde{\gamma}}(s),u_-(\tilde{\gamma}(s))\big)ds<(-A+c)t_c
		\]
		holds true for all $(u_-,L,0)$-calibrated curves $\tilde{\gamma}:[-t_c,0]\to M$.
		
		Since the rest of the proof is quite long, we divide it into three parts.
		
		\item [\bf{Step 1:}] Define
		\[
		\delta^c:=e^{-\lambda t_c}\bar{\delta},\quad u^{\delta}:=u_-+\delta,
		\]
		where $\delta\in[0,\delta^c]$  and $\bar{\delta}$ is as in Lemma \ref{coro1}. We aim to estimate the term $\|T^-_{t_c}u^{\delta}-u_-\|_{\infty}$.
		
		For any $x\in M$, let $\gamma:[-t_c,0]\to M$ be a $(u_-,L,0)$-calibrated curve with $\gamma(0)=x$. Let
		\begin{align}\label{888}
			\bar{u}(s):=u_-(\gamma(s-t_c)) ,\quad  \bar{u}^{\delta}(s):=T^-_s u^{\delta}(\gamma(s-t_c)),\quad s\in [0,t_c].
		\end{align}
		For any $s$, $\Delta s$ with $0\leqslant s<s+\Delta  s\leqslant t_c$, in view of $u_-\in\mathcal{S}^-$ and \eqref{888}, we have
		\begin{equation}\label{eq1.3}
			\begin{split}
				&\bar{u}(s+\Delta s)=u_-(\gamma(s+\Delta s-t_c))\\
				=&u_-(\gamma(s-t_c))+\int^{\Delta s}_0 L\big(\gamma(\tau+s-t_c), \dot{\gamma}(\tau+s-t_c),u_-(\gamma(\tau+s-t_c))\big)d\tau\\
				=&\bar{u}(s)+\int^{s+\Delta s}_s L\big(\gamma(\tau-t_c), \dot{\gamma}(\tau-t_c),\bar{u}(\tau)\big)d\tau.
			\end{split}
		\end{equation}
		In view of \eqref{888} and the semigroup property of $T^-_t$, we have
		\begin{equation}\label{eq1.4}
			\begin{split}
				&\bar{u}^{\delta}(s+\Delta s)=T^-_{s+\Delta s}u^{\delta}(\gamma(s+\Delta s-t_c))=T^-_{\Delta s}\circ T^-_su^{\delta}(\gamma(s+\Delta s-t_c))\\
				=&\inf_{\tilde{\gamma}(0)=\gamma(s+\Delta s-t_c)}\{T^-_su^{\delta}(\tilde{\gamma}(-\Delta s))+\int^0_{-\Delta s} L\big(\tilde{\gamma}(\tau),\dot{\tilde{\gamma}}(\tau), T^-_{\tau+\Delta s}\circ T^-_su^{\delta}(\tilde{\gamma}(\tau))\big)d\tau\}\\
				\leqslant&T^-_s u^{\delta}(\gamma(s-t_c))+\int^{\Delta s}_0 L\big(\gamma(\tau+s-t_c), \dot{\gamma}(\tau+s-t_c),T^-_{\tau+s} u^{\delta}(\gamma(\tau+s-t_c))\big)d\tau\\
				=&\bar{u}^{\delta}(s)+\int^{s+\Delta s}_s L\big(\gamma(\tau-t_c),\dot{\gamma}(\tau-t_c) ,\bar{u}^{\delta}(\tau)\big)d\tau.
			\end{split}
		\end{equation}
		Let $\bar{w}(s):=\bar{u}^{\delta}(s)-\bar{u}(s)$, $s\in[0,t_c]$. By Proposition \ref{prop1} (1),  $\bar{w}(s)\geqslant0$, $\forall s\in[0,t_c]$ and $\bar{w}(0)=\delta$. Combining \eqref{eq1.3} and \eqref{eq1.4}, we get that
		\begin{equation}\label{999}
			\begin{split}
				&\bar{w}(s+\Delta s)\\
				\leqslant&\bar{w}(s)+\int^{s+\Delta s}_s \big(L\big(\gamma(\tau-t_c), \dot{\gamma}(\tau-t_c),\bar{u}^{\delta}(\tau)\big)-L\big(\gamma(\tau-t_c), \dot{\gamma}(\tau-t_c),\bar{u}(\tau)\big)\big)d\tau\\
				\leqslant&\bar{w}(s)+\int^{s+\Delta s}_s\bar{w}(\tau)\int^1_0 \frac{\partial L}{\partial u}\big(\gamma(\tau-t_c), \dot{\gamma}(\tau-t_c),\bar{u}(\tau)+\theta \bar{w}(\tau)\big)d\theta d\tau.
			\end{split}
		\end{equation}
		By Proposition \ref{prop1} (2), one can deduce that $\bar{w}(s)$ is Lipschitz on $[0,t_c]$. So $\bar{w}$ is differentiable almost everywhere in $[0,t_c]$. Therefore, by \eqref{999} we have
		$$\dot{\bar{w}}(s)\leqslant \bar{w}(s)\int^1_0 \frac{\partial L}{\partial u}\Big(\gamma(s-t_c), \dot{\gamma}(s-t_c),u_-(\gamma(s-t_c))+\theta \bar{w}(s)\Big)d\theta,\quad \text{a.e.}\ s\in[0,t_c].$$
		By Proposition \ref{prop1} (3),
		\[
		|\bar{w}(s)|\leqslant\|T^-_su^{\delta}-u_-\|_{\infty}\leqslant e^{\lambda s}\|u^{\delta}-u_-\|_{\infty}\leqslant e^{\lambda t_c}\delta^c=\bar{\delta},\quad \forall s\in[0,t_c].
		\]
		Using Lemma \ref{coro1}, we have
		$$\dot{\bar{w}}(s)\leqslant \bar{w}(s)\Big(\frac{\partial L}{\partial u}\big(\gamma(s-t_c),\dot{\gamma}(s-t_c),u_-(\gamma(s-t_c))\big)+c\Big).$$
		Consider the following Cauchy problem
		\begin{equation*}
			\left\{
			\begin{aligned}
				\dot{w}(s)&= w(s)\Big(\frac{\partial L}{\partial u}\big(\gamma(s-t_c),\dot{\gamma}(s-t_c),u_-(\gamma(s-t_c))\big)+c\Big),\\
				w(0)&=\delta.
			\end{aligned}
			\right.
		\end{equation*}
		Hence, we get that
		$$|T^-_{t_c}u^{\delta}(x)-u_-(x)|=\bar{w}(t_c)\leqslant \bar{w}(0)e^{\int^{t_c}_0[\frac{\partial L}{\partial u}(\gamma(s-t_c),\dot{\gamma}(s-t_c),u_-(\gamma(s-t_c)))+c]ds}<\delta e^{(-A+2c)t_c},$$
		where the last inequality comes from Lemma \ref{lem1}.
		According to the arbitrariness of $x$, it yields
		$$\|T^-_{t_c}u^{\delta}-u_-\|_{\infty}\leqslant\delta e^{(-A+2c)t_c},\quad \forall\delta\in[0,\delta^c].$$

		\item [\bf{Step 2:}] Define
		\[
		u_{\delta}:=u_--\delta,\quad \delta\in[0,\delta_c],
		\]
		where $\delta_c$ is as in Corollary \ref{coro2}.
		We will estimate the term $\|T^-_{t_c}u_{\delta}-u_-\|_{\infty}$.
		
		For any $x\in M$, by Corollary \ref{coro2} we can find a minimizer $\gamma_{\delta}:[-t_c,0]\to M$ of $T^-_{t_c}u_{\delta}(x)$ with $\gamma_{\delta}(0)=x$ such that
		\begin{equation*}
			\begin{split}
				&\int^1_0\frac{\partial L}{\partial u}\Big(\gamma_{\delta}(s-t_c),\dot{\gamma}_{\delta}(s-t_c),u_-(\gamma_{\delta}(s-t_c))-\theta\big(u_-(\gamma_{\delta}(s-t_c))-T^-_su_{\delta}(\gamma_{\delta}(s-t_c))\big)\Big)d\theta\\
				&\leqslant\frac{\partial L}{\partial u}\big(\gamma_0(s-t_c),\dot{\gamma}_0(s-t_c),u_-(\gamma_0(s-t_c))\big)+c,\quad \forall s\in[0,t_c],
			\end{split}
		\end{equation*}
		for some $(u_-,L,0)$-calibrated curve $\gamma_0:[-t_c,0]\to M$. Define $\tilde{u}(s):=u_-(\gamma_{\delta}(s-t_c))$ and $\tilde{u}_{\delta}(s):=T^-_s u_{\delta}(\gamma_{\delta}(s-t_c))$, $s\in[0,t_c]$. For any $s$, $\Delta s$ satisfying $0\leqslant s<s+\Delta s\leqslant t_c$, we have
		\begin{equation}\label{eq1.5}
			\begin{split}
				&\tilde{u}(s+\Delta s)=u_-(\gamma_{\delta}(s+\Delta s-t_c))=T^-_{\Delta s}u_-(\gamma_{\delta}(s+\Delta s-t_c))\\
				=&\inf_{\tilde{\gamma}(0)=\gamma_{\delta}(s+\Delta s-t_c)}\Big\{u_-(\tilde{\gamma}(-\Delta s))+\int^0_{-\Delta s} L\big(\tilde{\gamma}(\tau), \dot{\tilde{\gamma}}(\tau),T^-_{\tau+\Delta s}u_-(\tilde{\gamma}(\tau))\big)d\tau\Big\}\\
				\leqslant&u_-(\gamma_{\delta}(s-t_c))+\int^{\Delta s}_0 L\big(\gamma_{\delta}(\tau+s-t_c), \dot{\gamma}_{\delta}(\tau+s-t_c),u_-(\gamma_{\delta}(\tau+s-t_c))\big)d\tau\\
				=&\tilde{u}(s)+\int^{s+\Delta s}_s L\big(\gamma_{\delta}(\tau-t_c), \dot{\gamma}_{\delta}(\tau-t_c),\tilde{u}(\tau)\big)d\tau,
			\end{split}
		\end{equation}
		and
		\begin{align}\label{eq1.6}
			\begin{split}
				&\tilde{u}_{\delta}(s+\Delta s)=T^-_{s+\Delta s}u_{\delta}(\gamma_{\delta}(s+\Delta s-t_c))=T^-_{\Delta s}\circ T^-_su_{\delta}(\gamma_{\delta}(s+\Delta s-t_c))\\
				=&\inf_{\tilde{\gamma}(0)=\gamma_{\delta}(s+\Delta s-t_c)}\{T^-_su_{\delta}(\tilde{\gamma}(-\Delta s))+\int^0_{-\Delta s} L\big(\tilde{\gamma}(\tau), \dot{\tilde{\gamma}}(\tau),T^-_{\tau+\Delta s}\circ T^-_su_{\delta}(\tilde{\gamma}(\tau))\big)d\tau\}\\
				=&T^-_s u_{\delta}(\gamma_{\delta}(s-t_c))+\int^{s+\Delta s}_s L\big(\gamma_{\delta}(\tau-t_c), \dot{\gamma}_{\delta}(\tau-t_c),T^-_{\tau} u_{\delta}(\gamma_{\delta}(\tau-t_c))\big)d\tau\\
				=&\tilde{u}_{\delta}(s)+\int^{s+\Delta s}_s L\big(\gamma_{\delta}(\tau-t_c), \dot{\gamma}_{\delta}(\tau-t_c),\tilde{u}_{\delta}(\tau)\big)d\tau.
			\end{split}
		\end{align}
		Define $\tilde{w}(s):=\tilde{u}(s)-\tilde{u}_{\delta}(s)$, $s\in[0,t_c]$. From Proposition \ref{prop1} (1) we have $\tilde{w}(s)\geqslant0$, $\forall s\in[0,t_c]$ and $\tilde{w}(0)=\delta$. By \eqref{eq1.5} and \eqref{eq1.6}, we get that
		\begin{equation*}
			\begin{split}
				&\tilde{w}(s+\Delta s)\\
				\leqslant&\tilde{w}(s)+\int^{s+\Delta s}_s \Big(L\big(\gamma_{\delta}(\tau-t_c), \dot{\gamma}_{\delta}(\tau-t_c),\tilde{u}(\tau)\big)-L\big(\gamma_{\delta}(\tau-t_c), \dot{\gamma}_{\delta}(\tau-t_c),\tilde{u}_{\delta}(\tau)\big)\Big)d\tau\\
				\leqslant&\tilde{w}(s)+\int^{s+\Delta s}_s\tilde{w}(\tau)\int^1_0 \frac{\partial L}{\partial u}\Big(\gamma_{\delta}(\tau-t_c), \dot{\gamma}_{\delta}(\tau-t_c),\tilde{u}(\tau)-\theta \tilde{w}(\tau)\Big)d\theta d\tau.
			\end{split}
		\end{equation*}
		By Proposition \ref{prop1} (2),  $\tilde{w}(s)$ is Lipschitz on $s\in[0,t_c]$ which implies that  $\tilde{w}$ is differentiable almost everywhere on $[0,t_c]$. Therefore according to above inequality, we have
		$$\dot{\tilde{w}}(s)\leqslant \tilde{w}(s)\int^1_0 \frac{\partial L}{\partial u}\Big(\gamma_{\delta}(s-t_c), \dot{\gamma}_{\delta}(s-t_c),u_-(\gamma_{\delta}(s-t_c))-\theta \tilde{w}(s)\Big)d\theta,\ \text{a.e.}\ s\in[0,t_c].
		$$
		From Corollary \ref{coro2}, we have
		$$\dot{\tilde{w}}(s)\leqslant\tilde{w}(s)\Big(\frac{\partial L}{\partial u}\big(\gamma_0(s-t_c),\dot{\gamma}_0(s-t_c),u_-(\gamma_0(s-t_c))\big)+c\Big).$$
		Consider the Cauchy problem
		\begin{equation*}
			\left\{
			\begin{aligned}
				\dot{w}(s)&= w(s)\Big(\frac{\partial L}{\partial u}\big(\gamma_0(s-t_c),\dot{\gamma}_0(s-t_c),u_-(\gamma_0(s-t_c))\big)+c\Big),\\
				w(0)&=\delta
			\end{aligned}
			\right.
		\end{equation*}
		and by Lemma \ref{lem1}, for any $x\in M$ we get
		$$|T^-_{t_c}u_{\delta}(x)-u_-(x)|=\tilde{w}(t_c)\leqslant \tilde{w}(0)e^{\int^{t_c}_0[\frac{\partial L}{\partial u}(\gamma_0(s-t_c),\dot{\gamma}_0(s-t_c),u_-(\gamma_0(s-t_c)))+c]ds}<\delta e^{(-A+2c)t_c}.$$
		It yields
		$$\|T^-_{t_c}u_{\delta}-u_-\|_{\infty}\leqslant\delta e^{(-A+2c)t_c},\quad \forall\delta\in[0,\delta_c].$$

		\item [\bf{Step 3:}] Let  $\Delta_c=\min\{\delta^c,\delta_c\}$. Note that $\Delta_c$ depends on $c$, since $\delta^c$ and $\delta_c$ depend on $c$.
		Take $\varphi\in C(M,\R)$ with $\|\varphi-u_-\|_\infty=\delta\leqslant\Delta_c$. Then $u_{\delta_c}\leqslant u_{\delta}\leqslant\varphi\leqslant u^{\delta}\leqslant u^{\delta^c}$.
		
		From Step 1, Step 2 and Proposition \ref{prop1} (1) we get that
		$$u_{\delta e^{(-A+2c)t_c}}=u_--\delta e^{(-A+2c)t_c}\leqslant T^-_{t_c}u_{\delta}\leqslant T^-_{t_c}\varphi\leqslant T^-_{t_c}u^{\delta}\leqslant u_-+\delta e^{(-A+2c)t_c}=u^{\delta e^{(-A+2c)t_c}}.$$
		And since $\delta e^{(-A+2c)t_c}<\delta\leqslant\Delta_c$, we still have
		$$u_--\delta e^{2(-A+2c)t_c}\leqslant T^-_{t_c}u_{\delta e^{(-A+2c)t_c}}\leqslant T^-_{2t_c}\varphi\leqslant T^-_{t_c}u^{\delta e^{(-A+2c)t_c}}\leqslant u_-+\delta e^{2(-A+2c)t_c}.$$
		So, one can deduce that
		\begin{align}\label{99}
			\|T^-_{nt_c}\varphi-u_-\|_{\infty}\leqslant\delta e^{n(-A+2c)t_c},\quad \forall n\in \mathbf{N}.
		\end{align}
		For any $t>0$, let $t=nt_c+\bar{t}$, $n\in \mathbf{N}$, $\bar{t}\in[0,t_c)$. Then using \eqref{99} and Proposition \ref{prop1} (3), we have
		\begin{equation}\label{eq1.11}
			\|T^-_t\varphi-u_-\|_{\infty}\leqslant e^{\lambda\bar{t}}\|T^-_{nt_c}\varphi-u_-\|_{\infty}\leqslant\delta\cdot e^{\lambda\bar{t}+n(-A+2c)t_c}\leqslant\|\varphi-u_-\|_\infty\cdot C_c\cdot e^{(-A+2c)t},
		\end{equation}
		where $C_c=e^{(\lambda+A-2c)t_c}$ depends on $c$.
		
		Let $c=\frac{A}{4}$ and define $\Delta:=\Delta_{\frac{A}{4}}$. Thus for any $\varphi\in C(M,\R)$ satisfying $\|\varphi-u_-\|_\infty\leqslant\Delta$, we obtain that
		\begin{equation}\label{asy1}
			\|T^-_t\varphi-u_-\|_{\infty}\leqslant\|\varphi-u_-\|_\infty\cdot C_{\frac{A}{4}}\cdot e^{-\frac{A}{2}t},\quad  \forall t>0,
		\end{equation}
		which means $u_-$ is Lyapunov stable and locally asymptotically stable.

		Moreover, suppose $\varphi\in C(M,\R)$ satisfies $\|\varphi-u_-\|_\infty\leqslant\Delta$. For any $c>0$, let $s_c:=\max\{1,-\frac{2}{A}(\ln\Delta_c-\ln(\Delta\cdot C_{\frac{A}{4}}))\}$. So we get from inequality \eqref{asy1} that $\|T^-_{s_c}\varphi-u_-\|\leqslant\Delta_c$. Therefore, we can use inequality \eqref{eq1.11} again and get that
		$$\|T^-_{t+s_c}\varphi-u_-\|\leqslant\|T^-_{s_c}\varphi-u_-\|\cdot C_c\cdot e^{(-A+2c)t}\leqslant\Delta_c\cdot C_c\cdot e^{(-A+2c)t},\quad \forall t>0.$$
		Thus,  $\limsup\limits_{t\to+\infty}\frac{\ln\|T^-_t\varphi-u_-\|_{\infty}}{t}\leqslant-A+2c,\ \forall c>0$. According to the arbitrariness of $c$, we have  $\limsup\limits_{t\to+\infty}\frac{\ln\|T^-_t\varphi-u_-\|_{\infty}}{t}\leqslant-A.$
	\end{proof}

	\subsection{The proof of Theorem \ref{th2} }
	
	Before proving Theorem \ref{th2}, we need to show the following result first.
	\begin{lemma}\label{lem4}
		Let $u_-\in\mathcal{S}^-$. If there exists $\mu\in\mathfrak{M}_{u_-}$ satisfying $$\int_{T^*M\times R}\frac{\partial H}{\partial u}(x,p,u)d\mu=:-A'<0,$$
		then for any  $t_0>0$ there exists  $x_0\in M$ and a $(u_-,L,0)$-calibrated curve $\gamma_0:(-\infty,0]\to M$ with $\gamma_0(0)=x_0$ such that
		\begin{equation}\label{eq4.1}
			\int^0_{-t_0}\frac{\partial L}{\partial u}\big(\gamma_0(s),\dot{\gamma_0}(s),u_-(\gamma_0(s))\big)ds\geqslant A'\cdot t_0.
		\end{equation}
	\end{lemma}
	
	\begin{proof}
		Notice that
		\begin{align}\label{1010}
			\int_{TM\times \R}\frac{\partial L}{\partial u}(x,\dot{x},u)d\mu_L=-\int_{T^*M\times \R}\frac{\partial H}{\partial u}(x,p,u)d\mu=A'>0,
		\end{align}
		where $\mu_L=\bar{\mathcal{L}}_\sharp \mu$ is a $\Phi^L_t$-invariant measure  and $\supp\mu_L\subset\bar{\mathcal{L}}(\tilde{\mathcal{N}}_{u_-})$.
		
		Assume by contradiction that there exists $\kappa>0$, such that
		$$\int^0_{-\kappa}\frac{\partial L}{\partial u}\big(\gamma_x(s),\dot{\gamma}_x(s),u_-(\gamma_x(s))\big)ds<A'\cdot\kappa$$
		holds for all $x\in M$ and all $(u_-,L,0)$-calibrated curves $\gamma_x:(-\infty,0]\to M$ with  $\gamma_x(0)=x$. Since  $\mu_L$ is $\Phi^L_t$-invariant, we have
		$$\int_{TM\times \R}\frac{\partial L}{\partial u}d\mu_L=\int_{TM\times \R}\frac{\partial L}{\partial u}d(\Phi^L_s)_\sharp\mu_L=\int_{TM\times \R}\frac{\partial L}{\partial u}\circ\Phi^L_sd\mu_L, \quad \forall s\in \R.$$
		Since $\supp\mu_L\subset\bar{\mathcal{L}}(\tilde{\mathcal{N}}_{u_-})$, we have $\supp\mu_L\subset\{(\gamma(0),\dot{\gamma}(0),u_-(\gamma(0))): \gamma:(-\infty,+\infty)\to M\ \ \text{is a}\  (u_-,L,0)\ \text{-calibrated curve}\}$ from Proposition \ref{iinv}. By integration on $[-\kappa,0]$, we get from Proposition \ref{propA5} that
		\begin{equation*}
			\begin{split}
				\kappa\int_{TM\times \R}\frac{\partial L}{\partial u}(x,\dot{x},u)d\mu_L
				=&\int^0_{-\kappa}\int_{TM\times \R}\frac{\partial L}{\partial u}\Big(\Phi^L_s(x,\dot{x},u)\Big)d\mu_L ds\\
				=&\int_{TM\times \R}\int^0_{-\kappa}\frac{\partial L}{\partial u}\big(\gamma_x(s),\dot{\gamma}_x(s),u_-(\gamma_x(s))\big)dsd\mu_L\\
				<&\int_{TM\times \R}A'\cdot\kappa d\mu_L=A'\cdot\kappa,	
			\end{split}
		\end{equation*}
		which contradicts \eqref{1010}.
	\end{proof}

	\begin{proof}[Proof of Theorem \ref{th2}]
		Let $c=\frac{A'}{2}$, where $A'$ is as in Lemma \ref{lem4}. By Lemma \ref{coro1},  we can find $\bar{\delta}>0$ such that \eqref{eq1.7} holds true for any $(u_-,L,0)$-calibrated curve $\gamma:[a,b]\to M$ and any $\xi\in C([a,b],\R)$ with $\|\xi\|_{\infty}\leqslant\bar{\delta}$.
		
		Let $\Delta'=\bar{\delta}$ and  $u_{\varepsilon}:=u_--\varepsilon$, where $\varepsilon>0$ is small enough. To finish the proof, it suffices to show that for any $1>\varepsilon>0$,
		$$\limsup_{t\to+\infty}\|T^-_tu_{\varepsilon}-u_-\|_{\infty}\geqslant\Delta'.$$
		Assume by contradiction that there are $1>\eps_0>0$ and  $t'>0$ such that for each $t>t'$,
		\begin{align}\label{1111}
			\|u_--T^-_tu_{\varepsilon_0}\|_{\infty}<\Delta'.
		\end{align}
		Let $t_1:=\inf\{\tau:  \|u_--T^-_tu_{\varepsilon_0}\|_{\infty}<\Delta',\  \forall t\geqslant\tau\}$.
		Let $t_0$ be  a positive constant to be determined later. Using Lemma \ref{lem4}, there are $x_0\in M$ and a $(u_-,L,0)$-calibrated curve $\gamma_0:(-\infty,0]\to M$ satisfying $\gamma_0(0)=x_0$ and \eqref{eq4.1}. Let $\hat{t}:=t_1+t_0$
		
		Let $\hat{u}(s):=u_-(\gamma_0(s-\hat{t}))$ and $\hat{u}_{\varepsilon_0}(s):=T^-_s u_{\varepsilon_0}(\gamma_0(s-\hat{t}))$, $s\in[0,\hat{t}]$. For every $s$, $\Delta s$ satisfying $0\leqslant s<s+\Delta s\leqslant \hat{t}$, we have
		\begin{equation*}
			\begin{split}
				&\hat{u}(s+\Delta s)=u_-(\gamma_0(s+\Delta s-\hat{t}))\\
				=&u_-(\gamma_0(s-\hat{t}))+\int^{\Delta s}_0 L\big(\gamma_0(\tau+s-\hat{t}), \dot{\gamma}_0(\tau+s-\hat{t}),u_-(\gamma_0(\tau+s-\hat{t}))\big)d\tau\\
				=&\hat{u}(s)+\int^{s+\Delta s}_s L\big(\gamma_0(\tau-\hat{t}), \dot{\gamma}_0(\tau-\hat{t}),\hat{u}(\tau)\big)d\tau,
			\end{split}
		\end{equation*}
		and
		\begin{equation*}
			\begin{split}
				&\hat{u}_{\varepsilon_0}(s+\Delta s)=T^-_{s+\Delta s}u_{\varepsilon_0}(\gamma_0(s+\Delta s-\hat{t}))=T^-_{\Delta s}\circ T^-_su_{\varepsilon_0}(\gamma_0(s+\Delta s-\hat{t}))\\
				=&\inf_{\tilde{\gamma}(0)=\gamma_0(s+\Delta s-\hat{t})}\Big\{T^-_su_{\varepsilon_0}(\tilde{\gamma}(-\Delta s))+\int^0_{-\Delta s} L\big(\tilde{\gamma}(\tau), \dot{\tilde{\gamma}}(\tau),T^-_{\tau+\Delta s}\circ T^-_su_{\varepsilon_0}(\tilde{\gamma}(\tau))\big)d\tau\Big\}\\
				\leqslant&T^-_s u_{\varepsilon_0}(\gamma_0(s-\hat{t}))+\int^{\Delta s}_0 L\big(\gamma_0(\tau+s-\hat{t}), \dot{\gamma}_0(\tau+s-\hat{t}),T^-_{\tau+s} u_{\varepsilon_0}(\gamma_0(\tau+s-\hat{t}))\big)d\tau\\
				=&\hat{u}_{\varepsilon_0}(s)+\int^{s+\Delta s}_s L\big(\gamma_0(\tau-\hat{t}), \dot{\gamma}_0(\tau-\hat{t}),\hat{u}_{\varepsilon_0}(\tau)\big)d\tau.
			\end{split}
		\end{equation*}
		Let $\hat{w}(s):=\hat{u}(s)-\hat{u}_{\varepsilon_0}(s)$ From Proposition \ref{prop1} (1) we have $\hat{w}(s)\geqslant0$, $\forall s\in[0,\hat{t}]$ and $\hat{w}(0)=\varepsilon_0$. In view of the above arguments, we get that
		\begin{equation*}
			\begin{split}
				&\hat{w}(s+\Delta s)\\
				\geqslant&\hat{w}(s)+\int^{s+\Delta s}_s \Big(L\big(\gamma_0(\tau-\hat{t}), \dot{\gamma}_0(\tau-\hat{t}),\hat{u}(\tau)\big)-L\big(\gamma_0(\tau-\hat{t}), \dot{\gamma}_0(\tau-\hat{t}),\hat{u}_{\varepsilon_0}(\tau)\big)\Big)d\tau\\
				\geqslant&\hat{w}(s)+\int^{s+\Delta s}_s\hat{w}(\tau)\int^1_0 \frac{\partial L}{\partial u}\big(\gamma_0(\tau-\hat{t}), \dot{\gamma}_0(\tau-\hat{t}),\hat{u}(\tau)-\theta \hat{w}(\tau)\big)d\theta d\tau.
			\end{split}
		\end{equation*}
		Since $\gamma_0$ is Lipschitz, by Proposition \ref{prop1} (2) we know $\hat{w}(s)$ is Lipschitz in $[0,\hat{t}]$. So $\hat{w}(s)$ is differentiable almost everywhere in $[0,\hat{t}]$. Therefore by the  above inequality, we have
		\begin{equation*}
			\begin{split}
				\dot{\hat{w}}(s)\geqslant \hat{w}(s)\int^1_0 \frac{\partial L}{\partial u}\big(\gamma_0(s-\hat{t}), \dot{\gamma}_0(s-\hat{t}),\hat{u}(s)-\theta \hat{w}(s)\big)d\theta
				=\hat{w}(s)\cdot g(s),\quad \text{a.e.}\  s\in[0,\hat{t}],
			\end{split}
		\end{equation*}
		where $g(s):=\int^1_0\frac{\partial L}{\partial u}\big(\gamma_0(s-\hat{t}),\dot{\gamma}_0(s-\hat{t}),\hat{u}(s)-\theta\hat{w}(s)\big)d\theta$, $s\in[0,\hat{t}]$. Consider the Cauchy problem
		\begin{equation*}
			\left\{
			\begin{aligned}
				\dot{w}(s)&= w(s)\cdot g(s),\\
				w(0)&=\varepsilon_0.
			\end{aligned}
			\right.
		\end{equation*}
		We can deduce that
		\begin{equation}\label{eq20}
			\begin{split}
				&u_-(x_0)-T^-_{\hat{t}}u_{\varepsilon_0}(x_0)=u_-(\gamma_0(0))-T^-_{\hat{t}}u_{\varepsilon_0}(\gamma_0(0))\\
				&=\hat{w}(\hat{t})\geqslant\hat{w}(0)\cdot e^{\int^{\hat{t}}_0g(s)ds}=\varepsilon_0\cdot e^{\int^{\hat{t}}_0g(s)ds}\\
				&=\varepsilon_0\cdot e^{\int^{t_1}_0g(s)ds}\cdot e^{\int^{t_1+t_0}_{t_1}g(s)ds}.
			\end{split}
		\end{equation}
		
		Next we estimate $\int^{t_1}_0g(s)ds$ and   $\int^{t_1+t_0}_{t_1}g(s)ds$, respectively.
		First, we deal with $\int^{t_1}_0g(s)ds$.	Note that
		\[
		U(x,t,\theta):=u_-(x)-\theta(u_-(x)-T^-_tu_{\varepsilon_0}(x)),\quad (x,t,\theta)\in M\times[0,t_1]\times[0,1]
		\]
		is a continuous funtion.
		From Proposition \ref{houjia}, we know that $\|\dot{\gamma}\|\leqslant C_{u_-,l}$ for all $(u_-,L,0)$-calibrated curves. Thus,  $\{\big(\gamma_0(s-\hat{t}),\dot{\gamma}_0(s-\hat{t}),\hat{u}(s)-\theta\hat{w}(s)\big): s\in[0,t_1]\}$ is contained in the following compact set:
		$$\{(x,\dot{x},u): x\in M,\ |u|\leqslant\|U\|_\infty,\ \|\dot{x}\|\leqslant C_{u_-,l}\}.$$
		So, there is a positive constant $D$ depending  only on $u_-$ and $t_1$, such that
		$$|\frac{\partial L}{\partial u}\big(\gamma_0(s-\hat{t}),\dot{\gamma}_0(s-\hat{t}),\hat{u}(s)-\theta\hat{w}(s)\big)|\leqslant D,\quad \forall s\in[0,t_1],\ \forall \theta\in[0,1].$$
		Thus, we get that
		\begin{equation}\label{eq21}
			\begin{split}
				\int^{t_1}_0g(s)ds\geqslant&-\int^{t_1}_0\int^1_0\Big|\frac{\partial L}{\partial u}\big(\gamma_0(s-\hat{t}),\dot{\gamma}_0(s-\hat{t}),\hat{u}(s)-\theta\hat{w}(s)\big)\Big|d\theta ds\\
				\geqslant&-\int^{t_1}_0\int^1_0Dd\theta ds\geqslant-Dt_1.
			\end{split}
		\end{equation}
		
		Second, we estimate $\int^{t_1+t_0}_{t_1}g(s)ds$.
		From \eqref{1111} we have
		$$|\hat{w}(t)|\leqslant\|u_--T^-_tu_{\varepsilon}\|_{\infty}<\Delta'=\bar{\delta},\quad t>t_1.
		$$
		Using Lemma \ref{coro1} and Lemma \ref{lem4}, we get that
		\begin{equation}\label{eq22}
			\begin{split}
				\int^{t_1+t_0}_{t_1}g(s)ds=&\int^{t_1+t_0}_{t_1}\int^1_0\frac{\partial L}{\partial u}\big(\gamma_0(s-\hat{t}),\dot{\gamma}_0(s-\hat{t}),u_-(\gamma_0(s-\hat{t}))-\theta\hat{w}(s)\big)d\theta ds\\
				\geqslant&\int^{t_1+t_0}_{t_1}\Big(\frac{\partial L}{\partial u}\big(\gamma_0(s-\hat{t}),\dot{\gamma}_0(s-\hat{t}),u_-(\gamma_0(s-\hat{t}))\big)-c\Big)ds\\
				=&\int^0_{-t_0}\Big(\frac{\partial L}{\partial u}\big(\gamma_0(s),\dot{\gamma}_0(s),u_-(\gamma_0(s))\big)-c\Big)ds\\
				\geqslant&A't_0-ct_0=\frac{A'}{2}t_0.
			\end{split}
		\end{equation}
		
		Take $t_0:=\max\{1,\frac{2Dt_1}{A'}+\frac{2}{A'}\ln\frac{2\Delta'}{\varepsilon_0}\}$ depending only on $u_-,\ t_1,\ \Delta',\ \varepsilon_0$ and $A'$. Then according to the formulas \eqref{eq20}, \eqref{eq21} and \eqref{eq22}, we know
		\begin{equation*}
			\begin{split}
				u_-(x_0)-T^-_{\hat{t}}u_{\varepsilon_0}(x_0)\geqslant&\varepsilon_0\cdot e^{\int^{t_1}_0g(s)ds}\cdot e^{\int^{t_1+t_0}_{t_1}g(s)ds}\\
				\geqslant&\varepsilon_0\cdot e^{-Dt_1}\cdot e^{\frac{A'}{2}t_0}=2\Delta'>\Delta',
			\end{split}
		\end{equation*}
		which contradicts \eqref{1111}. The proof is now complete.
	\end{proof}

	\subsection{Example \ref{ex1}, revisited}	
	Recall Example \ref{ex1} mentioned in the Introduction.
	\begin{align}\label{161}
		w_t(x,t)+\|Dw(x,t)\|^2-\langle Dg(x),Dw(x,t)\rangle-f(x)(w(x,t)-g(x))=0,\quad x\in M,
	\end{align}
	where $f$, $g$ are smooth functions on $M$.
	Here,
	\[
	H(x,p,u)=\|p\|^2-\langle Dg(x),p\rangle-f(x)(u-g(x)),
	\]
	and
	\[
	L(x,\dot{x},u)=\frac{1}{4}\|\dot{x}+Dg(x)\|^2+f(x)(u-g(x)).
	\]
	It is direct to check $H$ satisfies conditions (H1)-(H3) and $w=g$ is a classical solution of \eqref{161}.

	Take an arbitrary $\mu\in\mathfrak{M}_{g}$. By the property of invariant measures, recurrent points are almost everywhere in $\supp\mu$. Take an arbitrary recurrent point $(x_0,p_0,u_0)\in \supp\mu$. Define $\big(x(t),p(t),u(t)\big):=\Phi^H_t(x_0,p_0,u_0)$, $t\in\mathbf{R}$. There is a sequence $\{t_n\}$ such that
	$$t_n\to+\infty, \quad \big(x(t_n),p(t_n),u(t_n)\big)\to(x_0,p_0,u_0),\quad  n\ \to+\infty.$$
	Since $\supp\mu$ is $\Phi^H_t$-invariant, the orbit $\big(x(t),p(t),u(t)\big)$ lies on  $\Lambda_g$. Thus, we have
	\begin{equation}\label{Example1}
		\big(x(t),p(t),u(t)\big)=\big(x(t),Dg(x(t)),g(x(t))\big),\quad \forall t\in \mathbf{R},
	\end{equation}
	and $(x_0,p_0,u_0)=\big(x_0,Dg(x_0),g(x_0)\big)$. Notice that
	\begin{equation*}
		\dot{x}(t)=\frac{\partial H}{\partial p}\big(x(t),p(t),u(t)\big)
		=2p(t)-Dg(x(t))
		=2Dg(x(t))-Dg(x(t))
		=Dg(x(t)).
	\end{equation*}
	Then we have
	\begin{equation*}
		\int^{t_n}_0|\dot{x}(s)|^2ds=\int^{t_n}_0\langle Dg(x(s)),\dot{x}(s)\rangle ds=g(x(t_n))-g(x_0)\longrightarrow0,\quad n\to+\infty.
	\end{equation*}
	It yields that
	$$\dot{x}(t)=Dg(x(t))\equiv0,\quad \forall t\in \mathbf{R}.$$
	Thus, $Dg(x_0)=0$. By \eqref{Example1}, we get that for any $t\in\mathbf{R}$, $\big(x(t),p(t),u(t)\big)=(x_0,0,g(x_0))$.

	Let $\mathcal{G}:=\{(x,0,g(x)):Dg(x)=0\}$. It is clear that $\mathcal{G}$ is nonempty.  According to the above arguments, we can deduce that
	$$\supp\mu\subset\mathcal{G},\quad \forall\mu\in\mathfrak{M}_g.$$
	It is direct to check that $\mathcal{G}$ consists of fixed points of $\Phi^H_t$. For any point $(\tilde{x},0,g(\tilde{x}))\in\mathcal{G}$, $\delta_{(\tilde{x},0,g(\tilde{x}))}$ is a $\Phi^H_t$-invariant probability measure supported in $\Lambda_{g}$.

	By direct computations, $\frac{\partial H}{\partial u}(x,p,u)=-f(x)$. Hence,
	\begin{itemize}
		\item If $f(x)<0$ on $\{x\in M: Dg(x)=0\}$, then
		$$\int_{T^*M\times \R}\frac{\partial H}{\partial u}d\mu=\int_{\supp\mu} -f(x)d\mu>0,\quad \forall\mu\in\mathfrak{M}_g.$$
		By Theorem \ref{th1}, we deduce that $g(x)$ is locally asymptotically stable.
		\item If $f(\bar{x})>0$ for some $\bar{x}$ in $\{x\in M: Dg(x)=0\}$, then
		$$\int_{T^*M\times \R}\frac{\partial H}{\partial u}d\delta_{(\bar{x},0,g(\bar{x}))}=\frac{\partial H}{\partial u}(\bar{x},0,g(\bar{x}))=-f(\bar{x})<0.$$
		From Theorem \ref{th2}, we get that $g(x)$ is unstable.
	\end{itemize}

	\section{Uniqueness of stationary solutions}
	\subsection{The proof of Theorem \ref{th3}}	
	\begin{proof}[Proof of Theorem \ref{th3}]
		Suppose there is a viscosity solution of \eqref{1-2}, denoted by $u^-_1$. We aim to show that $u^-_1$ is the unique viscosity solution and it is  globally asymptotically stable.

		From \eqref{A4} we know that \eqref{A1} holds true.
		By Theorem \ref{th1}, there is $\Delta>0$ such that, for any $\delta\in(0,\Delta)$, $\lim_{t\to+\infty}T^-_tu_{\delta}=u^-_1$ and $\lim_{t\to+\infty}T^-_tu^{\delta}=u^-_1$, where $u_{\delta}=u^-_1-\delta$ and $u^{\delta}=u^-_1+\delta$. It is clear that $u^-_1\geqslant T^-_tu_{\delta}\geqslant u_{\delta}$, $\forall t\geqslant\tau$ for some $\tau>0$.  Corollary \ref{prop10}, we have
		\begin{equation}\label{eq5.2}
			T^+_tu_{\delta}\leqslant u_{\delta}<u^-_1<u^{\delta},\ \forall t\geqslant\tau.
		\end{equation}
		
		\item [\bf{Step 1:}] We show that the function $(x,t)\to T^+_tu_{\delta}(x)$ is unbounded from below on $M\times[0,+\infty)$.
		
		Assume by contradiction that the function $(x,t)\to T^+_tu_{\delta}(x)$ is bounded from below on $M\times[0,+\infty)$. In view of \eqref{eq5.2}, the function $(x,t)\to T^+_tu_{\delta}(x)$ is bounded from above on $M\times(0,+\infty)$. Hence, we can define $u^+_2(x):=\limsup_{t\to+\infty}T^+_tu_{\delta}(x)$, and  Proposition \ref{prop5.1} (1)  $u^+_2\in \mathcal{S}^+$. In view of \eqref{eq5.2}, we have
		\begin{equation}\label{eq5.3}
			u^+_2\leqslant u_{\delta}<u^-_1.
		\end{equation}
		Let $u^-_2(x):=\lim_{t\to+\infty}T^-_tu^+_2(x)$.  Then by Proposition \ref{prop6} (2), $u^-_2\in S^-$. Using Corollary \ref{uplus} and \eqref{eq5.3}, we get that
		\begin{equation}\label{eq5.4}
			u^+_2\leqslant u^-_2\leqslant u^-_1.
		\end{equation}
		There are two possibilities:
		
		Case (i). $u^-_2=u^-_1$. It means $\lim_{t\to+\infty}T^-_tu^+_2=u^-_1$ uniformly on $M$. From \eqref{eq5.3}, there is $t_0>0$ such that $u^+_2<T^-_{t_0}u^+_2$. And it yields that  $u^+_2=T^+_{t_0}u^+_2<T^+_{t_0}\circ T^-_{t_0}u^+_2\leqslant u^+_2$,  a contradiction.
		
		Case (ii). $u^-_2\neq u^-_1$. According to Theorem \ref{th1}, we can deduce that  $\|u^-_2-u^-_1\|_\infty>\Delta$ and there is  $\Delta_1\in(0,\delta)$ such that
		\begin{equation}\label{eq5.5}
			\lim_{t\to+\infty}T^-_t\varphi=u^-_2, \quad \forall\varphi\in C(M,\R)\ \text{with }\ \|\varphi-u^-_2\|_\infty\leqslant\Delta_1.
		\end{equation}
		Let $\varphi=\max\{u^+_2+\Delta_1,u^-_2\}$. Through \eqref{eq5.3} and \eqref{eq5.4},
		\begin{equation}\label{eq5.6}
			u^+_2<\varphi\leqslant u^-_2+\Delta_1,
		\end{equation}
		and
		\begin{equation}\label{eq28}
			\varphi\leqslant\max\{u_{\delta}+\delta,u^-_2\}\leqslant u^-_1.
		\end{equation}
		In the light of the definition of $u^-_2$ and \eqref{eq5.5}, we have
		$$u^-_2=\lim_{t\to+\infty}T^-_tu^+_2\leqslant\lim_{t\to+\infty}T^-_t\varphi\leqslant\lim_{t\to+\infty}T^-_t(u^-_2+\Delta_1)=u^-_2.$$
		It implies that
		\begin{align}\label{1212}
			\lim_{t\to+\infty}T^-_t\varphi=u^-_2.
		\end{align}
		On the other hand, from Proposition \ref{prop5.1},  $u^+_2(x):=\lim_{t\to+\infty}\sup_{s\geqslant t}T^+_su_{\delta}(x)$ uniformly on $x\in M$. By \eqref{eq5.6}, there is  $t_1>0$ such that
		$$T^+_{t_1}u_{\delta}\leqslant\varphi.$$
		Therefore, by the Proposition \ref{prop2} (1) and \eqref{eq28}, we get that
		$$T^-_tu_{\delta}\leqslant T^-_t\circ T^-_{t_1}\circ T^+_{t_1}u_{\delta}\leqslant T^-_{t+t_1}\varphi\leqslant T^-_{t+t_1}u^-_1=u^-_1.$$
		Letting $t\to+\infty$, we get that  $\lim_{t\to+\infty}T^-_t\varphi=u^-_1$ which  contradicts \eqref{1212}.
		
		\item [\bf{Step 2:}] We prove that  $\lim_{t\to+\infty}T^-_t\psi=u^-_1$ for all function $\psi\in C(M,\R)$ with $\psi\leqslant u^-_1$.
		
		Let $Q:=\inf_{x\in M}\psi(x)$. According to Step 1 and Proposition \ref{prop5.2} (i), there is  $s\in[0,+\infty)$ such that $T^+_su_{\delta}\leqslant Q\leqslant\psi\leqslant u^-_1$. Therefore, we have $$T^-_tu_{\delta}\leqslant T^-_t\circ T^-_s\circ T^+_su_{\delta}\leqslant T^-_{t+s}\psi\leqslant T^-_{t+s}u^-_1=u^-_1.$$
		Let $t$ tend to $+\infty$. Then $\lim_{t\to+\infty}T^-_t\psi=\lim_{t\to+\infty}T^-_tu_{\delta}=u^-_1$ uniformly on $x\in M$.
		
		\item [\bf{Step 3:}] We show that $u^-_1$ is the unique viscosity solution of equation \eqref{1-2}.
		
		For any given $u^-_3\in \mathcal{S}^-$, let $\bar{u}(x):=\min\{u^-_1(x),u^-_3(x)\}$, $x\in M$. Since $u^-_1$, $u^-_3$ are viscosity solutions, then $\bar{u}$ is also a viscosity solution.  By Step 2, we deduce that $u^-_1=\bar{u}\leqslant u^-_3$. By similar arguments, one can get that $u^-_3=\bar{u}\leqslant u^-_1$.
		
		\item [\bf{Step 4:}] We show $\lim_{t\to+\infty}T^-_t\psi=u^-_1$ for all function $\psi\in C(M,\R)$ with $\psi\geqslant u^-_1$.
		
		For any $\psi\in C(M,\R)$ satisfying $\psi\geqslant u^-_1$. There are two possibilities:
		
		Case (i). The function $(x,t)\to T^-_t\psi(x)$ is bounded from above on $M\times [0,+\infty)$. In this case, using the monotonicity of $T^-_t$, we get that $u^-_1=T^-_tu^-_1\leqslant T^-_t\psi$, which means the function $(x,t)\to T^-_t\psi(x)$ is bounded from below on $M\times [0,+\infty)$. Thus by Proposition \ref{prop5.1} (2), the uniform limit $\lim_{t\to+\infty}\inf_{s\geqslant t}T^-_s\psi(x)$ exists, denoted by $\psi_{\infty}$, and  $\psi_{\infty}\in\mathcal{S}^-$. From Step 3,  $u^-_1=\psi_{\infty}$. So there is a constant $\bar{t}>0$ such that
		$$u^-_1\leqslant T^-_{\bar{t}}\psi\leqslant u^{\delta}.$$
		From the above inequalities, we can deduce that  $\lim_{t\to+\infty}T^-_t\psi=u^-_1$.
		
		Case (ii). The funtion $(x,t)\to T^-_t\psi(x)$ is unbounded from above on $M\times [0,+\infty)$. Let $\phi:=\psi+\delta\geqslant u^-_1+\delta=u^{\delta}> u^-_1$. Using the monotonicity of $T^-_t$, the function $(x,t)\to T^-_t\phi(x)$ is unbounded above on $M\times [0,+\infty)$. Let $N:=\sup_{x\in M}\phi(x)+1$. Then by Proposition \ref{prop5.2} (ii) there is  $\sigma>0$ such that $T^-_{\sigma}\phi(x)\geqslant N,\ \forall x\in M$.  Thus $T^-_{\sigma}\phi>\phi$. From Corollary \ref{prop10}, we have $T^+_{\sigma}\phi\leqslant\phi$. Note that the function $(x,t)\to T^+_t\phi(x)$ is continuous on $M\times
		[0,\sigma]$. Thus there is a constant $\bar{N}>0$ such that $T^+_s\phi(x)<\bar{N}$,  $\forall x\in M$, $\forall s\in[0,\sigma]$. For any $t\in [0,+\infty)$,  $t=n\sigma+s$ where $n\in\mathbf{N}$ and $s\in[0,\sigma)$. Hence, we have
		$$T^+_t\phi=T^+_{(n-1)\sigma+s}\circ T^+_{\sigma}\phi\leqslant T^+_{(n-1)\sigma+s}\phi\leqslant\cdots\leqslant T^+_s\phi\leqslant\bar{N}.$$
		On the other hand, let $u^+_1:=\lim_{t\to+\infty}T^+_tu^-_1(x)$. By Proposition \ref{prop6} (1), $u^+_1$ is well defined and belongs to $\mathcal{S}^+$. From Corollary \ref{uplus},
		there holds $u^-_1\geqslant u^+_1$. So we have
		$$T^+_t\phi\geqslant T^+_tu^-_1\geqslant u^+_1,\quad \forall t\in [0,+\infty).$$
		Thus the function $(x,t)\to T^+_t\phi(x)$ is bounded on $M\times[0,+\infty)$. The uniform limit $\phi^+_{\infty}:=\lim_{t\to+\infty}\sup_{s\geqslant t}T^+_s\phi(x)\in \mathcal{S}^+$ by Proposition \ref{prop5.1}.
		There are two situations here:
		
		Case (a). If $T^+_{\kappa}\phi\leqslant u^{\delta}$ for some $\kappa>0$, then $\phi\leqslant T^-_{\kappa}u^{\delta}$ by  Corollary \ref{prop10}. Thus
		$$\lim_{t\to+\infty}T^-_t\phi\leqslant \lim_{t\to+\infty}T^-_{t+\kappa}u^{\delta}=u^-_1<+\infty,$$
		which conflict $\lim_{t\to+\infty}T^-_t\phi=+\infty$.
		
		Case (b). If, for any $t>0$, there is a $x_t\in M$ such that $T^+_t\phi(x_t)>u^{\delta}(x_t)$. Then we can find sequences $\{t_n\}\subset \R$ with $t_n\to+\infty$ and $\{x_n\}\subseteq M$ such that
		$$T^+_{t_n}\phi(x_n)>u^{\delta}(x_n),\quad \lim_{n\to+\infty}x_n=x_0\in M,\quad  \lim_{n\to+\infty}\sup_{s\geqslant t_n}T^+_{s}\phi=\phi^+_{\infty}\quad \text{uniformly on}\ M.$$
		It means $\phi^+_{\infty}(x_0)\geqslant u^{\delta}(x_0)$. Let $\phi^-_{\infty}:=\lim_{t\to+\infty}T^-_t\phi^+_{\infty}$. Then $\phi^-_{\infty}\in \mathcal{S}^-$ by Proposition \ref{prop6} (2). From Step 3, we have $\phi^-_{\infty}=u^-_1$. However, $\phi^-_{\infty}\geqslant\phi^+_{\infty}$ by Corollary \ref{uplus}.
		Therefore, $\phi^-_{\infty}(x_0)\geqslant\phi^+_{\infty}(x_0)\geqslant u^{\delta}(x_0)>u^-_1(x_0)$, which yields a contradiction.
	\end{proof}

	\subsection{The proof of Corollary \ref{addcoro1}}
	
	\begin{proof}[Proof of Corollary \ref{addcoro1}]
		Take an arbitrary $\mu_0\in\mathfrak{M}$. In view of the property of invariant measures, recurrent points are almost everywhere in $\supp\mu_0$. Take an arbitrary recurrent point $(x_0,p_0,u_0)\in \supp\mu_0$. Let $\big(x(t),p(t),u(t)\big):=\Phi^H_t(x_0,p_0,u_0)$, $t\in\mathbf{R}$. So there is a sequence $\{t_n\}$ such that
		$$t_n\to+\infty, \quad \big(x(t_n),p(t_n),u(t_n)\big)\to(x_0,p_0,u_0),\quad  n\ \to+\infty.$$
		From Corollary \ref{mather}, there exist $u_-\in\mathcal{S}^-$ and $\mu\in\mathfrak{M}_{u_-}$ such that $(x_0,p_0,u_0)\in \supp\mu$.
		Thus $\supp\mu\subset\Lambda_{u_-}$. Since $\supp\mu$ is  $\Phi^H_t$-invariant, the orbit $\{\big(x(t),p(t),u(t)\big)\}\subset \supp\mu\subset\Lambda_{u_-}$. Thus,  we have $H\big(x(t),p(t),u(t)\big)=0$, $\forall t\in \R$. As a consequence, we get that
		\begin{equation}\label{addeq5.8}
			\dot{u}(t)=\frac{\partial H}{\partial p}\big(x(t),p(t),u(t)\big)\cdot p(t)-H\big(x(t),p(t),u(t)\big)
			=\frac{\partial H}{\partial p}\big(x(t),p(t),u(t)\big)\cdot p(t).	
		\end{equation}
		For any $(x,p,u)\in T^*M\times\R$, define a  function $f(\lambda):=H(x,\lambda p,u),\ \lambda\in\R$. Since the Hamiltonian $H$ is strictly convex and reversible in $p$, it is easy to check that $f'(1)=\frac{\partial H}{\partial p}(x,p,u)\cdot p\geqslant0$ for all $p\in T^*_xM$ where the equality holds if and only if $p=0$.
		Then from \eqref{addeq5.8}, we get that
		$$0\leqslant\int^{t_n}_0 \frac{\partial H}{\partial p}\big(x(t),p(t),u(t)\big)\cdot p(t)ds=\int^{t_n}_0\dot{u}(s)ds=u(t_n)-u_0\longrightarrow0,\quad n\to+\infty.$$
		It implies that
		$$p(t)\equiv0,\quad \dot{u}(t)\equiv0,\quad \forall t\in \R.$$
		And we can deduce that
		$$\dot{x}(t)=\frac{\partial H}{\partial p}\big(x(t),p(t),u(t)\big)\equiv0,\quad \forall t\in \R.$$
		Therefore, $\big(x(t),p(t),u(t)\big)\equiv(x_0,0,u_0)$ is a fixed point of $\Phi^H_t$. Moreover, $(x_0,u_0)$ satisfies
		\begin{align*}
			\left\{
			\begin{array}{l}
				H(x_0,0,u_0)=0,\\
				\dot{p}(t)=-\frac{\partial H}{\partial x}(x_0,0,u_0)=0.
			\end{array}
			\right.
		\end{align*}
		Thus $(x_0,u_0)$ belongs to $\mathcal{B}$ in condition. Due to the above arguments, we can deduce that
		$$\supp\mu_0\subset\{(x,0,u): (x,u)\in\mathcal{B}\},\quad  \forall\mu_0\in\mathfrak{M}.$$
		If $\frac{\partial H}{\partial u}(x,0,u)>0$ on $\mathcal{B}$, then we have
		$$\int_{T^*M\times\R}\frac{\partial H}{\partial u}d\mu_0=\int_{\supp\mu_0}\frac{\partial H}{\partial u}d\mu_0>0,\quad\forall\mu_0\in\mathfrak{M}.$$
		According to Theorem \ref{th3},  equation \eqref{1-2} has at most one viscosity solution.
	\end{proof}

	\subsection{Example \ref{ex2}, revisited}	
	Consider the following equation
	\begin{align}\label{1616}
		\frac{1}{2}\|Du(x)\|^2+\frac{1}{2}u(x)+\sin u(x)=0,\quad x\in M.	
	\end{align}
	Here $H(x,p,u)=\frac{1}{2}\|p\|^2+\frac{1}{2}u+\sin u$.
	It is easy to check that $H(x,p,u)$ is smooth and satisfies (H1)-(H3).
	Notice that $u=0$ is a viscosity solution of equation \eqref{1616}. In order to apply Corollary \ref{addcoro1} to equation \eqref{1616}, we need to show \eqref{1-206} holds.

	Note that any point $(x,u)$ in $\mathcal{B}$ satisfies
	\begin{align*}
		\left\{
		\begin{array}{l}
			H(x,0,u)=\frac{1}{2}u+\sin u=0,\\
			\frac{\partial H}{\partial x}(x,0,u)\equiv0.
		\end{array}
		\right.
	\end{align*}
	By direct computation, $\mathcal{B}=\{(x,0)|\ x\in M\}$. So we get that
	$$\frac{\partial H}{\partial u}(x,0,u)=\frac{1}{2}+\cos u=\frac{3}{2}>0,\quad \forall(x,u)\in\mathcal{B}.$$
	Therefore, by Corollary \ref{addcoro1}, equation \eqref{1616} has a unique viscosity solution.

	\subsection{The proof of Corollary \ref{addcoro2}}
	
	\begin{proof}[Proof of Corollary \ref{addcoro2}]
		Take an arbitrary $\mu_0\in\mathfrak{M}$. From Corollary \ref{mather1}, we know $\supp\mu_0\subset\mathcal{E}$.	Take an arbitrary $(x_0,p_0,u_0)\in \supp\mu_0$. Let $\big(x(t),p(t),u(t)\big):=\Phi^H_t(x_0,p_0,u_0)$, $t\in\mathbf{R}$.
		Since $\supp\mu_0$ is $\Phi^H_t$-invariant, the orbit $\{\big(x(t),p(t),u(t)\big)\}\subset \supp\mu_0\subset\mathcal{E}$. Recalling condition \eqref{add1.10}, we have the following two different cases.
		
		Case (i). $\frac{\partial H}{\partial u}(x_0,p_0,u_0)\neq0$. Since $\frac{\partial H}{\partial u}(x,p,u)\Big|_\mathcal{E}\geqslant0$,  we have $\frac{\partial H}{\partial u}(x_0,p_0,u_0)>0$. So, we can choose a neighborhood $O_0$ of $(x_0,p_0,u_0)$ such that
		$$\mu_0(O_0)=\varepsilon_0>0\quad \text{and} \quad \frac{\partial H}{\partial u}\Big|_{O_0}>\delta_0>0$$
		for some $\varepsilon_0>0$ and $\delta_0>0$.
		Hence, we have
		$$\int_{T^*M\times\R}\frac{\partial H}{\partial u}d\mu_0=\int_{\supp\mu_0}\frac{\partial H}{\partial u}d\mu_0\geqslant\int_{O_0\cap\supp\mu_0}\frac{\partial H}{\partial u}d\mu_0\geqslant\varepsilon_0\delta_0>0.$$
		
		Case (ii). $\frac{\partial H}{\partial u}(x_0,p_0,u_0)=0$. In view of \eqref{add1.10}, one can deduce that  $\frac{\partial H}{\partial u}\big(x(t_1),p(t_1),u(t_1)\big)>0$ for some $t_1>0$. We can choose a neighborhood $O_1$ of $\big(x(t_1),p(t_1),u(t_1)\big)$ such that
		$$\mu_0(O_1)=\varepsilon_1>0\quad \text{and} \quad \frac{\partial H}{\partial u}\Big|_{O_1}>\delta_1>0.$$
		for some $\varepsilon_1>0$ and $\delta_1>0$.
		As a consequence, we have
		$$\int_{T^*M\times\R}\frac{\partial H}{\partial u}d\mu_0=\int_{\supp\mu_0}\frac{\partial H}{\partial u}d\mu_0\geqslant\int_{O_1\cap\supp\mu_0}\frac{\partial H}{\partial u}d\mu_0\geqslant\varepsilon_1\delta_1>0.$$
		
		So far, we have proved that \eqref{A4} holds true and thus by Theorem \ref{th3} equation \eqref{1-2} has at most one viscosity solution.
	\end{proof}

	\subsection{Example \ref{ex3}, revisited}
	Let us focus on the following equation
	\begin{align}\label{eq5.10}
		w_t(x,t)+\frac{1}{2}w^2_x(x,t)-a\cdot w_x(x,t)+(\sin x+b)\cdot w(x,t)=0,\quad x\in \mathbf{S},
	\end{align}
	where $a$, $b\in\R$. Here $H(x,p,u)=\frac{1}{2}p^2-a\cdot p+(\sin x+b)\cdot u$. We can easily check that $H(x,p,u)$ is smooth and satisfies (H1)-(H3). Notice that the constant function $w=0$ is a solution of equation \eqref{eq5.10}.
	
	Take an arbitrary $\mu\in\mathfrak{M}_{w}$ and take an arbitrary point $(x_0,p_0,u_0)\in \supp\mu$. Denote $\big(x(t), p(t), u(t)\big):=\Phi^H_t(x_0,p_0,u_0)$, $t\in\mathbf{R}$. Since $\supp\mu$ is $\Phi^H_t$-invariant, orbit $\big(x(t),p(t),u(t)\big)$ lies on  $\Lambda_w$. Thus, we have
	\begin{equation*}
		\big(x(t),p(t),u(t)\big)=(x(t),0,0),\quad \forall t\in \mathbf{R},
	\end{equation*}
	and $(x_0,p_0,u_0)=(x_0,0,0)$. Notice that
	\begin{equation*}
		\dot{x}(t)=\frac{\partial H}{\partial p}\big(x(t),p(t),u(t)\big)
		=p(t)-a
		=-a.
	\end{equation*}
	There are two situations.
	\begin{itemize}
		\item If $a=0$, according to above arguments, we can deduce that all the orbits in $\supp\mu$ are fixed points of the flow.
		
		\subitem{(1)} When $b<1$, it is easy to check that $\delta_{(\frac{3}{2}\pi,0,0)}$ is a $\Phi^H_t$-invariant probability measure supported in $\Lambda_{w}$. Thus,  $\delta_{(\frac{3}{2}\pi,0,0)}\in\mathfrak{M}_{w}$. Since
		$$\int_{T^*\mathbf{S}^1\times\R}\frac{\partial H}{\partial u}d\delta_{(\frac{3}{2}\pi,0,0)}=\sin(\frac{3}{2}\pi)+b=b-1<0,$$
		from Theorem \ref{th2}, we get that $w=0$ is unstable.
		
		\subitem{(2)} When $b>1$, one can check that
		$$\frac{\partial H}{\partial u}(x,p,u)=\sin x+b>0,\quad \forall (x,p,u)\in T^*\mathbf{S}\times\R.$$
		This implies condition \eqref{A4} holds true. By Theorem \ref{th3}, $w=0$ is the unique viscosity solution and is globally asymptotically stable.
		
		\item If $a\neq0$, we can deduce that any orbit in $\supp\mu$ satisfies $\big(x(t),p(t),u(t)\big)=(x(t),0,0)$ where $x(t):=x_0-at\mod2\pi$ for some $x_0\in\mathbf{S}$. As a consequence, $\mathfrak{M}_{w}$ has only one element $\mu$
		$$\mu(f)=\int_{T^*\mathbf{S}}fd\mu=\frac{1}{2\pi}\int^{2\pi}_0f(x,0,0)dx,\quad \forall f\in C(T^*\mathbf{S}\times\R,\R).$$
		
		\subitem{(1)} When $b<0$, we have
		$$\int_{T^*\mathbf{S}\times\R}\frac{\partial H}{\partial u}d\mu=\frac{1}{2\pi}\int^{2\pi}_0(\sin x+b) dx=b<0.$$
		From Theorem \ref{th2}, we know that $w=0$ is unstable.
		
		\subitem{(2)} When $0<b<1$, we get  that
		$$\int_{T^*\mathbf{S}\times\R}\frac{\partial H}{\partial u}d\mu=\frac{1}{2\pi}\int^{2\pi}_0(\sin x+b) dx=b>0.$$
		From Theorem \ref{th1}, we know that $w=0$ is locally asymptotically stable.
		
		\subitem{(3)} When $b\geqslant1$, for any $(\bar{x},\bar{p},\bar{u})\in\mathcal{E}$,
		$$\frac{\partial H}{\partial u}(\bar{x},\bar{p},\bar{u})=\sin\bar{x}+b\geqslant\sin\bar{x}+1\geqslant0.$$
		If $\frac{\partial H}{\partial u}(\bar{x},\bar{p},\bar{u})=0$, then $\sin\bar{x}=-b=-1$. Since $H(\bar{x},\bar{p},\bar{u})=0$, we can deduce that $(\bar{p}-a)^2=a^2\neq0$. So we have
		$$\mathfrak{L}^2_H\frac{\partial H}{\partial u}(\bar{x},\bar{p},\bar{u})=(-\bar{u}\cos\bar{x}-\bar{p}\sin\bar{x}-\bar{p})\cos\bar{x}-(\bar{p}-a)^2\sin\bar{x}=a^2\neq0.$$
		This implies condition \eqref{add1.10} holds true. According to Corollary \ref{addcoro2}, $w=0$ is the unique viscosity solution and is globally asymptotically stable.
	\end{itemize}

	\bigskip
	
	\bigskip
	
	\noindent {\bf Funding}
	Kaizhi Wang is supported by Key R\&D Program of China  (Grant No. 2022YFA1005900), National Natural Science Foundation of China (Grant Nos. 12171315, 11931016) and by Natural Science Foundation of Shanghai (Grant No. 22ZR1433100). Jun Yan is supported by National Natural Science Foundation of China (Grant Nos. 12171096,  12231010).

	%
	%
	%


	\bibliographystyle{plain}

\end{document}